\documentclass[11pt]{amsart}
  
\usepackage{enumerate}
\usepackage[all]{xy}
\usepackage[hyperindex=true,plainpages=false,colorlinks=false,pdfpagelabels]{hyperref}
\usepackage{verbatim}
\usepackage{pdfpages}
\usepackage{graphicx, color}
\usepackage{subfigure}
\theoremstyle{plain}
\newtheorem{The}{Theorem}[section]
\newtheorem*{The*}{Theorem}

\newtheorem{Lem}{Lemma}[section]
\newtheorem{Cor}{Corollary}[section]
\newtheorem*{Cor*}{Corollary}

\theoremstyle{definition}
\newtheorem{Rem}{Remark}[section]
\newtheorem{Exa}{Example}[section]
\newtheorem*{Rem*}{Remark}
\numberwithin{equation}{section}

\DeclareMathOperator{\End}{End}

\DeclareMathOperator{\SL}{SL}

\DeclareMathOperator{\SU}{SU}

\DeclareMathOperator{\Id}{Id}

\DeclareMathOperator{\res}{res}

\renewcommand{\Im}{\operatorname{Im}}
\renewcommand{\Re}{\operatorname{Re}}

\newcommand{\dvector}[1]{{\left(\begin{matrix}#1\end{matrix}\right)}}

\DeclareMathOperator{\dbar}{\bar\partial}
\DeclareMathOperator{\del}{\partial}

\newcommand{\R}{\mathbb{R}}
\newcommand{\Q}{\mathbb{Q}}
\newcommand{\C}{\mathbb{C}}

\newcommand{\Z}{\mathbb{Z}}

\newcommand{\CP}{\mathbb{CP}}

\newcommand{\bbS}{\mathbb{S}}

\begin{document}

\title[Navigating the space of symmetric CMC surfaces]{Navigating the space of symmetric CMC surfaces}
\author{Lynn Heller}
\author{Sebastian Heller}
\author{Nicholas Schmitt}
\address{Lynn Heller \\
  Institut f\"ur Differentialgeometrie\\
 Leibniz Universit{\"a}t Hannover\\ 
Welfengarten 1\\
30167 Hannover\\ Germany\\
 }
 \email{lynn.heller@math.uni-hannover.de}

\address{Sebastian Heller\\Department of Mathematics \\
University of Hamburg \\
Bundesstr. 55\\
20146 Hamburg\\ Germany\\
 }
 \email{seb.heller@gmail.com}
 
\address{Nicholas Schmitt \\
  Institut f\"ur Differentialgeometrie\\
 Leibniz Universit{\"a}t Hannover\\ 
Welfengarten 1\\
30167 Hannover\\ Germany\\
 }
\email{nschmitt@mathematik.uni-tuebingen.de}


\begin{abstract} 

In this paper we introduce a flow on the spectral data for symmetric CMC surfaces in the $3$-sphere. The flow is designed in such a way that it
changes the topology but fixes the intrinsic (metric) and certain extrinsic (periods) closing conditions of the CMC surfaces. By construction the flow yield closed (possibly branched) CMC surfaces at rational times and immersed higher genus CMC surfaces at integer times. We prove the short time existence of this flow near the spectral data of (certain classes of) CMC tori and  obtain thereby the existence of new families of closed (possibly branched) connected CMC surfaces of higher genus. Moreover, we prove that flowing the spectral data for the Clifford torus is equivalent to the flow of Plateau solutions by varying
the angle of the fundamental piece in Lawson's construction for the minimal surfaces $\xi_{g,1}.$ 

\end{abstract}

\maketitle

\setcounter{tocdepth}{2}
\tableofcontents

\typeout{== intro ============================================}

\section*{Introduction}
\label{sec:intro}
 \begin{figure}
\centering
\includegraphics[width=0.175\textwidth]{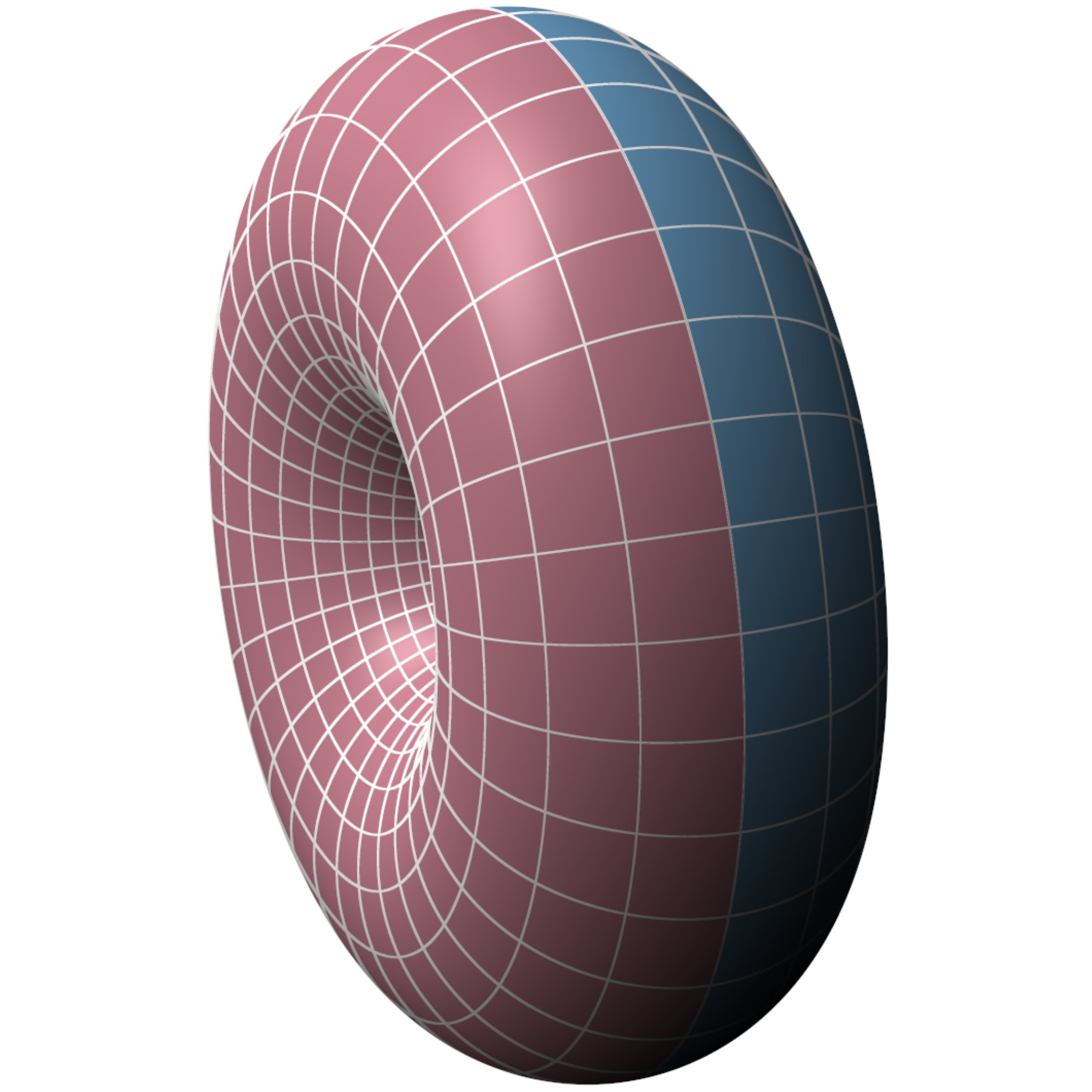}
\includegraphics[width=0.175\textwidth]{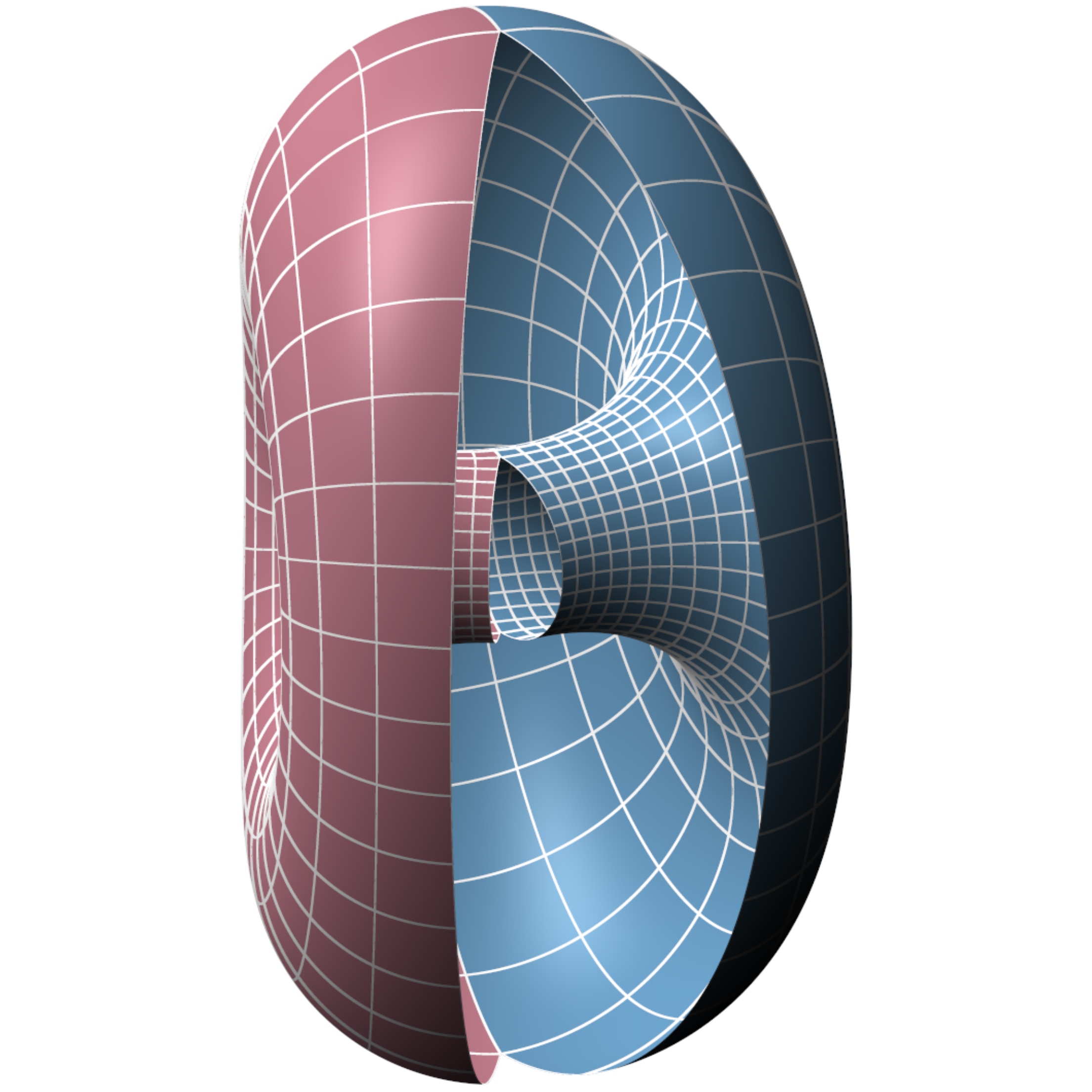}
\includegraphics[width=0.175\textwidth]{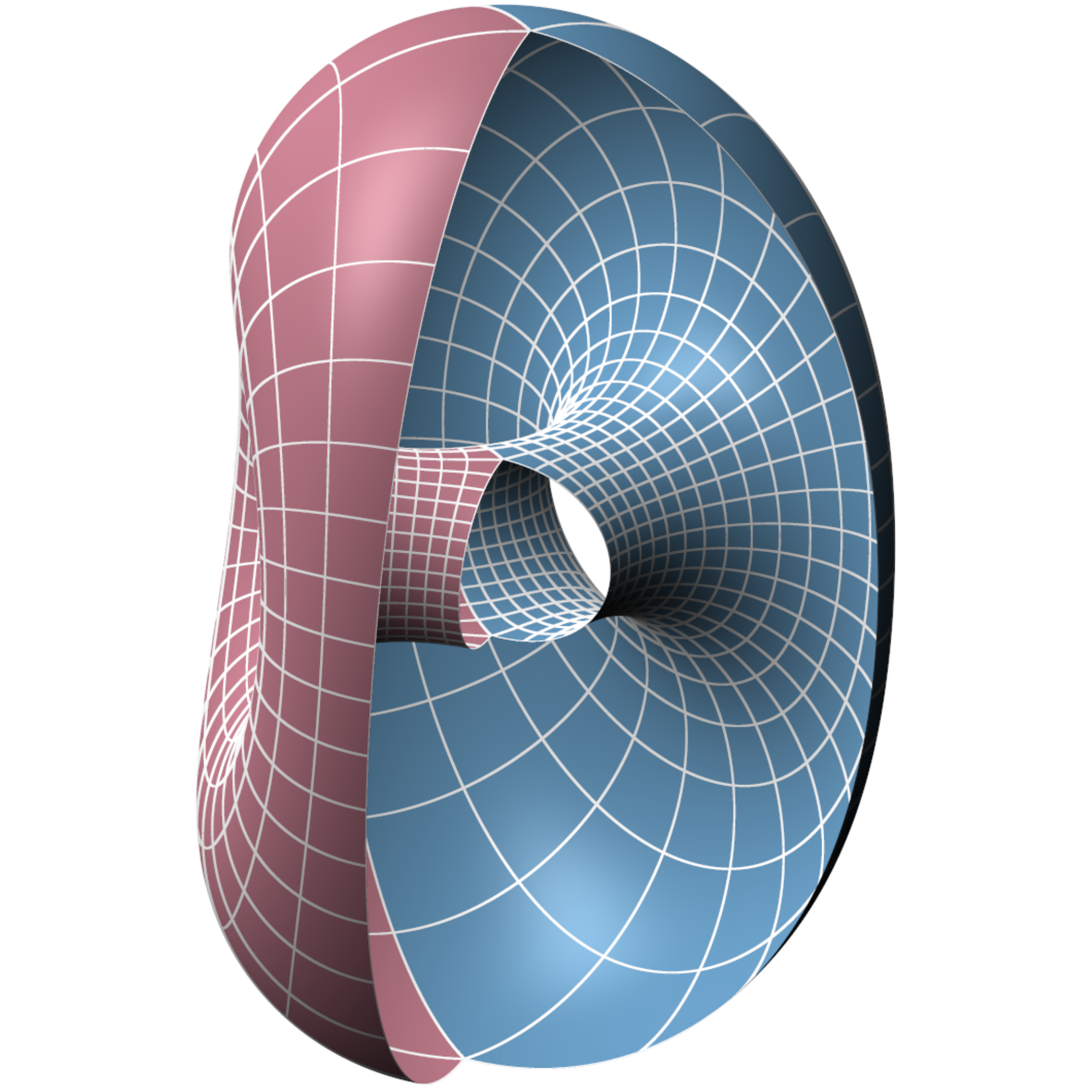}
\includegraphics[width=0.175\textwidth]{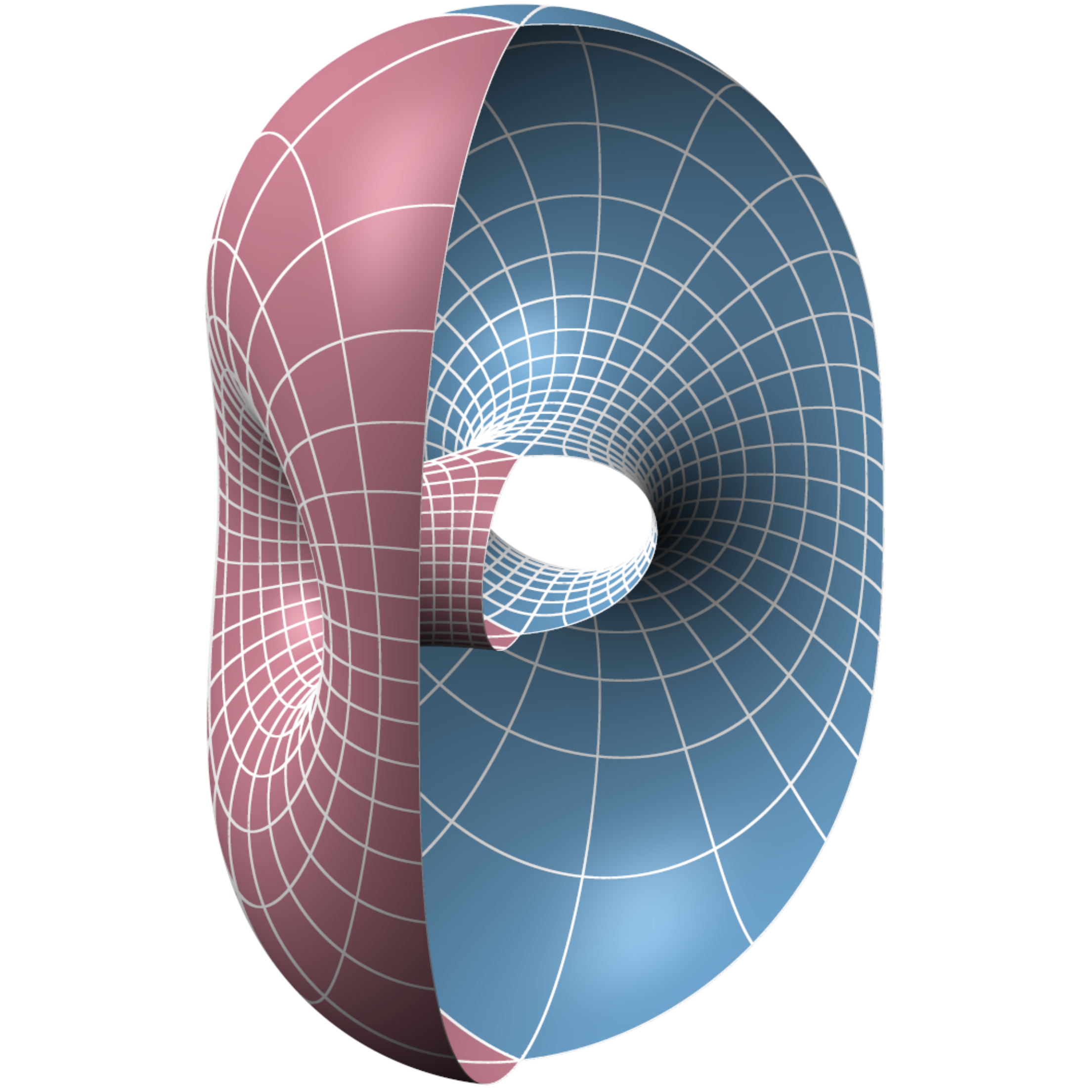}
\includegraphics[width=0.175\textwidth]{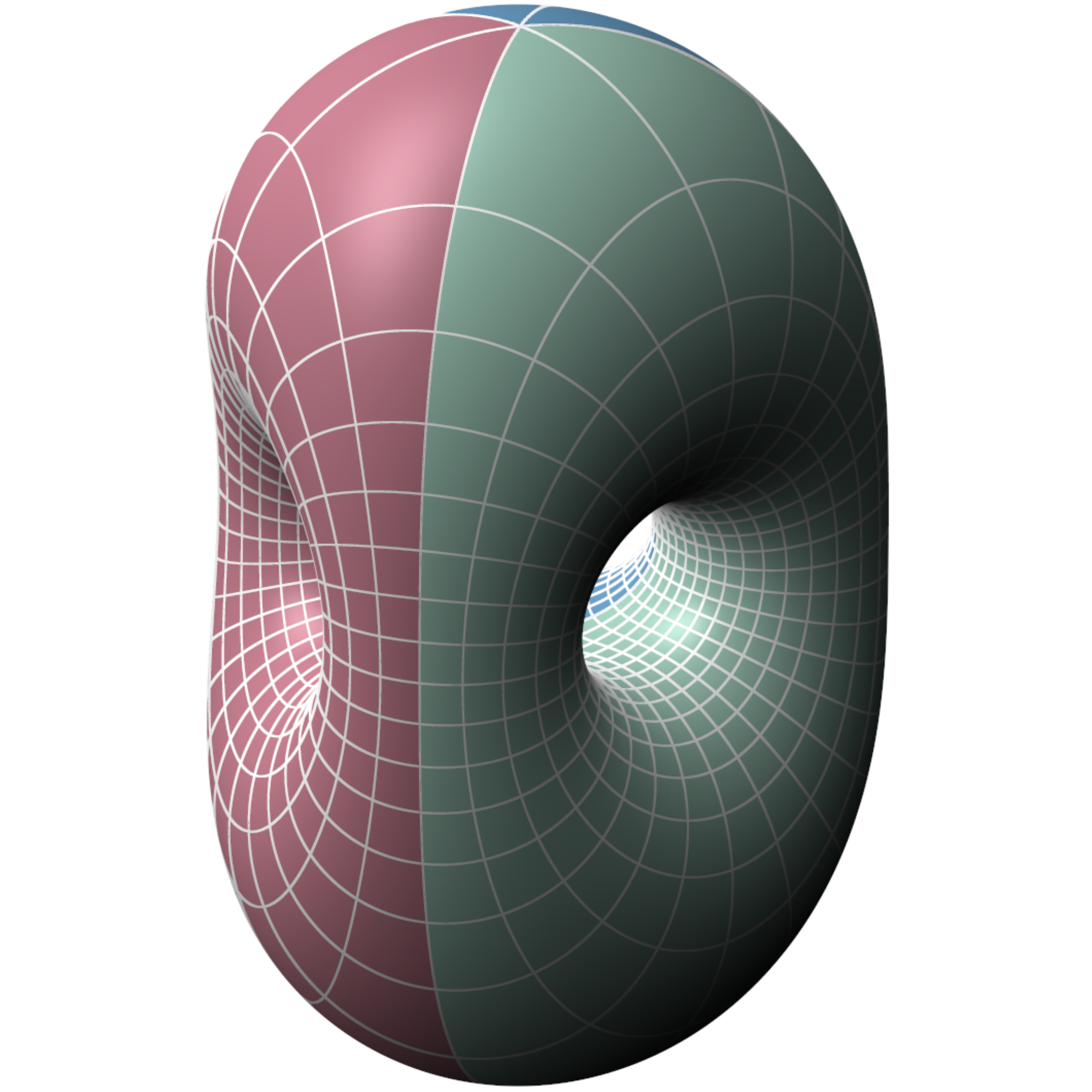}
\vspace{0.25cm}

\includegraphics[width=0.1\textwidth]{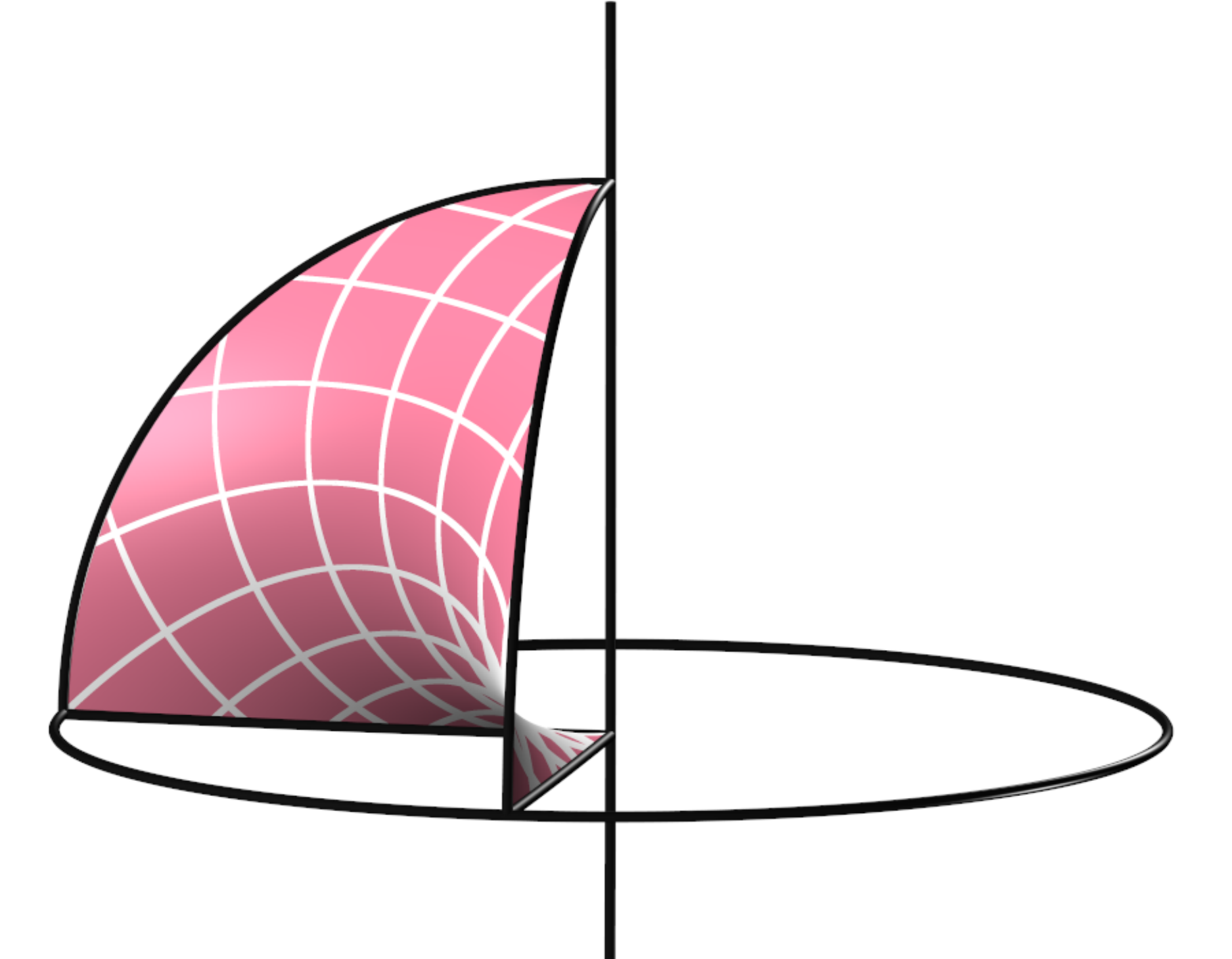}
\hspace{0.75cm}
\includegraphics[width=0.1\textwidth]{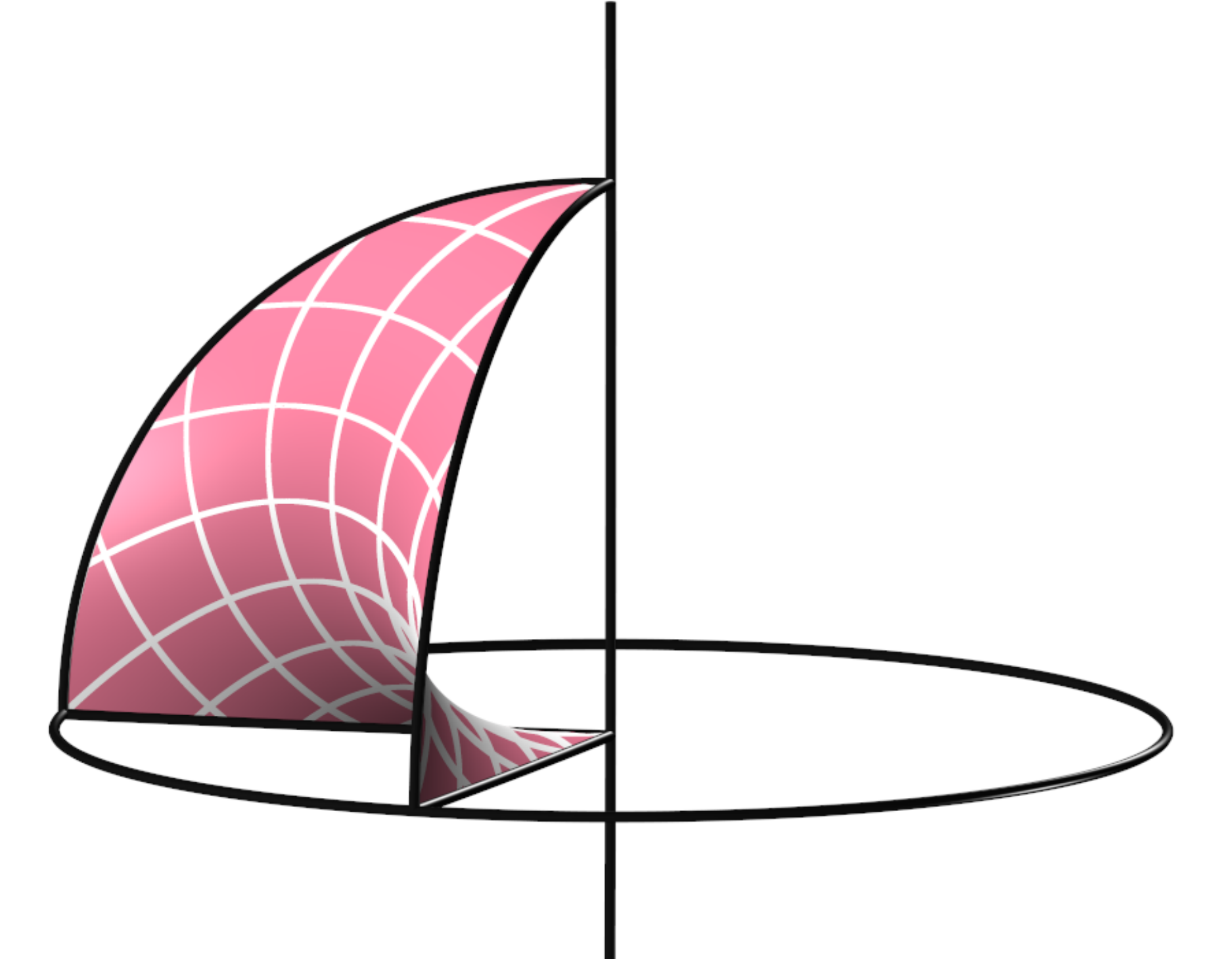}
\hspace{0.75cm}
\includegraphics[width=0.1\textwidth]{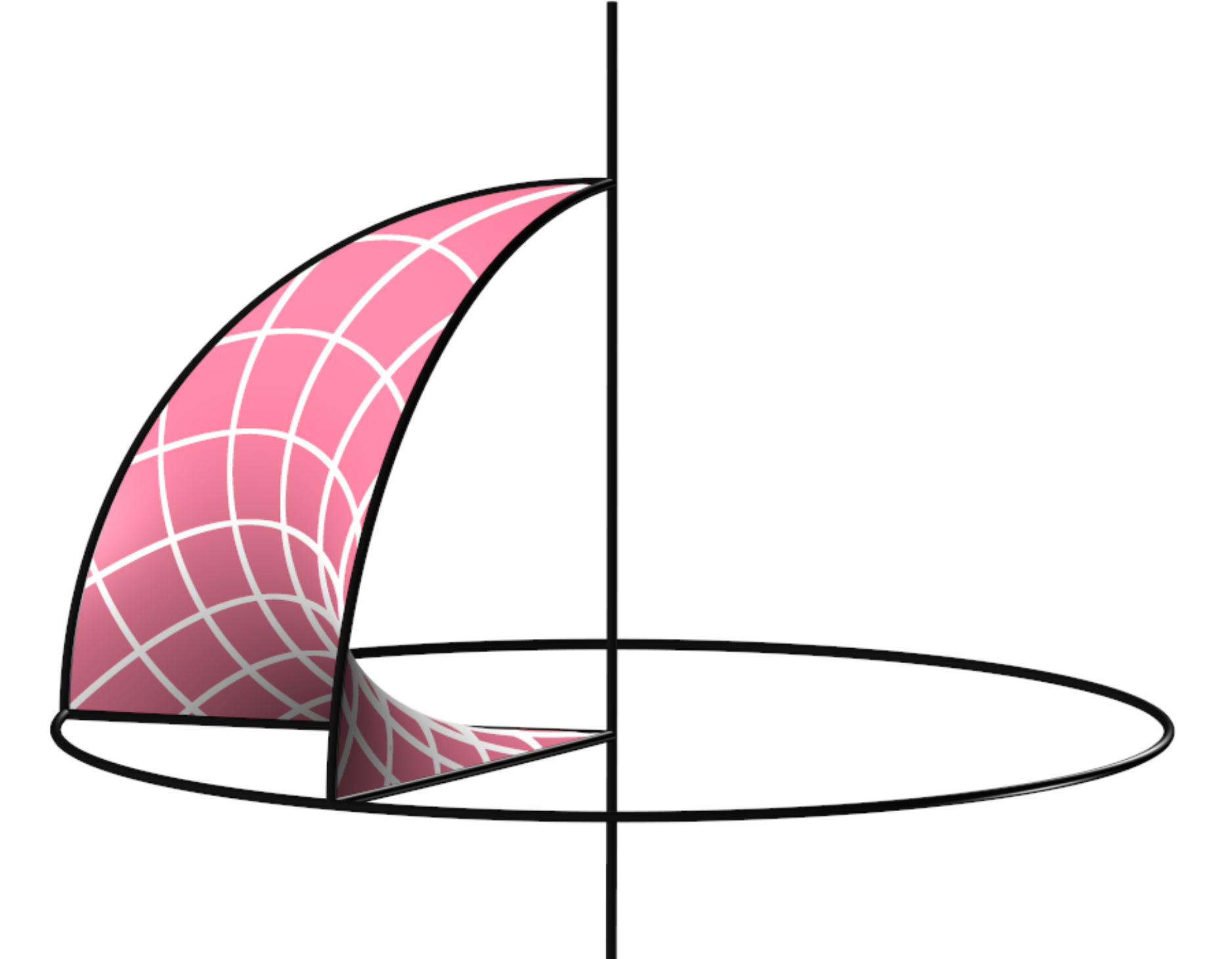}
\hspace{0.75cm}
\includegraphics[width=0.1\textwidth]{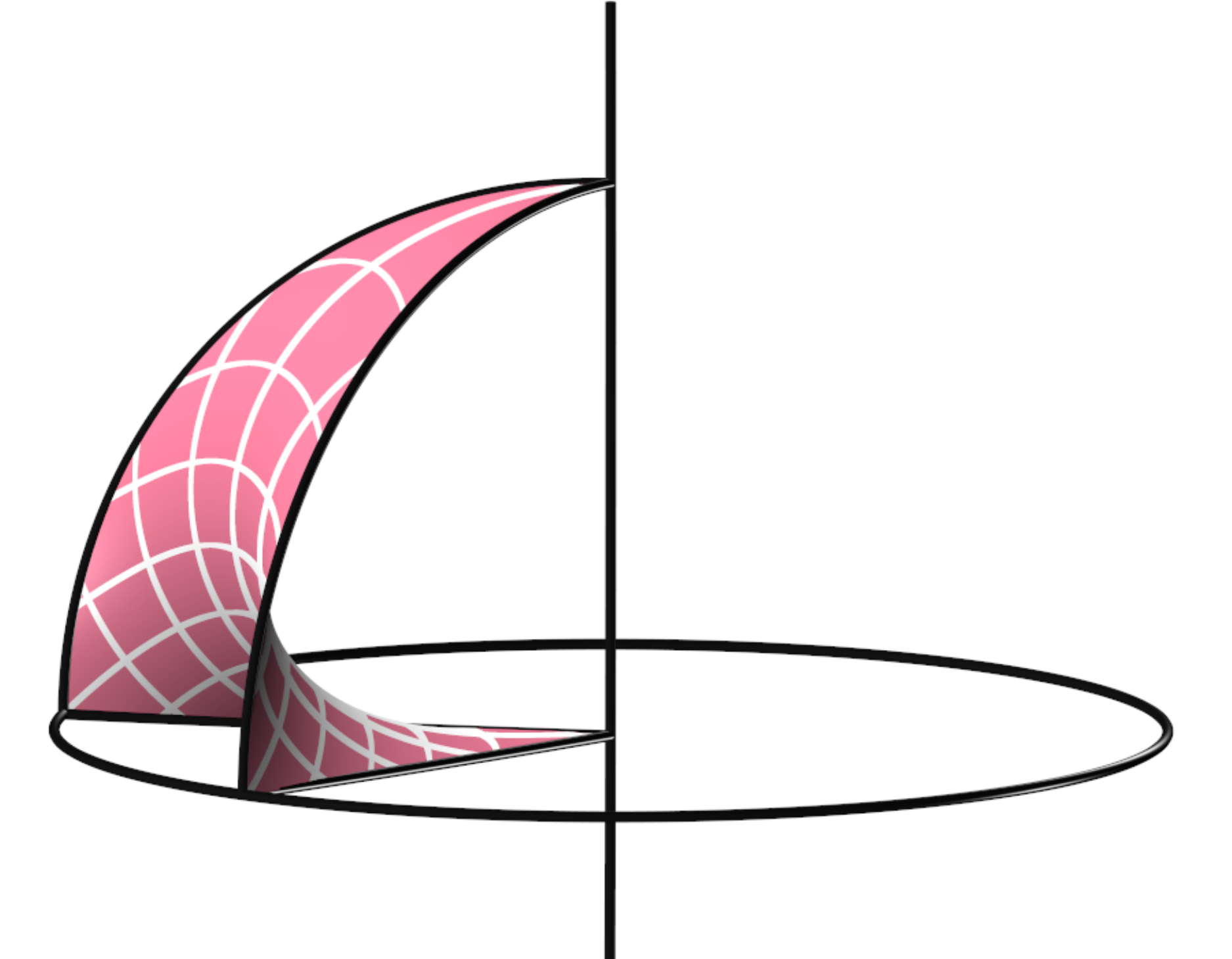}
\hspace{0.75cm}
\includegraphics[width=0.1\textwidth]{lawsonflow-frame3}
\caption{
\footnotesize
The deformation of minimal surfaces
from the Clifford torus to Lawson $\xi_{2,1}.$ The images in the second row show the corresponding Plateau solutions
for varying angles of the geodesic polygon.}
\label{fig:lawson}
\end{figure}

The investigation and construction of closed surfaces with special geometric properties is an important subject
in differential geometry. Of particular interest are minimal surfaces and constant mean curvature (CMC) surfaces
in space forms. 
Global properties of CMC surfaces were first considered by Hopf,
showing that all CMC spheres are round.
 This result was generalized by Alexandrov \cite{Ale} in the 1950s, who showed that the round spheres are the only embedded compact CMC surfaces in $\R^3.$ It was  a long standing conjecture by Hopf that this should be also true for immersed surfaces in Euclidean $3-$space until  Wente  \cite{We} constructed counter examples in the class of tori in 1986. 
Although the generalized Weierstrass representation for CMC surfaces \cite{DPW} gives all CMC immersions from simply connected domains into space forms, global questions (like the construction of closed surface with genus $g\geq2$) 
are very hard to study in this setup. 
The reason is that the moduli space of compact CMC surfaces of genus $g$ is finite dimensional, while the space of conformal
 CMC immersions of a disc or a plane is always infinite dimensional. For tori this problem was substantially simplified in
 the work of
Abresch \cite{Ab}, Pinkall and Sterling \cite{PiSt}, Hitchin \cite{Hi} and Bobenko \cite{B}  in the 1980s, 
by reducing to a finite dimensional problem via integrable system methods.
These methods were used to produce various new examples of CMC tori.
Recently Hauswirth, Kilian and Schmidt studied the moduli space of all
minimal tori in $S^2\times \R$ by integrable systems methods. They proved that properly embedded minimal annuli
in $S^2\times\R$ are foliated by circles \cite{HaKiSch, HaKiSch1}.
 The technique (based on the so-called Whitham deformation of the spectral data) 
 has also been applied for the investigation
of Alexandrov embedded CMC tori in $S^3$ \cite{HaKiSch2}, giving an alternative approach to the conjectures
of Lawson (proved by Brendle \cite{B}) and  Pinkall-Sterling (proved by Andrews and Li \cite{AL} using Brendle's method).
On the other hand, examples of and methods for closed CMC surfaces of higher genus are rare. Lawson \cite{L} constructed closed embedded minimal surfaces in the round 3-sphere for every genus and
Kapouleas \cite{Ka1,Ka2} showed the existence of compact CMC  surfaces in Euclidean 3-space for all genera.
Since these constructions are implicit even fundamental geometric properties like the area
cannot be explicitly computed. \\

 The integrable system approach to CMC surfaces is based on the associated 
$\C^*$-family of flat 
$\SL(2,{\C})$-connections $\lambda\in \C^*\mapsto\nabla^\lambda.$ 
Knowing the family of flat connections is tantamount to knowing the CMC surface, as the surface is given by the gauge
between two trivial connections $\nabla^{\lambda_1}$ and $\nabla^{\lambda_2}$
 for 
$\lambda_1\neq\lambda_2\in S^1\subset\C^*$ 
with mean curvature $H=i\frac{\lambda_1+\lambda_2}{\lambda_1-\lambda_2}.$ An important property of this family of flat connections is the unitarity for the connections along $\lambda \in S^1$. 
In the abelian case of CMC tori, $\nabla^\lambda$ splits for generic $\lambda\in\C^*$ into a direct sum of flat line bundle connections. 
Therefore, the associated $\C^*$-family of flat $\SL(2,\C)$-connections of a CMC torus
is characterized by spectral data parametrizing the corresponding family of flat line bundles on the torus.
For higher genus surfaces, flat $\SL(2,\C)$-connections are generically irreducible and therefore 
the abelian spectral curve theory for CMC tori is no longer applicable.
Nevertheless, if one restricts
to certain symmetric compact Riemann surfaces, it is still possible to  characterize flat symmetric 
$\SL(2,\C)$-connections in terms of flat line bundle connections
on an associated torus via abelianization \cite{He3, HeHe}.
Hence,
the associated family $\nabla^{\lambda}$ 
of a symmetric CMC surface of higher genus is again
determined by spectral data which parametrize flat line bundle connections on a torus. 
The drawback of this approach is that for higher genus the unitarity condition for $\nabla ^{\lambda}$ along the unit circle is only given implicitly in terms of the Narasimhan-Seshadri section. 
Understanding the construction of spectral data satisfying  the unitarity condition is our main purpose.\\

In this paper we propose a more explicit construction of higher genus CMC surfaces 
and the Narasimhan-Seshadri section. The basic idea is to start at a well known surface, e.g. a CMC torus, where the 
closing conditions are well understood in terms of the spectral data. Then the surface is deformed in a direction which 
changes the genus "continuously". The motivating example is the flow of Plateau solutions from the Clifford torus to Lawson's minimal 
surface of genus 2 and beyond by changing angles of the geodesic polygonal boundary (see Figure \ref{fig:lawson}). 
More generally we construct a flow from CMC tori which gives compact and branched CMC surfaces (with controlled 
branch points and branch order) for rational times and immersed CMC surfaces of genus $g$ for discrete times.
We construct this flow by deforming  (gauge equivalence classes of) the associated families of flat connections:
using the abelianization procedure and under the assumption of certain discrete symmetries, the family of flat connections 
can be reduced to a family of flat 
connections over a  4-punctured torus, i.e., connections on a trivial $\C^2$-bundle over the 2-torus with four simple pole-like singularities. The genus of the surface is encoded in the local monodromies, i.e.,
 in the eigenvalues of the residues 
of the flat connections. Deforming these eigenvalues induces a flow from given spectral data of CMC tori to the spectral 
data of higher genus CMC surfaces through branched CMC immersions (see Figure \ref{fig:lawson2} and Figure \ref{fig:lawson3}). \\

The paper is organized as follows: In Section \ref{spec_chap} we recall the integrable systems techniques for
CMC surfaces from a general point of view.
We then describe the spectral curve theory for CMC tori in our setup and consider two classes of examples, the homogeneous CMC tori and the 2-lobed Delaunay tori.
In Section \ref{arbitrary_genus}
we develop a spectral curve theory for higher genus CMC surfaces with discrete symmetries 
generalizing the theory for CMC tori  (see \cite{Hi} or Section \ref{CMC_tori}) 
and the theory for Lawson symmetric CMC surfaces of genus $2$ in \cite{He3}.
In Section \ref{Whitham flow} we define the flow on spectral data and prove its short time existence 
for initial data given by homogeneous CMC tori. In particular we prove that flowing the spectral data for the Clifford torus is equivalent to the flow of Plateau solutions by changing 
the angle of the fundamental piece in Lawson's construction for the minimal surfaces $\xi_{g,1}.$
We also prove the short time existence of two distinct flows starting at the  
2-lobed Delaunay tori of spectral genus $1$.
Geometrically, these two flows can be distinguished by the fact that the Delaunay tori are cut along curvature lines
between different pairs of points as shown in  Figure \ref{fig:lawson2} and Figure \ref{fig:lawson3}. 
In our theory, the existence of two different flow directions corresponds to the choice whether the family of flat connections contains only stable
points inside the unit disc.\\

\thanks{Acknowledgement: 
We would like to thank the anonymous  referee for his helpful comments which enabled us to improve the
presentation of the paper.
The first author is supported by the European Social Fund and by the Ministry of Science, Research and the Arts Baden-W\"urtemberg, the other authors are supported by the DFG through the project HE 6829/1-1.}

\section{Integrable Surface Theory}\label{spec_chap}
In this 
 Section, we recall the basics of integrable surface theory. We explain how CMC tori can be described from our perspective and consider two families of CMC tori as examples in detail.\\

In the integrable systems approach to CMC surfaces the crucial step is to translate the elliptic PDE $H \equiv const$ into a system of ODEs on the Riemann surface: to a CMC immersion $f\colon M\to S^3$ we can associate  a $\C^*-$family of flat $\SL(2, \C)$ connections as was introduced by
Hitchin \cite{Hi}. In order to construct CMC surfaces it is thus sufficient to write down appropriate families of flat connections, which will turn out to be easier than to solve the PDE directly.\\

In order to make the reader more comfortable with the correspondence between CMC surfaces and their associated families of flat connections, we briefly recall its construction starting with the case $H=0$ (following \cite{Hi}).
Consider $S^3$ as the Lie group $S^3\cong SU(2)$, then its Levi-Civita connection is given by  $$\nabla=d+\tfrac{1}{2}g^{-1} dg,$$ where $g^{-1} dg$ is the Maurer-Cartan form acting on
$\mathfrak{su}(2)$ by the adjoint representation. It gives rise to the (special unitary) spin connection 
$$\nabla=d+\tfrac{1}{2}g^{-1} dg$$ on the trivial $\C^2$-bundle over $S^3$, where $\mathfrak{su}(2)$ is acting on $\C^2$ through the standard representation.\\

Let $f: M \rightarrow S^3$ be a minimal conformal immersion from a Riemann surface $M$.  
Pulling back $\nabla$ by $f$ we obtain a special unitary connection (with respect to the standard hermitian inner product) on $\C^2\times M\to M,$
 also denote by $\nabla$, given by 
 $$\nabla=d+\tfrac{1}{2}f^{-1}df.$$
Since $f$ is minimal and conformal,
  it is harmonic, giving rise to the following equations:
\begin{equation}\label{harmeq}
\begin{split}
d^\nabla\phi&=0\\  d^\nabla*\phi&=0,
\end{split}
\end{equation}
where $\phi := \tfrac{1}{2}f^{-1}df$ is the connection $2-$form and $*$ is the negative of the Hodge star (i.e. $*dz=idz,\; *d\bar z=-id\bar z$ for a locally defined holomorphic function $z\colon U\subset M\to \C$).
The first equation of \eqref{harmeq} is a reformulation of the Maurer-Cartan equation for $f^{-1}df$ and the second equation is equivalent to $f$ being harmonic.
After decomposing $\phi$ into its complex linear part $\Phi=\tfrac{1}{2}(\phi-i*\phi)$ and its complex anti-linear part $-\Phi^*$, so that $\phi=\Phi-\Phi^*,$ 
we can rewrite \eqref{harmeq}
as
\begin{equation}\label{harmeq2}
\begin{split}
d^\nabla\Phi&=\dbar^\nabla\Phi=0\\  d^\nabla\Phi^*&=\partial^\nabla\Phi^*=0,
\end{split}
\end{equation}
where$\partial^\nabla:=\tfrac{1}{2}(\nabla-i*\nabla)$ and  $\dbar^\nabla:=\tfrac{1}{2}(\nabla+i*\nabla).$\\

Consider now the $\C^*-$family of special linear connections
$$\nabla^\lambda=\nabla+\lambda^{-1}\Phi-\lambda\Phi^*, \quad \lambda \in \C^*,$$
which is called the associated family of flat $\SL(2, \C)$ connections of $f$. That the curvature for all $\nabla^{\lambda}$ vanishes for all $\lambda$ can be seen as follows: firstly, by equation \eqref{harmeq2} the curvature of $\nabla^\lambda$
 is independent of $\lambda\in\C^*.$ Further, we have by definition $\nabla^{1}=d+f^{-1} df$ and  $\nabla^{-1}=d$ are both trivial, and hence of zero curvature. \\

In this formulation the conformality of $f$ is equivalent to $\Phi$ being nilpotent  (see Proposition 1.8 of \cite{Hi}) and $f$ being immersed translates to $\Phi$ being non vanishing.
For minimal tori the connections
$\nabla^\lambda$ are unitary for $\lambda\in S^1,$ since in this case the connection $1-$form $\lambda^{-1}\Phi-\lambda\Phi^*$
is skew-adjoint. Moreover, the immersion
$f$ is constructed as the gauge transformation between $\nabla^{1}$ and $\nabla^{-1}$ (where we have identified $S^3=SU(2)$). \\

For conformally parametrized CMC surfaces with non-zero mean curvature $H\neq0,$ the construction becomes slightly more complicated, since $f$ is
no longer harmonic.
Nevertheless, by the Lawson correspondence or by the fact that the Gauss map of a CMC surface in $S^3$ is harmonic,
we still obtain an associated family of flat connections as follows:
split $f^{-1}df=\Psi-\Psi^*$
into complex linear and complex anti-linear parts as before, and rescale
$$\Phi=\frac{\lambda_2}{1+\lambda_2}\Psi \;\; \text{ and }\; \;  \Phi^*=\frac{1}{1+\lambda_2}\Psi^*,$$
where $$\lambda_2=\frac{- i H+1}{ i H+1}\in S^1\subset\C^*.$$
Define $\nabla:=d+\Phi-\Phi^*$ (which is no longer the pullback of the spin connection of $S^3$) and consider  the associated family of connections 
\begin{equation}\label{associated_family}
\lambda\in\C^*\mapsto \nabla^\lambda=\nabla+\lambda^{-1}\Phi-\lambda\Phi^*.
\end{equation}
Again we have $\nabla^{-1}=d$ and $\nabla^{\lambda_2}=d+f^{-1}df$ are both trivial and that the residue term $\Psi$ at $\lambda=0$ is nilpotent. Vanishing of the curvature of $\nabla^\lambda$ for all $\lambda$ follows from its independence of $\lambda.$ This is equivalent to $d^\nabla\Phi=d^\nabla\Phi^*=0$ which itself follows from a not too long computation involving
the Maurer-Cartan equation and the  formula relating the laplacian of the immersion and its mean curvature; see for example Lemma 2.2 in \cite{SKKR}.\\

This shows one direction of 
Theorem \ref{The1} below.
The reverse direction can be proven by reversing the computations. See also
~\cite{B,He1} for more details.
\begin{The}[\cite{Hi,B}]
\label{The1}
Let $f\colon M\to S^3$ be a conformal CMC immersion. Then its  associated family of flat $\SL(2,\C)$ connections \eqref{associated_family}
is unitary for $\lambda\in S^1\subset\C^*$ and trivial for
$\lambda_1\neq\lambda_2\in  S^1$.  Conversely, given such a family of flat $\SL(2, \C)$ connections with nilpotent $\Phi$, the immersion $f$ given by the gauge between
$\nabla^{\lambda_1}$ and $\nabla^{\lambda_2}$ (identifying
$\SU(2)= S^3$) is conformal and of constant mean curvature 
$H=i\frac{\lambda_1+\lambda_2}{\lambda_1-\lambda_2}$ and has $\nabla^\lambda$ as its associated family.
\end{The}
\begin{Rem} 
Following \cite{KSS} we call the spectral parameter $\lambda_1,\lambda_2\in\C^*$ Sym points. 
The existences of two Sym points is the {\em extrinsic closing condition} while the unitarity of the connections $\nabla^\lambda$ along the unit circle is the {\em intrinsic closing condition}.
\end{Rem}
\begin{Rem}For compact CMC surfaces which are
not totally umbilic the generic connection $\nabla^\lambda$ of the
associated family is not trivial~\cite{Hi}. Moreover, for CMC immersions from a
compact Riemann surface of genus $g\geq2$, the connection
$\nabla^\lambda$ of the associated family is irreducible for generic $\lambda\in\C^*$~\cite{He1}.
\end{Rem}
\begin{Rem}\label{gensym}
In the above theorem we can weaken the condition that there are two spectral parameter $\lambda_1,\lambda_2\in\C^*$ such that $\nabla^{\lambda_k}$ is trivial
(for $k=1,2$) as follows: it is sufficient that the two connections $\nabla^{\lambda_k}$ have the same monodromy representation with values in $\Z_2=\{\pm\Id\}\subset  SU(2).$
Because this subgroup is the center of $SU(2)$, the gauge between these two connections is still well-defined in this situation and it is a conformal CMC immersion with $H=i\frac{\lambda_1+\lambda_2}{\lambda_1-\lambda_2}$.
By an abuse of notation, we also call this family {\em the associated family of flat connections} even if it differs from \eqref{associated_family} by a
$\lambda$-independent  shift given by tensoring with a flat $\Z_2$ line bundle.
\end{Rem}

For $M$ a torus Hitchin \cite{Hi} developed a theory classifying all possible families $\nabla^{\lambda}$ and explicitly parametrizing the corresponding CMC immersions. This procedure depends crucially on the fact that the first fundamental group of the torus is abelian so there is no straightforward generalization of the theory to higher genus surfaces. For higher genus CMC surfaces (to which we restrict if not otherwise stated) it has been proven useful to first consider the gauge equivalence classes of the connections $\nabla^{\lambda}$. Thus let 
$\mathcal A^2=\mathcal A^2(M)$  be the moduli space of flat $\SL(2,\C)$ connections modulo gauge transformations.  This space inherits the structure of a complex analytic variety of dimension $6g-6$ whose singular set consists of the gauge classes of reducible connections. This can be seen by identifying $\mathcal A^2$ with the character variety of $\SL(2,\C)$
representations of the fundamental group of $M$ modulo conjugation as in \cite{G}, or by carrying out an
infinite dimensional K\"ahler reduction as in \cite{Hi1}.\\

For a CMC surface $f$ with associated family of flat connections $\nabla^{\lambda}$ we consider the map
\begin{equation}\label{mathcald}
\mathcal D\colon \C^* \to\mathcal A^2,\; \lambda\mapsto[\nabla^\lambda].\end{equation}
Although in general $\mathcal D$ does not uniquely determine a CMC surface in $S^3$, those CMC surfaces
corresponding to the same $\mathcal D$ are related by a well understood transformation called dressing (see  \cite{BDLQ} for general information about dressing transformations in the setup of associated families). We note that we only consider dressing transformations 
which preserve the topology of the surface. 
Without the topological constraint (i.e., for a simply connected CMC surface) the space of dressing transformations is infinite dimensional, whereas for compact surfaces this space is finite dimensional.
In the abelian case of CMC tori, the dressing transformations  are induced by a shift of the eigenline bundle of the spectral curve (see \cite{McI})
 and are usually called  isospectral deformations. For higher genus CMC immersions 
the space of these dressing transformations are "based" at the  finitely many points $\lambda^i\in\C^*$ at which the holomorphic map
$\mathcal D\colon\C^*\to \mathcal A^2$ represents the gauge class of a reducible connection  (Theorem 7 in  \cite{He3}). \\

The following theorem might be considered as a variation or generalization of the DPW method \cite{DPW} . It
summarizes the above discussion
and generalizes it to the case of (possibly) branched CMC surfaces. Under a bound on the number of branch points the theory for immersed CMC surfaces carries over. The branched CMC surfaces constructed in this paper
(in particular in Theorem \ref{branched_CMC}) obey the given bound. 

\begin{The}\label{lifting_theorem} Let $M$ be a compact Riemann surface of genus $g$ and
let $\mathcal D\colon\C^*\to\mathcal A^2=\mathcal A^2(M)$ be a holomorphic map satisfying
\begin{enumerate}
\item the unit circle $S^1\subset\C^*$ is mapped
into the real analytic subvariety consisting of gauge equivalence classes of unitary flat connections,
\item around $\lambda=0$ there exists a  local lift $\tilde\nabla^\lambda$ of $\mathcal D$ with an expansion $$\tilde\nabla^\lambda\sim\lambda^{-1}\Psi+\tilde\nabla^0 +  \text{higher order terms in } \lambda$$ for a nilpotent $\Psi\in\Gamma(M,K\End_0(V)),$
\item there are two distinct points
$\lambda_1,\lambda_2\in S^1\subset\C^*$ such that $\mathcal D(\lambda_k)$  $k=1,2$ represents the trivial gauge class.
\end{enumerate}
Then
there exists a (possibly branched) CMC surface $f\colon M\to S^3$ inducing the map $\mathcal D$
as the family of gauge equivalence classes $\mathcal D(\lambda)=[\nabla^\lambda]$. The branch points of $f$ are given by the zeros of $\Psi$
and $f$ is unique up to dressing transformations if the number of zeros of $\Psi$ (counted with multiplicity) is less than $2g-2$.

Conversely, every CMC surface determines a holomorphic $\C^*$-curve into $\mathcal A^2$ via \eqref{mathcald}.
\end{The}
\begin{proof} 
We have two cases to consider. Either the family $\tilde \nabla^\lambda$ is reducible for all $\lambda$ or it is generically irreducible.
In the first case the corresponding CMC surface is a branched covering of a CMC torus by \cite{Ge} and the reconstruction of the surface is carried out in \cite{Hi}. In the latter case the proof of the construction works analogously to the proof of Theorem 8 in \cite{He3}. Uniqueness part follows with the same arguments as in Theorem 7 of \cite{He3} provided that the Higgs pair $(\dbar^{\tilde\nabla^0},\Psi)$ is stable, i.e.,  every $\Psi$-invariant line subbundle of $V$ has negative degree. Since $\Psi$ is nilpotent, the kernel bundle $L=\ker \Psi$ is the only $\Psi$-invariant subbundle. Moreover,
because $\Psi$ gives rise to a non-vanishing holomorphic section of $KL^2$ (see Section 2.1 in \cite{He1} or the proof of Lemma \ref{atlambda=0again} below)
 the degree of $L$ is negative if and only if the number of zeros of $\Psi$ is less than the degree of the canonical bundle.
\end{proof}
\begin{Rem}\label{gensym2}
As in Remark \ref{gensym} we can weaken the extrinsic closing condition so that $\mathcal D(\lambda_k)$ only need to represent for $k=1,2$ the same flat
$\Z_2$-bundle.
\end{Rem}
\begin{Rem}
 Since $\mathcal D$ is holomorphic and maps into the real analytic subvariety of $\mathcal A^2$ consisting of gauge equivalence classes of unitary flat connections for $\lambda \in S^1$,  $\mathcal D$ is already determined by its values on $D_1:= \{\lambda\  | \ |\lambda|^2 \leq1\}$ as a consequence of the Schwarzian reflection principle. \end{Rem}

\section{CMC tori revisited}\label{CMC_tori}
In this section we consider the case of CMC tori and rephrase the well-known spectral curve theory for CMC tori of \cite{Hi} in the context of Theorem  \ref{lifting_theorem}.  Additionally, we prove two technical lemmas showing the non-degeneracy of the initial data of the flow defined in Section \ref{Whitham flow}.
  
 \subsection{Flat line bundles on tori} 
Consider the Riemann surface of genus $1$ given by 
 \[T^2=\C/\Gamma,\]
 where $\Gamma=2\Z+2\tau\Z$ is the lattice generated by $2$ and $2\tau$ for some $\tau\in\C$ with 
$\Im(\tau)>0.$ We assume by a change of basis that 
$-\tfrac{1}{2}<\Re(\tau)<\tfrac{1}{2}.$
The Jacobian of $T^2$
is given by
\[\mathrm{Jac}(T^2)=\overline{H^0(T^2,K)}/\Lambda\] where
 \[\Lambda=\{\bar\eta\in \overline{H^0(T^2,K)}\mid \int_\gamma (-\eta+\bar\eta)\in 2\pi i\Z \text{ for all closed curves } \gamma \text{ in } T^2\}.\] 
 
The Jacobian can be viewed as the moduli space of holomorphic structures on the 
 trivial line bundle. By fixing the global anti-holomorphic 1-form $d\bar w$ 
 we can identify
 \[\Lambda \cong \tfrac{\pi i}{\tau-\bar\tau}\Z+\tfrac{\pi i\tau}{\tau-\bar\tau}\Z.\]
The moduli space of flat line bundle connections  $\mathcal A^1=\mathcal A^1(T^2)$ is similarly given by
\[\mathcal A^1=\mathcal H^1(T^2,\C)/\tilde\Lambda,\] 
where  $\mathcal H^1(T^2,\C)$ is the space of complex valued harmonic 1-forms on $T^2$
and  \[\tilde\Lambda=\{\omega\in \mathcal H^1(T^2,\C)\mid \int_\gamma \omega\in 2\pi i\Z \text{ for all closed curves } \gamma \text{ in } T^2\}\]
is a  lattice of full rank. The moduli space of flat line bundle connections can be seen as an affine holomorphic bundle
\begin{equation}\label{affine_bundle}
(.)''\colon\mathcal A^1(T^2)\to \mathrm{Jac}(T^2)
\end{equation}
by assigning to a representative $\nabla$ of a gauge class  $[\nabla] \in\mathcal A^1(T^2)$ the isomorphism class of the induced holomorphic  structure $\dbar^\nabla.$ 

\subsection{Spectral curve theory for CMC tori}

The main difference between CMC tori and  higher genus CMC surfaces  is that the first fundamental group of a torus $\pi_1(T^2)$ is abelian. Thus for a flat unitary connection and $p \in T^2$ there is a basis of $V_p$ which simultaneously diagonalizes the monodromy of $\nabla$ along both generators of $\pi_1(T^2,p).$  Therefore $\nabla$ splits into two line bundle connections on its parallel eigenlines $L^{\pm}$, which are dual to each other. In fact, a generic flat $SL(2, \C)$ connection on a torus has diagonalizable monodromy and splits into two flat line bundle connections.\\

For the associated family $\nabla^{\lambda}$ of a CMC torus this implies that $\nabla^{\lambda}$ splits into flat line bundle connections on the eigenlines of the monodromy $L^{\pm}$ for generic $\lambda \in \C^*,$ since the family is unitary along $S^1.$ More concretely, this means that for a generic $\lambda \in \C^*$ the connection $\nabla^{\lambda}$  is gauge equivalent to
\begin{equation}\label{connection}
d+ \begin{pmatrix} -\chi(\lambda) d\bar w + \alpha(\lambda) dw &0\\0&  \chi(\lambda) d\bar w  - \alpha(\lambda) dw\end{pmatrix},\end{equation}
with respect to the splitting $V = L^+_{\lambda} \oplus L_{\lambda}^-$, where $dw$ is the (non trivial) holomorphic 1-form on the torus.
The functions $\chi(\lambda)$ and $\alpha(\lambda)$ are locally defined and holomorphic in $\lambda$ away from exceptional values of the spectral parameter $\lambda_i$ (and are independent of $w \in T^2$). \\

In fact, it is shown in \cite{Hi} that $\nabla^{\lambda}$ is gauge equivalent to \eqref{connection} except at finitely many points $\lambda_1, \dots, \lambda_k$. 
At these points $\lambda_i$ the eigenlines of the monodromy $L^{\pm}_\lambda$  coalesce. The necessary condition for this to happen is that the corresponding flat line bundle connection is self-dual. Then the eigenlines are equipped with a flat spin connection, implying that the trace of the monodromy of $\nabla^{\lambda_i}$ along an arbitrary closed curve is $\pm2.$ \\

 In order to obtain globally defined holomorphic maps $\chi$ and $\alpha$  we replace the spectral plane $\C^*$ by a double covering of $\C^*$ branched at $\lambda_i.$ This new parameter space can be compactified by adding two points over $\lambda=0$ and $\lambda = \infty$ as shown in \cite{Hi}. 
The resulting (compact) hyperelliptic Riemann surface determined by the equation \[\Sigma : \xi^2 = \lambda \Pi_{i =1}^k(\lambda - \lambda_i)\]
 is called the spectral curve of the CMC torus.\\

By \eqref{associated_family} the associated  family of holomorphic structures 
$\bar \del^{\lambda}=(\nabla^\lambda)''$  extends through $\lambda= 0$ while the family of anti-holomorphic structures $\del^{\lambda}=(\nabla^\lambda)'$ has a simple pole at $\lambda = 0.$ Thus $\nabla^{\lambda}$ and its flat eigenline bundles are parametrized
by the spectral data $(\Sigma,\chi,\alpha)$, where
\[\chi \colon \Sigma\setminus\{\infty\}\to \mathrm{Jac}(T^2) \cong \C /\Lambda\] is an odd holomorphic map 
to the Jacobian of  $T^2$ which is isomorphic to $\C/ \Lambda$ for a lattice $\Lambda,$  and \[\alpha \colon \Sigma\to \overline{\mathrm{Jac}(T^2)}\] is an odd meromorphic map 
to the moduli space of anti-holomorphic line bundles over $T^2$ of degree $0$ whose only pole is a first order pole at $\lambda = 0.$ \\

In general, a CMC immersion is not uniquely determined by its spectral data $(\Sigma,\chi,\alpha)$ due to the fact that a 
CMC immersion whose spectral curve $\Sigma$ has spectral genus $\geq1$ always admit non-trivial isospectral deformations (including reparametrizations). These  deformations are given by shifts of the so-called (holomorphic) eigenline bundle 
$E_p\to\Sigma$
(in which definition it is necessary to fix a base point $p \in T^2$) in the Picard variety. 
Hitchin has shown in \cite{Hi} that the spectral data $(\Sigma,\chi,\alpha)$
together with the eigenline bundle $E_p$ uniquely determine the CMC immersion as a conformal map;
see also \cite{McI} for the relationship between the eigenline bundle and dressing transformations. For
the CMC tori of
spectral genus  $\leq2$ an isospectral deformation only changes the parametrization of the CMC surface.
Hence, we ignore the eigenline bundle and the isospectral deformations in the discussion of CMC tori of spectral genus $\leq2$  in Section \ref{Tori_spec_gen_0} and \ref{Tori_spec_gen_1} below.

The extrinsic closing condition, i.e., the condition that the immersion has no periods, is guaranteed by the existence
of two points $\xi_1,\xi_2\in \Sigma$ lying over the  Sym points $\lambda_1\neq\lambda_2\in S^1 \subset \C^*$, i.e. $\lambda(\xi_k)=\lambda_k$ for $k=1,2$, 
which satisfy
 \[ \chi(\xi_1)=\chi(\xi_2)= 0\in \mathrm{Jac}(T^2).\]

\begin{Rem}\label{holo_info}
The map $\chi\colon\Sigma\setminus\{\infty\}\to \mathrm{Jac}(T^2)$ uniquely determines the meromorphic map $\alpha$ by the condition that $\nabla^{\lambda}$ is unitary along $S^1 \subset \C^*,$ i.e.,  $\alpha(\xi) = \overline{ \chi (\xi)}$ for points $\xi\in\lambda^{-1}(S^1)$ lying over the unit circle.
Also, the map $\chi$ is already determined by its values on the preimage of the
closed unit disc.
Thus, to construct a map $\mathcal D$ satisfying the properties of Theorem \ref{lifting_theorem}
boils down to write down an appropriate holomorphic map $\chi$ on the preimage of the disc $D_{1+ \epsilon} \subset \C^*.$ The main issue will be  that the corresponding map $\alpha$ must have simple pole over $\lambda = 0.$ For tori, in contrast to the spectral data of higher genus CMC surfaces, this condition is understood and is equivalent to $\chi \colon\Sigma\setminus\{\infty\}\to \mathrm{Jac}(T^2)$ having a first order pole at the point lying over $\lambda = \infty,$
i.e., $d\chi$ is a meromorphic 1-form with a pole of order $2$ over $\lambda = \infty$ without residue.
\end{Rem}
\begin{Rem}\label{shift}
In order to obtain a unified theory for higher genus surfaces, we apply an overall shift
of $\mathrm{Jac}(T^2)$ by the half lattice point $-\tfrac{\pi i (1+ \tau)}{4\tau}$. The maps $\chi$ and $\alpha$ are shifted accordingly:
\begin{equation}\label{shift-chi}
\begin{split}
\chi_{shift}&=\chi-\tfrac{\pi i (1+ \tau)}{4\tau}\\
\alpha_{shift}&=\alpha+\tfrac{\pi i (1+ \bar\tau)}{4\bar\tau}.\\
\end{split}
\end{equation}
 In the following, abusing notations, we will denote $\chi_{shift}$ and $\alpha_{shift}$  again by $\chi$  and $\alpha$.
 Note that this shift gives rise to tensoring the associated family by a flat $\Z_2$-bundle; see also \ref{gensym} and \ref{gensym2}.
\end{Rem}

We next determine the spectral data for those tori, which  will serve as the initial condition for the short time existence of our flow in the Sections \ref{31} and \ref{32}.
\subsection{Homogeneous CMC tori}\label{Tori_spec_gen_0}
Homogenous tori are the simplest CMC tori in $S^3.$ They are  the product of two circles with different radii and can be parametrized by 
$$f(x,y) = (\tfrac{1}{r} e^{ir x}, \tfrac{1}{s} e^{isy}) \subset S^3 \subset \C^2, \quad r,s \in \R, \quad r^2+ s^2 = 1.$$
Thus (simply wrapped) homogenous tori always have rectangular conformal types so
we identify $T^2=\C/\Gamma$ where $\Gamma=2\Z+2\tau \Z$ for some
$\tau\in i \R^{\geq1}$. The Jacobian of $T^2$ is the dual torus given by
\[\mathrm{Jac}(T^2)=\overline{H^0(T^2,K)}/\Lambda \cong \C/(\tfrac{\pi i}{2 \tau}\Z+\tfrac{\pi i}{2}\Z),\]
where the isomorphism is given by the trivializing section $d\bar w.$
For homogenous tori the spectral curve $\Sigma$, which  has genus $0$, is defined by the algebraic 
equation
\[\xi^2=\lambda.\]
Since $\Sigma \cong\CP^1$ is simply connected, every meromorphic map
\[\chi\colon  \CP^1  \to \mathrm{Jac}(T^2)\]
lifts to a meromorphic function
$\hat{\chi} \colon  \CP^1  \to \C$.
By Remark \ref{holo_info} and because $\chi$ is odd we obtain
\begin{equation}\label{hom_tori_spec}
\hat{\chi}(\xi)=\tfrac{\pi i R}{4\tau}\xi d\bar w+\gamma
\end{equation}
for some $R\in\C^*$ and $\gamma\in\tfrac{1}{2}\Lambda.$ 
The constant term $\gamma$ determines the spin class of the corresponding CMC immersion.
By rotating the spectral plane we may assume that $R\in\R^{>0}.$\\

The extrinsic closing condition is that there are two trivial connections $\nabla^{\lambda_1}$ and $\nabla^{\lambda_2}$ for distinct $\lambda_1\neq\lambda_2 \in S^1.$ If the intrinsic closing condition holds the extrinsic closing condition is equivalent to $\chi(\pm\xi_i) = 0 \in \mathrm{Jac}(T^2)$ at all $\xi_i$ satisfying  
$\xi_i^2 = \lambda_i$.   After applying the shift \eqref{shift-chi} the extrinsic closing condition is equivalent to the existence of four points
\[\pm\xi_1,\pm\xi_2\in S^1\subset\C^*\]
such that 
\[\hat{\chi}(\pm\xi_{1,2})\in \tfrac{\pi i (1+\tau)}{4\tau}d\bar w+\Lambda.\]

The function $\hat\chi$ is linear for $\Sigma \cong \C P^1$ thus the image of $S^1$ under $\hat \chi$ is a circle itself. Further, since $\Gamma$ is rectangular the lattice $\Lambda$ of its Jacobian is also rectangular. A computation (comparing the monodromies of the $\nabla^{\lambda}$ corresponding to $\chi$ and the actual associated family) shows that the spectral data of homogenous tori are given by the choice $\gamma = 0$ (after the shift) and 
\begin{equation}\label{homogeneous_sym}
R=\sqrt{1+\tau\bar\tau},
\end{equation}
see \S  6 of \cite{Hi}. Note that \eqref{homogeneous_sym} gives the smallest possible $R$ in \eqref{hom_tori_spec} for which the image of $S^1$ under $\chi$ contains four lattice points of $\Lambda.$

\subsection{2-lobed Delaunay tori}\label{Tori_spec_gen_1}
Next we describe the spectral data $(\Sigma,\chi)$ for certain CMC tori of revolution. Such a torus is given by the rotation of a profile curve in the upper half plane, 
viewed as the hyperbolic plane  $H^2,$ around the $x-$axis, where we consider $S^3$ as the one point compactification of $\R^3.$ The torus has constant 
mean curvature if and only if its profile curve is elastic in $H^2$. The $2$-lobed Delaunay CMC tori are those whose profile curve closes after two periods of its 
geodesic curvature in $H^2$; see \cite{LHe, KSS} for details.  The conformal type of a CMC torus of revolution is rectangular and determined by a 
lattice $\Gamma=2\Z+2\tau \Z$ for some
$\tau\in i \R^{>0}$. The corresponding spectral curve $\Sigma$ is a torus and thus it can be identified with $\C / \Gamma_{spec}$ for a lattice $ \Gamma_{spec}$.  The spectral curve branches over $\lambda = 0$ and  the branch points reflect across the unit circle, thus $\Sigma= \C / \Gamma_{spec}$ 
is of rectangular conformal type and can be identified with
\[\Sigma=\C/(\Z+\tau_{spec}\Z),\] where $\tau_{spec}\in i\R.$
We use the coordinate $\xi$ on the universal covering $\C$ of $\Sigma$ and consider
\[\lambda\colon\C/(\Z+\tau_{spec}\Z)\to\CP^1\]
as the holomorphic map of order two determined by its ramification points
$$[0],[\tfrac{1}{2}],[\tfrac{1+\tau_{spec}}{2}],[\tfrac{\tau_{spec}}{2}]\in\C/(\Z+\tau_{spec}\Z)$$ and its branch points
\begin{equation}\label{lambda-r}
\lambda([0])=0,\, \lambda([\tfrac{1}{2}])=r, \, \lambda([\tfrac{1+\tau_{spec}}{2}])=\frac{1}{r}, \,  \lambda([\tfrac{\tau_{spec}}{2}])=\infty,
\end{equation} for some $r \in (0,1).$
The preimage of the unit circle $S^1\subset\C^*$ under $\lambda$ has two components
\[C^\pm=\{[s\pm\tfrac{1}{4}\tau_{spec}]\in\C/(\Z+\tau_{spec}\Z)\mid s\in\R\}\]
and the preimage of the unit disc is
\[\Sigma^0=\{[x+y i]\in\C/(\Z+\tau_{spec}\Z)\mid x,y\in\R,\, \frac{i}{4}\tau_{spec}<y<-\frac{i}{4}\tau_{spec}\}.\]
Because there are  no meromorphic functions of degree $1$ on a compact Riemann surface of genus $g> 0$, the map
\[\chi \colon\Sigma\setminus\lambda^{-1}(\{\infty\})\to \mathrm{Jac}(T^2)\cong  \C/\Lambda\]
must have periods, i.e., there is no global lift of $\chi$ mapping to the universal covering of $\mathrm{Jac}(T^2)$.
Nevertheless, the differential $d \chi$ is a well-defined meromorphic 1-form (with values in $\overline{H^0(T^2,K)} \cong \C$) with a double pole at $[\frac{\tau_{spec}}{2}]=\lambda^{-1}(\infty)$ and no other singularities. As
$\chi$ is odd, we know that $\chi([0])$ is a spin point (or half lattice point) of $\mathrm{Jac}(T^2).$ 
The 2-lobed Delaunay tori bifurcate from the homogeneous tori, i.e.,  for $\tau_{spec}\to0$ the spectral data converge to the spectral data of a homogeneous torus, as is shown in \cite{KSS}. Thus (after the shift) $\chi([0])=0 \in \mathrm{Jac}(T^2)$
is the trivial holomorphic line bundle. Moreover,  $d\chi$ has only a period
in the direction $1\in(\Z+\tau_{spec}\Z)=\pi_1(\bar\Sigma),$ and no period in the direction of $\tau_{spec},$ as one can deduce by carefully studying the
behavior of the spectral data for $\tau_{spec} \to 0.$
Thus
\begin{equation}\label{weierstrass_zeta_L}
d\chi=(a\wp(\xi-\tfrac{\tau_{spec}}{2})+b) d\xi\otimes\frac{\pi i}{2\tau} d\bar w,
\end{equation}
where $\wp$ is the Weierstrass $\wp$-function on $\C/\Lambda$ and $a$ and $b$ are uniquely determined by 
\begin{equation}\label{2lobe_spec_periods}
\begin{split}
\int_{\tau_{spec}}(a\wp(\xi-\tfrac{\tau_{spec}}{2})+b) d\xi&=0,\\
\int_{1}(a\wp(\xi-\tfrac{\tau_{spec}}{2})+b) d\xi&=2.
\end{split}
\end{equation}
The last equation reflects the fact that the $2$-lobed Delaunay tori converge to the homogenous torus with $\tau = \sqrt{3} i$ and there the image of $S^1$ for the limit map $\tfrac{2\tau}{\pi i}\chi$ has perimeter $2.$ For closed surfaces the period along $1 \in \Gamma_{spec}$ is necessarily an integer due to the normalization in \eqref{weierstrass_zeta_L}.
Thus it remains constant along the continuous family.  \\

The condition that the map $\chi$ takes values in $\tfrac{\pi i (1+\tau)}{4\tau}d\bar w+\Lambda$ at two disjoint $\xi_i \in C^+$ then determines the conformal type $\tau=\tau(\tau_{spce})$ of the $2-$lobed Delaunay torus as follows: let $\zeta$ be the Weierstrass 
$\zeta$-function defined by $-\zeta'(\xi)=\wp(\xi)$ and $\lim_{\xi\to 0}(\zeta(\xi)-\xi^{-1})=0$, and let $\eta_3:=\zeta(\tfrac{\tau_{spec}}{2})$ and $\eta_1:=\zeta(\tfrac{1}{2})$. 
It is well-known (see \cite{WhWa}) that
\[\eta_1\tfrac{\tau_{spec}}{2}-\eta_3\tfrac{1}{2}=\tfrac{\pi}{2} i.\]
Together with \eqref{2lobe_spec_periods} this yields
\begin{equation}\label{2lobe_spec_periods2}
\begin{split}
a&=-\tfrac{\tau_{spec}}{\pi i}\\
b&=-2\tfrac{\eta_3}{\pi i}.\\
\end{split}
\end{equation}

Since $\Sigma$ is rectangular, the constants $a$ and $b$ are real; moreover,
$a$ is negative. By construction there exists a smallest $s_0\in\R^{>0}$ such that
\begin{equation}\label{definition-s0}
\Re \int_{0}^{s_0+\frac{1}{4}\tau_{spec}}(a\wp(\xi-\tfrac{\tau_{spec}}{2})+b d\xi=1.
\end{equation}
Define
\[h:=\Im\int_{0}^{s_0+\frac{1}{4}\tau_{spec}}(a\wp(\xi-\tfrac{\tau_{spec}}{2})+b) d\xi.\]
Using that the imaginary part of $\int(a\wp(\xi-\tfrac{\tau_{spec}}{2})+b) d\xi$
is monotonic decreasing along the curves $[0;\tfrac{1}{2}]\ni t\mapsto[t+\frac{1}{4}\tau_{spec}]\in\Sigma$
 and $[0;\tfrac{1}{4}]\ni t \mapsto[\frac{1}{2}+t\tau_{spec}]\in\Sigma$ together with 
 \[\Im\int_0^{\tfrac{1}{2}}(a\wp(\xi-\tfrac{\tau_{spec}}{2})+b) d\xi=0\]
we obtain that $h>0.$\\

We define the conformal structure of $T^2=\C/\Gamma$  by the lattice $\Gamma=2\Z+2\tau\Z$ with
\begin{equation}\label{tau_tauSpec}
\tau:=h i.
\end{equation}

Since $a$ and $b$ are real and $\Sigma$ is of rectangular conformal type one can compute
that
\[\int_{0}^{-s_0+\frac{1}{4}\tau_{spec}}(a\wp(\xi-\frac{\tau_{spec}}{2})+b) d\xi=-1+\tfrac{\tau}{2}\in\tfrac{1}{2}\Lambda.\]
Moreover, because the period of $(a\wp(\xi-\frac{\tau_{spec}}{2})+b) d\xi$ with respect to $1 \in \Gamma_{spec}$ is $2$, we obtain that
$s_0<\frac{1}{2},$ and therefore \[[-s_0+\frac{1}{4}\tau_{spec}]\neq[s_0+\frac{1}{4}\tau_{spec}]\in\Sigma.\]

This yields the existence of two distinct points $\xi_1=[s_0+\frac{1}{4}\tau_{spec}]\in C^+$
and $\xi_2=[-s_0+\frac{1}{4}\tau_{spec}]\in C^+$ with the property that 
\[\chi(\xi_{k})\in \frac{\pi i (1+\tau)}{4\tau}d\bar w+\Lambda, \; k\in\{1,2\}.\]
The corresponding flat $\SL(2,\C)$ connections $\nabla^{\lambda_k}$ for $\lambda_1\neq\lambda_2\in S^1$
 are then trivial on $T^2$  so $\tau= hi$ is indeed the conformal type of the immersion. Altogether, we have seen the existence
of a family of CMC tori parametrized by the real parameter $\Im(\tau_{spec}).$ 
Note that by Hitchin's energy formula for harmonic maps (\S 13 in \cite{Hi}),  the area of the CMC torus is
$8 i\tau ab.$
Hence $i\tau b<0$ for all $\tau_{spec}\in i\R^{>0}$ yielding $b>0.$\\

We will need the following two lemmas in Section \ref{32} below.
\begin{Lem}\label{non_degenerate1}
The spectral data $(\Sigma,\chi)$ of a 2-lobed Delaunay torus have the property that 
\[\{\xi\in\bar\Sigma^0 \mid \chi(\xi)=0\in \mathrm{Jac}(T^2)\}=\{[0],[\tfrac{1}{2}]\}\subset \Sigma,\]
where $\bar\Sigma^0=\lambda^{-1}(\{\lambda\in\C\mid\lambda\bar\lambda \leq 1\}.$ 
\end{Lem}
\begin{proof}
We first show the statement for $\tau_{spec}\to0,$
where the corresponding 2-lobed Delaunay tori  converge to the homogeneous torus with $R= 2.$  
For homogeneous tori the spectral curve $\Sigma \cong \C P^1$ is given by $\xi^2= \lambda$ and the spectral datum $\chi$ with $R=2$ has the property that
\[\{\xi\in\C \mid \xi\bar \xi\leq 1 \text{ and } \chi(\xi)=0\in \mathrm{Jac}(T^2)\}= \lambda^{-1}(\{0, 1\}).\]
For the 2-lobed Delaunay tori, we have $\chi ([\frac{1}{2}])= 0 \in \mathrm{Jac}(T^2),$ and 
\[\lambda([\tfrac{1}{2}])=r\to 1\,\, \text{ and } \,\,\, \lambda([\tfrac{1+\tau_{spec}}{2}])=\tfrac{1}{r}\to1
\text{ as }\tau_{spec} \to 0.\] 
Because the family of associated holomorphic structures $\dbar^\lambda$ of the Delaunay tori
converge, as $\tau_{spec} \to 0$, uniformly against the family of  associated holomorphic structures of the homogeneous torus with $R= 2$
on the set $\{\lambda\in\C\mid\lambda\bar\lambda\leq1\}$,
the statement holds in the limit.\\

The full statement follows from continuity of the spectral data within the Whitham deformation together with the fact that for 2-lobed Delaunay tori there is
no other point $\lambda_{new}\in S^1$  appearing on the unit circle such that
 \[\chi(\xi_{new})=0\in \mathrm{Jac}(T^2)\]
 for $\xi_{new}\in\lambda^{-1}(\lambda_{new}):$ 
 
 By Proposition 2.1 in \cite{KSS2} together with the fact that the 2-lobed Delaunay tori are not minimal Theorem 4.5 in \cite{KSS}, the Willmore energy $\int_{T^2} (H^2 + 1) dA$ of the 2-lobed Delaunay tori is monotonically increasing in $\tau_{spec}$ and converges to $8 \pi$ for $\tau_{spec} \to \infty,$ where the surfaces converge to a branched double cover of a geodesic sphere (\cite[Lemma 4.4]{KSS}). If there were
a point $\xi_{new}\in\lambda^{-1}(S^1)$, the monodromy of the corresponding flat $\SL(2,\C)$ connection
would be $-\Id$ along both generators $1$ and $\tau$ of the first fundamental group $\pi_1(T^2).$ At the ramification point $\xi_b=[\tfrac{1}{2}]\in T^2$
the trace of monodromy of the corresponding $\SL(2,\C)$ connection is also $-2$ along both generators
$1,\tau.$ An adaption of the proof of Theorem 6.7 in \cite{FLPP}  gives a lower bound on the Willmore energy depending on the dimension of the space of 
monodromy-free holomorphic sections: On the two fold covering of $T^2$ determined by the lattice $(1-\tau)\Z+(1+\tau)\Z$ the parallel sections with respect to 
$\nabla^{\xi_b}$ and $ \nabla^{\xi_{new}}$ and
the parallel sections with respect to the trivial connections at the two Sym points yield
 a subspace  of the space of quaternionic holomorphic sections with respect to
the induced quaternionic holomorphic structure whose quaternionic dimension is at least 4. The Pl\"ucker formula  in
\cite[Equation (89)]{FLPP} gives the Willmore estimate
\[W(f)\geq\tfrac{\pi}{2}4^2=8\pi\]
which is a contradiction.
\end{proof}

\begin{Lem}\label{non_degenerate2}
Let $(\Sigma,\chi)$ be the spectral data of a 2-lobed Delaunay torus, where $\chi$ is given by \eqref{weierstrass_zeta_L} and \eqref{2lobe_spec_periods2} with $\chi([0])=0.$ Consider the preimage $\xi_1\in C^+\subset\Sigma$
of a Sym  point $\lambda_1\in S^1$ given by $\xi_1=s_0+\frac{1}{4}\tau_{spec}\in \C.$ Then the vectors
\[\frac{\partial \chi }{\partial \tau_{spec}}{\mid _{\xi_1}}\in\overline{H^0(T^2,K)},\, \,\,\,\frac{\partial \chi}{\partial \xi}{\mid_{\xi_1}}\in\overline{H^0(T^2,K)}\]
are $\R$-linear independent, where $\frac{\partial \chi}{\partial \tau_{spec}}{\mid_{ \xi_1}}$ is the
derivative of $\chi$ 
with respect to $\tau_{spec} \in i \R^{>0}$ at $s_0+\tfrac{1}{4}\tau_{spec}$. 
\end{Lem}
\begin{proof}
By  Theorem 5 and Section 3.5 in \cite{LHe} we obtain 
$\tfrac{\del \tau}{\del \tau_{spec}} \neq 0$ along the family of 2-lobed Delaunay tori.
Hence by \eqref{tau_tauSpec}, it is enough to prove that
\[\Re\frac{2\tau}{\pi id\bar w}\frac{\partial \chi}{\partial \xi}{\mid_{\xi_1}}=\Re(a\wp(\xi_1-\tfrac{\tau_{spec}}{2})+b)\neq0\]
for all $\tau_{spec}\in i\R^{>0}$ with $a,b\in\R$ as in \eqref{2lobe_spec_periods2}.
Because  $s_0\in\R^{>0}$ is the smallest number satisfying \eqref{definition-s0} by definition of the Sym point
and because $\int_0^\xi(a\wp(\xi-\tfrac{\tau_{spec}}{2})+b) d\xi$ is imaginary along the imaginary axis, $\Re(a\wp(\xi_1-\tfrac{\tau_{spec}}{2})+b)<0$ is impossible. \\

If \[\Re(a\wp(\xi_1-\tfrac{\tau_{spec}}{2})+b)=0\] then the function \[g: \R \to \R, \quad  g(s) = \int_0^{s + \tfrac{\tau_{spec}}{4}} (a\wp(\xi-\tfrac{\tau_{spec}}{2})+b)d \xi\] would have either a local maximum ($s_0$ as a minimum does not occur, because $s_0$ is the smallest number satisfying \eqref{definition-s0} ) or an inflection point at $s_0.$ \\

If $g(s)$ has an inflection point at $s_0,$ then $g''(s_0) = 0$.  Since $\frac{2\tau}{\pi id\bar w}d\chi = (a\wp(s-\tfrac{\tau_{spec}}{4})+b)=\tilde a\lambda+\tilde b$ for
real $\tilde a,\tilde b\in\R,$ and because $\lambda : \Sigma \to \C P^1$ maps the line $\R+\tfrac{\tau_{spec}}{4}$ onto
the unit circle, the map $\R\ni s\mapsto a\wp(s-\tfrac{\tau_{spec}}{4})+b)$ maps
the real line onto a circle in $\C$ centered on the real axis. Moreover, the circle intersect the $x$-axis exactly
for  $s\in \tfrac{1}{2}\Z$ and these are the only critical points of the real-valued function
$g'(s).$ But we have $0<s_0$ by definition and $s_0 \neq \tfrac{1}{2}$ because the real part of
$\R \ni t\mapsto\frac{2\tau}{\pi id\bar w}\chi(\tfrac{1}{2}+t i)$ is constant and
$\frac{2\tau}{\pi id\bar w}\chi(\tfrac{1}{2})=1$ by \eqref{2lobe_spec_periods}, which leads to a contradiction. Thus for every $\tau_{spec}$ the Sym point satisfies $0<s_0 <\tfrac{1}{2}$ and $s_0$ is not a inflection point of $g.$\\

If $g$ has a local maximum at $s_0,$ then by \eqref{2lobe_spec_periods}
there must be another  $\tilde s_0\in]s_0,\tfrac{1}{2}[$ which satisfies \eqref{definition-s0}. This is not the case as $\tau_{spec}\to0$ or $\tau_{spec}\to\infty$; see \cite{WhWa}.
Further, since all involved functions and the constants $a,b$ depend real analytically on $\tau_{spec},$ we get
a conformal type $\tau_{spec_0}$ for the spectral curve such that for
$\Im(\tau_{spec})<\Im(\tau_{spec_0})$ there is only one point $s_0\in]0,\tfrac{1}{2}[$ satisfying \eqref{definition-s0}
and for $\Im(\tau_{spec_0})+\epsilon>\Im(\tau_{spec})>\Im(\tau_{spec_0})$ (for some small $\epsilon>0$) there are at least two points $s_0\neq\tilde s_0\in]0,\tfrac{1}{2}[$ satisfying
\eqref{definition-s0}. Hence, for $\tau_{spec_0}$ we get that the corresponding $s_0$ is an inflection point of $g_{\tau_{spec_0}}$ and we get also a contradiction in this case. 
\end{proof}

\section{Spectral curve theory}\label{arbitrary_genus}

In the previous section we have shown how CMC tori can be described by a holomorphic map
from the associated spectral curve into the Jacobian of the CMC torus. The main ingredient was the
reduction of (a generic) flat $\SL(2,\C)$ connection over the torus to a flat line bundle connection.
Now we want to explain how to generalize the torus spectral curve theory to the case of  higher genus CMC surfaces with certain discrete symmetries. The theory for the particular case of Lawson symmetric CMC surfaces of genus $2$ was developed in \cite{He3}.
In Section \ref{Abel} we start with a $\lambda$-independent discussion of certain flat  $\SL(2,\C)$ connections in terms of flat line bundles, and in \ref{thespecdata} we consider 
the $\lambda$-dependent version thereof.

\subsection{Abelianization}\label{Abel}
As in Section \ref{CMC_tori} we consider 
the lattice $\Gamma=2\Z+2\tau\Z$ and the corresponding
Riemann surface 
$T^2 =\C/\Gamma$ of genus 1. Let $\sigma \colon T^2 \to \C P^1$ be the elliptic involution and let \begin{equation}\label{z}z\colon T^2\to\CP^1\end{equation}
 be the induced double covering with its four ramification points  by \begin{equation}\label{rampe}P_1=[0],\,P_2=[1],\,P_3=[1+\tau],\,P_4=[\tau]\in T^2.\end{equation}
 By applying a Moebius transformation 
 we assume that the branch points $p_k$ of $z$
are given by
 \begin{equation}\label{brampe}
 p_1=z([0])=0,\, p_2=z([1])=1,\, p_3=z([1+\tau])=\infty,\, p_4=z([\tau])=m\in \CP^1\end{equation}
 for some $m\in\C\setminus\{0,1\}.$
 For $\rho\in]-\tfrac{1}{2},\tfrac{1}{2}[$  consider the moduli space \[\mathcal A^2_\rho(\CP^1\setminus\{p_1,\dots,p_4\})\] of flat $\SL(2,\C)$ connections on the 4-punctured sphere $\CP^1\setminus\{p_1,\dots,p_4\}$ with the additional property that the local monodromies $M_k$  around every puncture $p_k$ are contained in the conjugacy class of
\begin{equation}\label{local_monodromies}
\dvector{\exp{(2\pi i\frac{2\rho+1}{4})} & 0 \\ 0 & \exp{(-2\pi i\frac{2\rho+1}{4})}}.
\end{equation}

We have the following important relation between flat line bundles on the torus and $\mathcal A^2_\rho(\CP^1\setminus\{p_1,\dots,p_4\}):$
\begin{The}[\cite{HeHe}]\label{2:1} Let  $\rho\in]-\tfrac{1}{2},\tfrac{1}{2}[$ and let $\pi^1\colon \mathcal A^1(T^2)\to \mathrm{Jac}(T^2)$ be the projection from the moduli space of flat line bundle connections to the moduli space of holomorphic structures on a trivial line bundle.
 Then  
 there is a 2:1 correspondence $\Pi$
 between the open and dense set  $\mathcal A^1(T^2)\setminus(\pi^1)^{-1}(\Lambda)$ and an open and dense subset of $\mathcal A^2_\rho(\CP^1\setminus\{p_1,\dots,p_4\})$ of flat $\SL(2,\C)$ with local monodromies lying in the conjugacy class given by \eqref{local_monodromies}. \\
 
 This 2:1 correspondence extends to the $\pi^1$-preimage  of $\Lambda\equiv \frac{\pi i}{\tau-\bar\tau}\Z+\frac{\pi i\tau}{\tau-\bar\tau}\Z$ 
 in the following sense: 
 Let $\chi\mapsto\alpha(\chi)$ be a holomorphic map on $U\subset \C\setminus(\frac{\pi i}{\tau-\bar\tau}\Z+\frac{\pi i\tau}{\tau-\bar\tau}\Z)$ and $\gamma\in \bar U\cap \frac{\pi i}{\tau-\bar\tau}\Z+\frac{\pi i\tau}{\tau-\bar\tau}\Z.$ Then the map $\chi \mapsto \Pi ([\nabla^{\chi, \alpha(\chi)}])$ extends to $\chi=\gamma$ 
 if and only if $\alpha$ expands around $\chi=\gamma$ as
\begin{equation}\label{a_spin_expansion}
\alpha(\chi)\sim_\gamma\pm\frac{4\pi i}{\tau-\bar\tau}\frac{\rho}{\chi-\gamma}+\bar\gamma+\,\text{ higher order terms in } \chi.\end{equation}
 \end{The}
 A proof of theorem can be found in \cite{HeHe}, but we shortly recall the construction of $\Pi$ below; see \eqref{Pi}.
  The reader should be aware of the differing normalizations
here and in \cite{HeHe} and \cite{He3}. 
It is also important to mention that
the abelianization is related to (and in fact motivated by) the Hitchin system \cite{Hi1}. These two theories are not equivalent since we use
the complex structure $J$ on the moduli space of flat connections induced by the complex group $\SL(2,\C)$ whereas
the Hitchin system is making use of the complex structure $I$.\\

In order to describe the map $\Pi$ in Theorem \ref{2:1}
we need the lattice $\tfrac{1}{2}\Gamma=\Z+\tau\Z$ and
the (shifted) theta-function $\vartheta$ on $\C/ \tfrac{1}{2}\Gamma,$  i.e.,
the unique (up to a multiplicative constant) entire function $\vartheta\colon\C\to\C$ satisfying $\vartheta(0) = 0$ and
\begin{equation}\label{theta-function}
\vartheta(w+1) = \vartheta (w),\,\,  \vartheta(w+ \tau) = - \vartheta (w)e^{-2\pi i w}\end{equation}
for all $w\in\C.$
For every $x  \notin\tfrac{1}{2}\Gamma$ we define the function 
\begin{equation}\label{beta-function}
\beta_{x}(w) = \frac{\vartheta(w-  x)}{\vartheta(w)}e^{\tfrac{2\pi i }{\bar\tau-\tau} x(w-\bar w)}.\end{equation}
For $x \in \C\setminus\tfrac{1}{2}\Gamma$ the function $\beta_x$ is doubly periodic in $w$ with respect to the lattice  $\tfrac{1}{2}\Gamma$
and satisfies 
\[\left(\dbar-\frac{2\pi i}{\tau-\bar\tau}xd\bar w \right)\beta_{x}=0.\] Thus $\beta_x$ is a meromorphic section of the trivial bundle $\underline\C\to\C/\tfrac{1}{2}\Gamma$ equipped with the holomorphic structure $\dbar-\frac{2\pi i}{\tau-\bar\tau}xd\bar w$  and has a simple zero at $w=x$ and a first order pole  at $w = 0.$  
We can view $\beta_x$ on the bigger torus $T^2 = \C/ \Gamma\to \C/\tfrac{1}{2}\Gamma$  as a meromorphic section with respect to the holomorphic structure $\dbar-\frac{2\pi i}{\tau-\bar\tau}xd\bar w$. This section has four simple  zeros and four simple poles. \\

For a given flat line bundle connection
\begin{equation}\label{linebundle_connection1form}
d^{\chi,\alpha}=d+\alpha dw-\chi d\bar w
\end{equation}
with $\chi\in\C\setminus(\frac{\pi i}{\tau-\bar\tau}\Z+\frac{\pi i\tau}{\tau-\bar\tau}\Z)$,  $x=\frac{\tau-\bar\tau}{2\pi i}\chi,$  
and $\alpha\in\C$ we consider the flat connection $\hat\nabla^{\chi,\alpha}$
 on the trivial rank $2$ bundle
$\C^2\times T^2\to T^2$: 
 \begin{equation}\label{connection1form}
\hat\nabla^{\chi,\alpha}=d+ \dvector{-\chi d\bar w + \alpha dw & \rho\frac{\vartheta'(0)}{\vartheta(-2x) }\beta_{2x}(w) dw \\ \rho\frac{\vartheta'(0)}{\vartheta(2x)}\beta_{-2x}(w) dw & \chi d\bar w -\alpha  d w  }.
 \end{equation}
Note that the off-diagonal of \eqref{connection1form} only depends on the holomorphic structure $\dbar-\chi d\bar w$ and is independent of $\alpha.$\\

In \cite{HeHe} it is shown that  $\hat\nabla^{\chi,\alpha}$ is  gauge equivalent (via a two-valued gauge transformation with singularities at $P_1,\dots,P_4$) to a flat $\SL(2,\C)$ connection  which is given by the pullback of a representative $\tilde\nabla^{\chi,\alpha}$ of an element of $\mathcal A^2_\rho(\CP^1\setminus\{p_1,\dots,p_4\})$. 
The map $\Pi$ is then given by
\begin{equation}\label{Pi}\Pi([d^{\chi,\alpha}])=[\tilde\nabla^{\chi,\alpha}]\in \mathcal A^2_\rho(\CP^1\setminus\{p_1,\dots,p_4\}.\end{equation}
Replacing the line bundle connection
$d^{\chi,\alpha}$
by its dual connection $d^{-\chi,-\alpha}$ in \eqref{connection1form}
yields the same gauge equivalence class in $\mathcal A^2_\rho(\CP^1\setminus\{p_1,\dots,p_4\}),$ i.e., \[\Pi([d^{-\chi,-\alpha}])=\Pi([d^{\chi,\alpha}])=[\tilde\nabla^{\chi,\alpha}].\]

\begin{Rem}\label{Sign}
Note that the ambiguity of the sign of the residue $\pm\frac{4\pi i \rho}{\tau-\bar\tau}$ in \eqref{a_spin_expansion} is meaningful: For positive residue the corresponding underlying parabolic structure is stable (see \cite{HeHe} for more details). If the residue is negative, the underlying parabolic structure is unstable, which implies that
the corresponding $\SL(2,\C)$ connection on the 4-punctured sphere cannot be unitary. It should be also mentioned that
parabolic stability reduces in the case of rational $\rho$ to the stability of a related holomorphic structure on a certain covering of the torus.

\end{Rem}

We use $(\rho,\chi,\alpha)$ to parametrize  flat connections of the form \eqref{connection1form}
with four simple poles at $P_1,\dots,P_4$ on $T^2$. We think of $\chi$ as a point in $\mathrm{Jac}(T^2)$ and consider the pair $(\chi,\alpha)$ as a point $\mathcal A^1(T^2)$ via
\eqref{linebundle_connection1form}.
By the Mehta-Seshadri Theorem \cite{MSe}  (see also \cite{HeHe} for a treatment  of the Mehta-Seshadri theorem in the setup at hand)
there exist for every $\chi\in \C\setminus(\frac{\pi i}{\tau-\bar\tau}\Z+\frac{\pi i\tau }{\tau-\bar\tau}\Z)$
 a unique $\alpha=\alpha^u_\rho(\chi)\in\C$
such that the monodromy representation of the connection 
given by
\eqref{connection1form} with $(\rho,\chi,\alpha^u_\rho(\chi))$ is unitarizable.
In fact,  for every $\rho\in]-\tfrac{1}{2},\tfrac{1}{2}]$
this map $\chi\mapsto \alpha^u_\rho(\chi)$ induces
a real analytic section of the affine holomorphic bundle $\mathcal A^{1}(T^2)\to \mathrm{Jac}(T^2)$ which we  denote by 

\begin{equation}\label{alphaurho}
\alpha^{MS}_\rho\in\Gamma(\mathrm{Jac}(T^2)\setminus\{0\},\mathcal A^1 (T^2))
\end{equation}
given by
$\alpha^{MS}_\rho([\dbar-\chi d\bar w])=[d+\alpha^u_\rho(\chi)dw-\chi d\bar w].$
Note that $\alpha^u_\rho$ satisfies the following functional properties:
\begin{equation}\label{lift_alpha}
\begin{split}
\alpha^u_\rho(\chi+\frac{\pi i}{\tau-\bar\tau}\tau)&=\alpha^u_\rho(\chi)+\frac{\pi i}{\tau-\bar\tau}\bar\tau\\
\alpha^u_\rho(\chi+\frac{\pi i}{\tau-\bar\tau})&=\alpha^u_\rho(\chi)+\frac{\pi i}{\tau-\bar\tau}
\end{split}
\end{equation}
for all $\rho\in]\tfrac{-1}{2},\tfrac{1}{2}[$ and $\chi\in\C\setminus\hat\Lambda.$
\begin{Lem}\label{a^u_symmetries}
Let $T^2=\C/(2\Z+2\tau \Z)$ and $\rho\in]-\tfrac{1}{2},\tfrac{1}{2}[.$ Then the section $\alpha^{MS}_\rho$ in \eqref{alphaurho}
is odd with respect to the involution on $\mathrm{Jac}(T^2)$ induced by mapping a holomorphic structure to its dual structure.
If $\tau\in i\R$ then $\alpha^{MS}_\rho$ is real in the  sense that
\[\alpha^u_\rho(\bar\chi)=\overline{\alpha^u_\rho(\chi)}\]
for all $\chi\in\C\setminus(\frac{\pi i}{\tau-\bar\tau}\Z+\frac{\pi i\tau }{\tau-\bar\tau}\Z).$ 
 \end{Lem}
\begin{proof}
Because the dual line bundle connection induces the same $\SL(2,\C)$ connection on the punctured torus,  the
section $\alpha^{MS}_\rho$ is odd for all $\rho\in {]\tfrac{-1}{2},\tfrac{1}{2}[}.$ The second assertion follows from the fact that
the complex conjugated Riemann surface of a rectangular torus is isomorphic to the torus itself.
\end{proof}
\begin{Lem}\label{real_analytic}
Away from the origin $0\in \mathrm{Jac}(T^2)$ the map
\[\alpha^{MS}\colon]\tfrac{-1}{2},\tfrac{1}{2}[\times \mathrm{Jac}(T^2)\setminus\{0\}\to \mathcal A^1(T^2)\]
is real analytic.
\end{Lem}
\begin{proof}
As indicated in (13) and Remark 7 in \cite{He3}, a connection of the form \eqref{connection1form} is the pullback of a connection on a one-punctured torus.
A necessary (respectively sufficient) condition that such a connection is unitary with respect to a suitable hermitian metric is 
that the traces of two independent global monodromies lie the interval $[-2,2] \subset \R$ (respectively in $]-2,2[ \subset \R$).
The statement thus follows from the fact that
the monodromy depends analytically on $(\rho,\chi,\alpha).$ 
\end{proof}

\begin{Rem}\label{Rem:Z3_versus_Z2}
It was shown in \cite{He3} that for $\rho=\tfrac{1}{6}$ the moduli space $\mathcal A^2_\rho(\CP^1\setminus\{p_1,\dots,p_4\})$ is equivalent to the moduli space of
Lawson symmetric flat $\SL(2,\C)$ connections on a Lawson symmetric  Riemann surface of genus $2.$ A similar interpretation is true  for certain discrete values of $\rho;$
 see also Section \ref{Rational_weights} below. By changing $\rho$ we are therefore  continuously deforming (certain subspaces of) the moduli space of flat connections on varying Riemann surfaces of different genera.
\end{Rem}

\subsection{The spectral data}\label{thespecdata}

The spectral curve theory (as  described below)  unifies the spectral curve theory for CMC tori
(see Section \ref{CMC_tori}) and the spectral curve theory for Lawson symmetric CMC surfaces of genus $2$
as developed in \cite{He3}. The main idea is to parametrize $\lambda$-families of (gauge equivalence classes of)
$\SL(2,\C)$ connections in terms of families of (gauge equivalence classes of) line bundle connections. A slight drawback
of this approach is that the line bundle connections can only be parametrized in terms of a spectral curve, i.e., a double covering $\lambda\colon \Sigma\to \C$ of the spectral plane. In fact, in order to obtain conformally parametrized CMC surfaces,
the spectral curve branches over $\lambda=0$ and possibly other discrete spectral values $\lambda_1,\lambda_2,\dots,\in\C.$ 
The spectral curve $\Sigma$ is equipped with an odd holomorphic map $\chi\colon\Sigma\to \mathrm{Jac}(T^2)$  (with respect to the hyperelliptic involution) and with a corresponding
 odd meromorphic map $\alpha$ such that $$d^{\alpha, \chi}  = d + \alpha dw - \chi d\bar w $$ is a lift of $\chi$
into the affine moduli space
$\mathcal A^1(T^2)$ of flat line bundle connections on $T^2.$ 
Note that the maps $\chi$ and $d^{\alpha, \chi}$ must be odd because the corresponding  family of flat $\SL(2,\C)$ connections is well-defined in terms of $\lambda.$ Moreover, $d^{\alpha, \chi}$ must have a first order pole in the affine holomorphic bundle $\mathcal A^1(T^2)\to \mathrm{Jac}(T^2)$ at $\lambda^{-1}(0)$ as a consequence of \eqref{associated_family}.
The spectral data $(\Sigma,\chi,d^{\alpha, \chi})$ of a family $\lambda\mapsto[\nabla^\lambda]$   give rise to a
commuting diagram
\vspace{0.25cm}

\begin{equation}\label{diagramm}
 \xymatrix{
        &   &       \mathcal A^1(T^2) \ar[d]^{''} \ar[ddr]^{\Pi} \\
                \Sigma \ar[rru]^{{d^{\alpha, \chi}}} \ar[rr]_{\chi} \ar[d]^\lambda      &   &  \mathrm{Jac}(T^2)\\
                \C \ar[rrr]^{[ \nabla^{\lambda} ]}  & & &\mathcal A^2
  }
 \end{equation}
 
 where $\Pi$
 is the 2:1 covering map discussed in Theorem \ref{2:1}, and 
 $\mathcal A^2=\mathcal A^2_\rho(\CP^1\setminus\{p_1,\dots,p_4\}).$
 \begin{Rem}
Note that the spectral curves for higher genus CMC surfaces (as discussed in Section \ref{thespecdata} below) are not compact as in the case of CMC tori. Nevertheless, we use the same notations for the spectral curve and the spectral data as in the torus case.
\end{Rem}
In order to define a CMC surface
the spectral data must further satisfy the conditions stated in the following theorem:
\begin{The}\label{slitting_tori}
Let $\rho\in]0;\tfrac{1}{2}[$ and
let $\lambda\colon\Sigma\to D_{1+\epsilon}\subset\C$ be a double covering of the $(1+\epsilon)$ disc branched over finitely many points
$\lambda_0=0,\lambda_3,\dots,\lambda_k.$ Further, let $\chi\colon\Sigma\to \mathrm{Jac}(T^2)$ be an odd map (with respect to the involution $\sigma$ on $\Sigma$ induced by $\lambda$) and $d^{\alpha, \chi}$ be an odd meromorphic lift of $\chi$ to $\mathcal A^1(T^2)$  satisfying the following conditions:
\begin{enumerate}
\item $\chi(\lambda^{-1}(0))=0\in \mathrm{Jac}(T^2);$
\item $\mathcal D$ has a first order pole at $\lambda^{-1}(0);$
\item $\mathcal D$ has a first order pole satisfying the condition induced by \eqref{a_spin_expansion}  at every $\xi_s\in\Sigma\setminus\lambda^{-1}(\{0\})$ with $\chi(\xi_s)=0\in \mathrm{Jac}(T^2)$ and no other singularities;
\item for all $\xi \in \lambda^{-1}(S^1)$: $d^{\alpha, \chi}(\xi)=\alpha^{MS}_\rho(\chi(\xi));$ 
\item  there are four distinct points
$\xi_1,\sigma(\xi_1),\xi_2,\sigma(\xi_2)\in\lambda^{-1}(S^1)\subset\Sigma$
with
$\chi(\xi_{k})\in\frac{\pi i (1+\tau)}{2\tau-2\bar\tau}+\Lambda$ for $k\in\{1,2\}$,
where $\Lambda=\frac{\pi i }{\tau-\bar\tau}\Z+\frac{\pi i \tau}{\tau-\bar\tau}\Z$ is the lattice generating
 $\mathrm{Jac}(T^2).$ 
\end{enumerate}
Then there exists a
conformal CMC immersion
\[f\colon T^2\setminus (l_1\cup l_2)\to S^3\]
whose associated family of flat connections give rise to the spectral data
$(\Sigma,\chi,\mathcal D),$ where
 $l_1=\{[t\tau]\mid  t\in[0,1]\}$ and $l_2=\{[t\tau+1]\mid  t\in[0,1]\}.$
\end{The}

\begin{proof}
For all $\rho\in]0,\tfrac{1}{2}[$, $\chi\in \mathrm{Jac}(T^2)\setminus\{ \text{Spin bundles} \}$ and  $\alpha\in\C$ the connection $\hat\nabla^{\chi,\alpha}$
given by \eqref{connection1form}   is irreducible. 
By condition (3)  the family of (gauge equivalence classes of)
 $\SL(2,\C)$ connections has no singularities on $D_{1+\epsilon}$ except at $\lambda=0.$
By the proofs of
 Theorem 6 and Theorem 8 in \cite{He3} applied to our situation  we obtain a family of flat $\SL(2,\C)$ connections of the form 
 \[\nabla^\lambda=\lambda^{-1}\Phi+\nabla-\lambda\Phi^*\] on 
 $T^2\setminus \{P_1,\dots,P_4\}$
 (where  the $P_k$ are as in \eqref{rampe})
with the property that $\nabla^{\lambda}$ is unitary with respect
 to a fixed hermitian metric for $\lambda\in S^1$ and such that  $\Phi$ is a complex linear, nilpotent and 
 nowhere vanishing 1-form on  $T^2\setminus \{P_1,\dots,P_4\}.$   \\
 
Because of Remark \ref{gensym} it remains to prove that the restrictions of both connections $\nabla^{\lambda_k},$ $k=1,2$ 
to $T^2\setminus (l_1\cup l_2)\subset T^2\setminus \{P_1,\dots,P_4\}$  
are gauge equivalent to 
the same flat connection which has monodromy taking values in $\{\pm \Id\}.$
 Note that  for any $\chi_0\in\frac{\pi i (1+\tau)}{2\tau-2\bar\tau}+\Lambda$
the (unitarizable) connection $\hat\nabla^{\chi_0,\alpha^u_\rho(\chi_0)}$ 
yields the same gauge equivalence class due to the periodicity properties; compare with \eqref{lift_alpha}.\\

We proceed as follows: We start with a reducible connection on the 4-punctured sphere, 
whose monodromy is easily computable, and show that its pullback is gauge equivalent (by a two-valued gauge transformation, i.e., the gauge is well-defined up to sign) to $\hat\nabla^{\chi_0,\alpha^u_\rho(\chi_0)}$.
Note that the monodromy representations of  two flat connections which are gauge equivalent by a two-valued gauge transformation differ by a representation
with values in $\{\pm\Id\}.$\\

We have identified $T^2=\C/(2\Z+2\tau\Z)$ and we make use of the elliptic function $z\colon T^2\to\CP^1$ as in \eqref{z}.
For $\epsilon=\tfrac{2\rho+1}{4}$ consider the Fuchsian system 
\begin{equation}\label{almost-trivial}d+\epsilon\dvector{1&0\\0 &-1}\frac{dz}{z}+\epsilon\dvector{-1&0\\0 &1}\frac{dz}{z-1}+\epsilon\dvector{-1&0\\0 &1}\frac{dz}{z-m}\end{equation} on the 4-punctured sphere $\CP^1\setminus\{0,1,\infty,m\}.$
The pullback  of this Fuchsian system via the degree 2 map $z$  is a
 flat meromorphic connection $\tilde\nabla$
which has first order poles at the ramification points $P_k$ of $z$. The local monodromies are determined by the eigenvalues  $\pm(\rho+\tfrac{1}{2})$ of the residues of $\tilde\nabla$. We claim that the 
monodromies of  $\tilde\nabla$ along the curves $\gamma_1^\pm: \; t\in[0,2]\mapsto t\pm \tfrac{\tau}{2}$ and $\gamma_2^\pm: \; t\in[0,2]\mapsto t \tau\pm \tfrac{1}{2}$ are all the identity $\Id,$ which then implies that
  the monodromy representation of $\tilde\nabla$ is trivial when  we restrict it  to $T^2\setminus (l_1\cup l_2).$
But the monodromy representation of the (diagonal) Fuchsian system on the 4-punctured sphere can be computed
via the residue theorem, and
 it can be easily related to the monodromy representation of $\tilde\nabla$
on the 4-punctured torus.
For example, the closed planar curves $z(\gamma_1^\pm)$ have
winding number $-1$ around $p_1=0$ and $p_2=1$, and winding number 0 around $p_4=m.$
Therefore,  the monodromy of $\tilde\nabla$ is trivial  along $\gamma_1^\pm.$\\

 \begin{figure}
\centering
\includegraphics[width=0.75\textwidth]{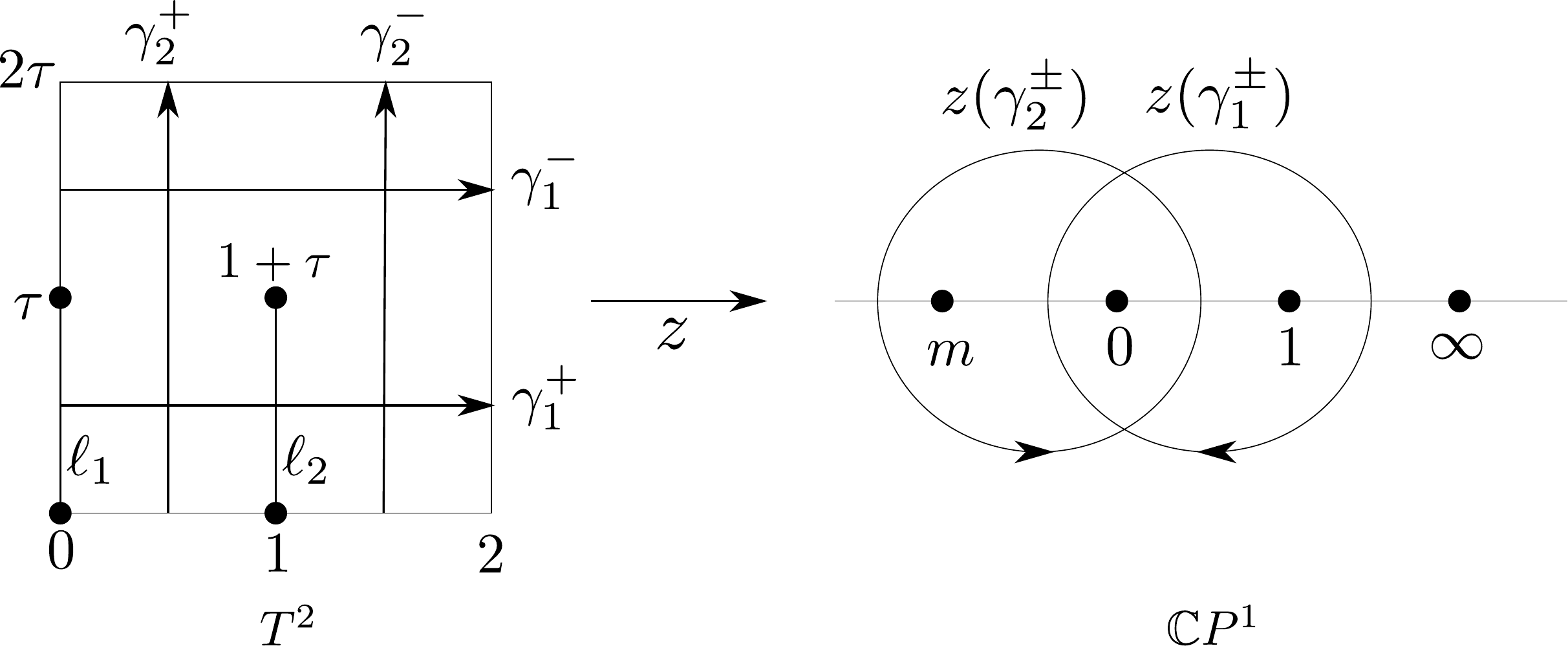}
\caption{
\footnotesize
Four paths  $\gamma^\pm_k$ $k=1,2$ in the cut torus $T^2\setminus (l_1\cup l_2)$ and their images under $z$ in $\CP^1$.
These are used in the proof of Theorem \ref{slitting_tori}.
}
\label{fig:lawsonpath}
\end{figure}

To construct the (two-valued) gauge transformation we consider the following meromorphic 1-form
\[\phi=\dvector{0& 1\\0 & 0}\frac{dz}{z}+\dvector{0& 0\\1 & 0}\frac{dz}{z-1}+\dvector{0& 0\\-1 & 0}\frac{dz}{z-m}\]
and its eigenline bundles $E^\pm$ which are well-defined on $T^2.$ 
The standard computation of eigenlines shows that
each eigenline bundle  corresponds to the divisor
\[D=-[0]-[1+\tau]\equiv -[1]-[\tau]\]
of degree $-2$, and that they are realized as holomorphic line subbundles $E^\pm$ of the trivial $\C^2$-bundle over $T^2$ as the image of the map
\[\dvector{\pm \hat a \\ \hat b}\colon L(-[0]-[1+\tau])\to \C^2,\]
where $\hat a,\hat b\in H^0(T^2, L(-[0]-[1+\tau])^*)$ are determined up to a multiplicative constant 
by their divisors
\[(\hat a)=[1]+[\tau]\;\; \text{ and } \;\;(\hat b)=[0]+[1+\tau].\]
Note that the holomorphic line bundles $L(-[0]-[1+\tau])$ and $L(-2[0])$ differ by the spin bundle
$L([0]-[1+\tau])$ which corresponds to the holomorphic structure
\[\dbar^0-\tfrac{\pi i(1+\tau)}{2\tau-2\bar\tau} d\bar w.\]
Consider the meromorphic frame $s=s_{-2[0]}$ (with double pole at $[0]\in \C/(2\Z+2\tau\Z)=T^2$) of the line bundle $L(-2[0]).$
With respect to the frame $(s\oplus 0,0\oplus s)$ of $L(-2[0])\oplus L(-2[0])$, the gauge which diagonalizes $\phi$ is given 
by
\[g=\dvector{ a&-a\\b&b}\]
where 
\begin{equation}\label{aandb}
a(w)=c \beta_{[\tfrac{1}{2}]}(\tfrac{1}{2}w) \beta_{[\tfrac{\tau}{2}]}(\tfrac{1}{2}w)
\text{ and } b(w)= \beta_{[\tfrac{1+\tau}{2}]}(\tfrac{1}{2}w)\end{equation}  for some constant $c\in\C^*$ and $\beta_x$ as in \eqref{beta-function}.
Write
\[\tilde \nabla=d+\dvector{f&0\\0&-f}dw\]
for the unique odd meromorphic function $f\colon\C/(2\Z+2\tau\Z)\to\C$ with simple poles at $[0]$, $[1]$, $[1+\tau]$ and $[\tau]$
and residues given by $\rho+\tfrac{1}{2},-(\rho+\tfrac{1}{2}),\rho+\tfrac{1}{2},$ and $-(\rho+\tfrac{1}{2})$, respectively.
Then we obtain
\begin{equation}\label{nab1}
\tilde\nabla.g=d+\dvector{0&-f\\-f&0}+\tfrac{1}{2}\dvector{\tfrac{da}{a}+\tfrac{db}{b}& -\tfrac{da}{a}+\tfrac{db}{b}\\ -\tfrac{da}{a}+\tfrac{db}{b}& \tfrac{da}{a}+\tfrac{db}{b}}.\end{equation}
Consider the unique flat meromorphic line bundle connection $\nabla^S$ on $L(-2[0])$ satisfying
\[(\nabla^S\otimes\nabla^S) s_{-[0]-[1]-[1+\tau]-[\tau]}=0\]
for the meromorphic section $s_{-[0]-[1]-[1+\tau]-[\tau]}$ of the bundle $L(-2[0])\otimes L(-2[0])$ with simple poles at $[0]$, $[1]$, $[1+\tau]$ and $[\tau]$.
It satisfies \[\nabla^S s_{-2[0]}=\tfrac{1}{2} \tfrac{d\wp'}{\wp'} s_{-2[0]},\]
where $\wp'$ is the  derivative of the Weierstrass $\wp$-function of $T^2$, which is an odd meromorphic function with zeros at $[1]$, $[1+\tau]$ and $[\tau]$ and a pole of order 3 at $[0].$
Using the $\vartheta$-function in \eqref{theta-function} and denoting
\[\vartheta_0(w)=\vartheta(\tfrac{1}{2}w), \vartheta_1(w)=\vartheta(\tfrac{1}{2}(w-[1])), \vartheta_2(w)=\vartheta(\tfrac{1}{2}(w-[1+\tau])) \text{ and } \vartheta_3(w)=\vartheta(\tfrac{1}{2}(w-[\tau])) \]
we get the identity
\[\frac{d\wp'(w)}{\wp'(w)}=-3\frac{d \vartheta_0}{\vartheta_0}+\frac{d \vartheta_1}{\vartheta_1}+\frac{d \vartheta_2}{\vartheta_2}+\frac{d \vartheta_3}{\vartheta_3}+constant \;dw\]
where the constant-term can be computed explicitly.
Therefore, 
\[\hat\nabla:=(\nabla^S)^*\otimes \tilde\nabla.g\] is a  rank 2 connection on the trivial $\C^2$-bundle
whose diagonal part is smooth. Note that the global monodromies (along $\gamma_1^\pm$ and $\gamma_2^\pm$) of $\nabla^S$  are $-\Id$, and its local monodromies  are  $-\Id$ as well.
 The
$(0,1)$-part of the upper left entry and of the lower right entry of the corresponding connection 1-form (with respect to the standard frame of $\C^2$)
are both given by \[\frac{\pi i (1+\tau)}{2\tau-2\bar\tau}d\bar w\]
as can be computed by looking at \eqref{nab1} using \eqref{aandb} and \eqref{beta-function}. 
Finally, we gauge $\hat\nabla$ by the diagonal gauge transformation
\[h=\dvector{i\exp(\frac{2\pi i}{\tau-\bar\tau}(1+\tau)(w-\bar w))&0\\0&-i}.\]
It remains to prove that for $\chi_0=\frac{\pi i (1+\tau)}{2\tau-2\bar\tau}$ the connections $\hat\nabla^{\chi_0,\alpha^u_\rho(\chi_0)}$   and  $\hat\nabla.h$ are equal. 
By the definition of $h$ it follows that the $(0,1)$-parts of the diagonals  (of the connection 1-forms) of  $\hat\nabla^{\chi_0,\alpha^u_\rho(\chi_0)}$  and  $\hat\nabla.h$ (with respect to the standard frame)
coincide. The off-diagonal parts of \eqref{nab1} and of \eqref{connection1form} coincide as well, as one can deduce 
by using standard identities of elliptic functions and \eqref{aandb}.
Hence, both flat connections give rise to the same parabolic structure (see  \cite{MSe} for more details on parabolic structures). On the other hand, both connections 
$\hat\nabla^{\chi_0,\alpha^u_\rho(\chi_0)}$  and  $\hat\nabla.h$ are unitarizable. Therefore they must coincide
by the uniqueness part of the Mehta-Seshadri theorem \cite{MSe}.  
Alternatively, that the $(1,0)$-parts of the diagonal parts of $\hat\nabla^{\chi_0,\alpha^u_\rho(\chi_0)}$  and  
$\hat\nabla.h$ coincide can be computed directly.
 This finishes the proof of
the theorem.
 \end{proof}

 \begin{Rem}
Instead of cutting the torus $T^2$ along the curves  $l_1$ and $l_2$, cutting the torus along $\tilde l_1=\{[t]\mid  t\in[0;1]\}$ and 
 $\tilde l_2=\{[\tau+t]\mid  t\in[0;1]\}$ yields a different CMC immersion.  However the analytic continuations of both CMC surfaces (after a suitable isometry of the 3-sphere) are the same.
 Moreover, for rational weights, this analytic continuation yields a closed (possibly branched) CMC surface; see Section
 \ref{Rational_weights} below.
 \end{Rem}
 \begin{Rem}\label{bound}
 If the torus $T^2$ has rectangular conformal structure, i.e., $\tau\in i\R,$ and if the Hopf differential is a real multiple of $(dz)^2$, then the four boundary curves of the immersion $f\colon T^2\setminus (l_1\cup l_2)\to S^3$ are curvature lines. The boundary curves on the surface belonging to the same $l_i$ in the domain intersect at an angle of $4\pi \rho.$ For rational $\rho$ this follows from Section
 \ref{Rational_weights}; see also Figure \ref{fig:lawson}.
 \end{Rem}
 
\subsection{Rational weights}\label{Rational_weights}

Let $(\Sigma, \chi, \mathcal D)$ be spectral data satisfying the conditions of Theorem \ref{slitting_tori} for a
 rational $\rho\in ]0, \tfrac{1}{2}[.$ 
 The purpose of this section is to show that the analytic continuation of the surface $f$ given by Theorem \ref{slitting_tori} gives a closed (possibly branched) CMC surface in $S^3$.
  The spectral data give a family of flat connections $\lambda\mapsto\hat\nabla^\lambda$ on the 4-punctured torus  with singularities
determined by $\rho$ around the punctures. 
Pointwise in $\lambda$, the connections $\hat\nabla^\lambda$ give rise to meromorphic connections (with eigenvalues of the residues given by $\pm\tfrac{2\rho+1}{4}$) on the 4-punctured sphere by Theorem \ref{2:1} (see Theorem 3.4 and Theorem 3.5 in \cite{HeHe} for details). \\

Further, the proofs of Theorem 1 and Theorem 2 in \cite{HeHe} (together the the condition that $\chi$ is odd)
show that for small open subsets of $\C^*$ the correspondence works in families;
i.e., for every $\lambda_0\in\C^*$ there is an open neighborhood $U\subset\C^*$ of $\lambda_0$ and a family
of meromorphic connections
$\lambda\in U\mapsto D^{U,\lambda}$ such that $D^{U,\lambda}$ and $\hat\nabla^\lambda$
are related by the 2:1 correspondence of Theorem \ref{2:1} (see also the proof of Lemma \ref{atlambda=0} below). We need the following lemma in order to analyze the situation at $\lambda=0.$
\begin{Lem}\label{atlambda=0}
 There is  an open neighborhood $U$ of $\lambda=0$ and a family $\lambda\in U\setminus\{0\}\mapsto D^\lambda$ of
meromorphic connections  on the 4-punctured sphere with first order pole at $\lambda=0$ (and nilpotent residue $\Psi$)
such that the 2:1 correspondence of Theorem \ref{2:1} relates  $\hat\nabla^\lambda$ and $D^\lambda$ for all $\lambda\in U\setminus\{0\}.$
\end{Lem}
\begin{proof}
We work with Fuchsian systems of the form
\[\nabla^{u,s}=\nabla^u+s\Psi^u\]
where
$\nabla^u$ and $\Psi(u)$ are given by equations (2.3) respectively (2.5) in \cite{HeHe}.
Let
\[\frac{2\rho+1}{4}=\frac{p}{q}\] for coprime integers $p$ and $q$
and set
\begin{equation}\label{concrete_parabolic_structure}
\begin{split}
A_1&=\frac{p}{q}\dvector{1 &0\\ 2 &-1},\, \,\,\,A_2=\frac{p}{q}\dvector{-1  &0 \\-2 & 1},\, \,\,\,
A_m=\frac{p}{q}\dvector{-1 &2 u \\ 0 & 1} ,\, \,\,\,
\end{split}
\end{equation}
and
\begin{equation}\label{concrete_Higgs}
\begin{split}
\Psi_1=\dvector{u&-u \\ u & -u},\, \,\,\,\, \Psi_2&=\dvector{0 &0 \\ 1-u & 0}, \, \,\,\,\, \Psi_m=\dvector{-u &u^2 \\ -1 & u}.
\end{split}
\end{equation}
We define the Fuchsian system
\begin{equation}\label{Fuchsu}
\nabla^u:=d+A_1\frac{dz}{z-1}+A_2\frac{dz}{z}+A_m\frac{dz}{z-m}\end{equation}
and the so-called parabolic Higgs field 
\begin{equation}\label{psi^u}
\Psi=\Psi(u):=\Psi_1\frac{dz}{z-1}+\Psi_2\frac{dz}{z}+\Psi_m\frac{dz}{z-m}.
\end{equation}

The 2:1 correspondence of Theorem \ref{2:1} can be expressed in terms of the connections \eqref{connection1form} and the Fuchsian systems $\nabla^{u,s}$
explicitly (on open dense subsets of the moduli spaces):
because of (3.2) in \cite{HeHe}
$u(\chi)$ is given in terms of $\chi$ as the unique meromorphic function
$u(\chi)$  of degree $2$ on $\mathrm{Jac}(T^2)$ which branches at the half lattice points with branch values $0,1,\infty,m$
(see Equation (3.5) in \cite{HeHe}). In particular, composing $u(\chi)$ with the map $\chi$  on the spectral curve (see Diagram \ref{diagramm}), we obtain
a well-defined   meromorphic function $u=u(\lambda)$ on the spectral plane $\C$ with $u(0)=m$ since the map $\chi$ is odd. Similarly, we obtain a well-defined holomorphic function
$s(\lambda)$ on $U\setminus\{0\}$ (after shrinking $U$ if necessary).
Since 
$\Psi(m)$ is nilpotent, it remains to show that the function $s=s(\lambda)$ determined by the 2:1 correspondence of Theorem \ref{2:1}
has  a first order pole at $\lambda=0.$ This can be done by making use of the computations in the proof of Theorem 2 in \cite{HeHe}: 
for an appropriate odd meromorphic function $a\colon \lambda^{-1}(U)\to \C$ with simple pole at $\lambda^{-1}(0)$ the family of connections  
determined by
\[\xi\in\lambda^{-1}( U)\mapsto \Pi(\hat\nabla^{\chi(\xi),\alpha(\xi)+a(\xi)})\]
via the 2:1 correspondence in Theorem \ref{2:1}
extends to $\lambda^{-1}(0)$, giving a Fuchsian system whose underlying parabolic structure is unstable since $4\tfrac{2\rho+1}{4}>0$ (see $\S$ 2.4 in \cite{HeHe}). Because
adding a holomorphic $u$-dependent family of parabolic Higgs fields $\tilde s(u)\Psi^u$ changes the first and higher order term in the expansion (3.7) in \cite{HeHe} (see the end of the proof of Theorem 3.5), a change of the order -1 term via \[\alpha(\xi)+a(\xi)\rightsquigarrow \alpha(\xi)\]
yields a first order pole of the function $s(\lambda)$ at $\lambda=0$. Hence, the family \[\xi\in\lambda^{-1}( U)\mapsto \hat\nabla^{\chi(\xi),\alpha(\xi)}\]
determines the required family of meromorphic connections
\begin{equation}\label{dlambda}\lambda\mapsto D^\lambda=\nabla^{u(\lambda)}+s(\lambda)\Psi(u(\lambda)).\end{equation}
\end{proof}
\begin{Rem}\label{2:1explicitRem}
In the above proof and in the proof of  Theorem \ref{slitting_tori}
we have respectively obtained a direct relation between connections of the form
 \eqref{connection1form} on the 4-punctured torus and Fuchsian systems on the 4-punctured sphere.
Both cases 
rely on the abelianization procedure developed in $\S3$ of \cite{HeHe}.
\end{Rem}

In order to obtain compact (possibly branched) CMC surfaces, we proceed as follows: we pull back the connections $D^\lambda$ and respectively $D^{U,\lambda}$ to a certain covering $M$ of the 4-punctured sphere,
such that the pullbacks of the connections are gauge equivalent to smooth flat connections. 
As these smooth connections ($D^\lambda$,  $D^{U,\lambda}$ for the various sets $U\subset\C^*$) on $M$ are still gauge-equivalent on the intersections of the corresponding sets, they give rise to a map 
$\hat{\mathcal D}\colon\C^*\to\mathcal A^2(M)$ as in Theorem \ref{lifting_theorem}. It remains to prove that $\hat{\mathcal D}$
satisfies conditions (1), (2) and (3) of Theorem \ref{lifting_theorem}. Condition (1) is satisfied as one can see either by looking at the monodromy representation
of the connections which are pulled back to $M$ (which must be simultaneously unitarizable if the monodromy of the corresponding connection on the 4-punctured sphere is simultaneously unitarizable), or by pulling back and gauging the hermitian metric which is parallel with respect to the corresponding connection on the 4-punctured sphere.
To treat  condition (2) we need the following lemma:

\begin{Lem}\label{atlambda=0again}
 Let $z_0$ be one of the four branch points of $T^2\to\CP^1$. There exists a local holomorphic coordinate $x$ centered at $z_0$ such that with respect to an appropriate (locally defined) holomorphic frame
 the family $D^\lambda$ of connections expands at $\lambda=0$  as
 \begin{equation}\label{DPW_sphere}
D^\lambda\sim d+\dvector{\tfrac{p}{q} &  \lambda^{-1} b(\lambda)  \\ a(\lambda) x & -\tfrac{p}{q}}\tfrac{dx}{x}+\lambda^{-1}\eta+\omega^\lambda,
\end{equation}
where $a(\lambda)$ and $b(\lambda)$ are holomorphic functions in $\lambda$, 
$\eta$ is a holomorphic 1-form whose lower left entry vanishes at $x=0$,
and $\omega^\lambda$ is a holomorphic 1-form in a neighborhood of $x=0$
 which depends holomorphically on $\lambda$ and whose lower left entry vanishes at $x=0.$
 \end{Lem}
\begin{proof}
We use the frame
 which diagonalizes the residue at $z_0$ of $\nabla^m$  given by \eqref{Fuchsu}  for $u=m$ whose upper left entry is the positive eigenvalue $\tfrac{p}{q}.$ Note that this frame
 has the property that the residue of $\Psi(m)$ is a nilpotent $\mathfrak{sl}(2,\C)$ matrix whose only non-zero entry is in upper right.
The  expansion \eqref{DPW_sphere}  is then a consequence of \eqref{dlambda} where $u(0)=m$ and $s$ has a first order pole at $\lambda=0$.
\end{proof}

To treat conditions (2)  and (3) of Theorem \ref{lifting_theorem} in detail we need to specify the covering $M:$
The eigenvalues
of the residues of the meromorphic connections on the 4-punctured sphere
  are $\frac{2\rho+1}{4}=\tfrac{p}{q}\in \Q$ with $p$ and $q$ coprime and $\lambda$-independent.
There are two cases depending on whether 
 $q$ is odd or even. \\
 
 We first consider the case when $q$ is odd, and take $\pi\colon M\to\CP^1$ to be the $q$-fold covering of
 $\CP^1$ totally branched over  $0,$ $1$, $m$ and $\infty$
determined by the compact algebraic curve  of genus $q-1$ corresponding to the affine algebraic equation equation
 \begin{equation}\label{algeq}Y^q=  \frac{Z}{(Z-1)(Z-m)}.\end{equation}
 Note that for the standard coordinate $z$ on $\C\subset\CP^1$ we have $z\circ \pi=Z$ on $M.$
 We will see in the proof of Theorem \ref{branched_CMC} that this equation determines the compact Riemann surface on which the analytic continuation of the CMC surface $f$ (given by Theorem \ref{slitting_tori}) closes in an analogous way to 
 the Riemann extension for germs of holomorphic functions.\\
 
 Using a coordinate $y$ on $M$ satisfying $y^q=x$ (where $x$ is as in Lemma \ref{atlambda=0again}  above) and the expansion \eqref{DPW_sphere},
 the local behavior of  the pullback of $D^\lambda$ is given by 
\begin{equation}\label{DPW_upstairs}
\pi^*D^\lambda\sim d+\dvector{p &  q \lambda^{-1}b(\lambda) \\ q a(\lambda) y^{q} & -p}\tfrac{dy}{y}+\pi^*\eta \lambda^{-1}+\pi^*\omega^\lambda.
\end{equation}
Gauging \eqref{DPW_upstairs} by
 \[g(y)=\dvector{y^{-p} & 0 \\ 0 &y^p}\] yields
 \begin{equation}\label{DPW_upupstairs}
 \begin{split}
 \pi^*D^\lambda.g(y)\sim\; d&+\dvector{0 & q \lambda^{-1} y^{2p-1}\\ q a(0) y^{q-2p-1} &0}dy\\&+g^{-1}(y)\pi^*\eta g(y)\lambda^{-1} +g^{-1}(y)\pi^*\omega^\lambda g(y),
 \end{split}\end{equation}
 where $g^{-1}(y)\pi^*\eta g(y)$  and $g^{-1}(y)\pi^*\omega^\lambda g(y)$ are holomorphic in $y$ because the lower left entries of $\eta$ and $\omega^\lambda$ vanish at $x=0$ due to
 the previous Lemma and because \[2q-1-2p>0\] as $\tfrac{1}{2}>\tfrac{p}{q}$. Moreover the lower left entries of the 1-forms $g^{-1}(y)\pi^*\eta g(y)$  and 
 $g^{-1}(y)\pi^*\omega^\lambda g(y)$ vanish at $y=0.$ 
 These locally defined gauges $g(y)$ glue together to form a globally defined gauge transformation on $M$, gauging the family $\lambda\mapsto D^\lambda$
 to a family of smooth flat connections \[\lambda\mapsto \bar\nabla^\lambda=\lambda^{-1}\Psi+\bar\nabla+\dots\]
on $M$  in a punctured neighborhood of $\lambda=0$ with first order pole at $\lambda=0$ and nilpotent residue $\Psi\in\Omega^{(1,0)}(M;\mathfrak{sl}(2,\C))$. We can do the same with
 the pullbacks of the connections $D^{U,\lambda}$ to obtain a well-defined map
 \[\hat{\mathcal D}\colon \C^*\to \mathcal A^2(M)\]
 satisfying the properties (1) and (2) of Theorem \ref{lifting_theorem}.  By the concrete construction in the proof of Lemma \ref{atlambda=0} 
 we have constructed (respectively chosen) the wrong lift $\lambda\mapsto\bar\nabla^\lambda$ of the map $\hat{\mathcal D}.$ This can be either seen from the
 geometric or from the algebraic viewpoint: algebraically, we would like to have a stable holomorphic structure at $\lambda=0;$ see for example Theorem 3.4 in \cite{He2} for the case $p=1$ 
 and $q=3.$ But the corresponding parabolic structure on the 4-punctured sphere  is unstable since $4\tfrac{2\rho+1}{4}>0$ (see $\S$ 2.4 in \cite{HeHe}).
  Geometrically, we want that the branch order is less than the umbilic order, and we want that the number of
  zeros (counted with multiplicity) of the nilpotent residue at $\lambda=0$ is less then $2g-2$ (see Theorem \ref{lifting_theorem}).
 In order to overcome this issue we do the following: let $L\subset\C^2\times M$ be the holomorphic kernel bundle of  the residue term $\Psi$ of $\lambda\mapsto\bar\nabla^\lambda$
 at $\lambda=0$ (note that this line bundle extends holomorphically through the zeros of $\Psi$).
 Since $\Psi$ is nilpotent and the connections are $\SL(2,\C)$ connections, the holomorphic structures have trivial determinant bundle and 
 we obtain an induced vector bundle homomorphism
 \[\Psi\colon \C^2/L\cong L^*\to K L;\]
 see $\S2$ in\cite{He1} for more details. We choose a topological complement $L^*\subset \C^2\times M$ of $L$ in $\C^2$ and
decompose
\[\C^2=L\oplus L^* \;\;\; \text{ and }\;\;  \bar\nabla^ \lambda=\lambda^{-1}\dvector{0&\alpha\\0&0}+\dvector{\nabla^L &\beta\\ \gamma&\nabla^{L^*}}+\lambda(\dots).\]
Since $L$ is a holomorphic line subbundle, $\gamma$ is a complex linear 1-form, and flatness of $\bar\nabla^\lambda$ (for all $\lambda\in U$)
implies that $\alpha\in H^0(M, KL^2)$ and $\gamma\in H^0(M, KL^{-2})$ are holomorphic. Define a gauge transformation
\[\Lambda=\dvector{\frac{1}{\sqrt{\lambda}}&0\\0&\sqrt{\lambda}}\] with respect to $\C^2=L\oplus L^*$ obtaining
\[\tilde\nabla^\lambda:=\bar\nabla^\lambda.\Lambda.\] Note that even though the gauge is double-valued (in $\lambda$), the family $\tilde\nabla^\lambda$ is well-defined.
By construction $\tilde\nabla^\lambda$ is also a lift of $\hat{\mathcal D}$ in a neighborhood of $\lambda=0$
satisfying condition (2) of Theorem \ref{lifting_theorem}. We therefore have proven that in the case of $q$ is odd there exists
a map $\hat{\mathcal D}$ satisfying the condition (1) and (2) of Theorem \ref{lifting_theorem}.
From \eqref{DPW_upupstairs} we obtain that
 the residue of $\tilde\nabla^\lambda$ at $\lambda=0$ has  zeros of order $(q-2p-1)$ at the four preimages $\pi^{-1}(z_0)$ of the singular points on $\CP^1$ and no other zeros, and the Hopf differential
 (which is given by $\alpha$ up to a multiplicative constant) has zeros of order $2p-1$ at the same points. Note that for $0<\rho<\tfrac{1}{4}$ the number $4(q-2p-1)$ 
 of zeros of  the residue of $\tilde\nabla^\lambda$ at $\lambda=0$ is less than $2g-2=2q-4.$\\
 
In the second case, when $q$ is even, we pull back the connections
 to the totally branched $\tfrac{q}{2}$-covering $\pi\colon M\to \CP^1$
 determined by the algebraic equation
 \[Y^{\tfrac{q}{2}}=\frac{Z}{(Z-1)(Z-m)}.\]
  A similar argument as in the first case using a two-valued gauge
 \[g(y)=\dvector{y^{-\tfrac{p}{2}}&0\\0&y^{\tfrac{p}{2}}}\]
 shows the existence of a map $\mathcal D$ satisfying the conditions (1) and (2) of Theorem \ref{lifting_theorem} and determines the corresponding
 branch and umbilic orders at the  preimages $\pi^{-1}(p_k)$  ($k=1,\dots,4$) of the singular points on $\CP^1.$

  \begin{The}\label{branched_CMC}
Let $\rho\in\Q$ with $0<\rho<\tfrac{1}{2}$ and let $p$ and $q$ be coprime integers satisfying $\tfrac{p}{q}=\frac{2\rho+1}{4}$. 
The analytic continuation  of the CMC surface 
 \[f\colon T^2\setminus (l_1\cup l_2)\to S^3\]
in Theorem \ref{slitting_tori} is a compact (branched) CMC surface $\hat f\colon M\to S^3.$ For odd $q$  the surface $\hat f$ is of genus $g=q-1$ with
four umbilic branch points of branch order $q-2p-1$ and umbilic order $2p-1$.
For even $q$, $\hat f$ is of genus $g=\tfrac{q}{2}-1$ with branch order $\tfrac{q}{2}-p-1$ and umbilic order $p-1$ at these four points. The surface has no other
 branch 
points or umbilics.
 \end{The}
 \begin{proof}
 We need to prove that the map $\hat{\mathcal D}\colon \C^*\to\mathcal A^2(M)$ satisfies condition (3) in Theorem \ref{lifting_theorem} and to relate the analytic continuation of the surface given by Theorem \ref{slitting_tori}
 with the surface $\hat f\colon M\to S^3$ corresponding to $\hat{\mathcal D}.$ 
 As we are primarily interested in proving the existence of closed CMC surfaces, the second part  is actually not so important, 
 and we only sketch the idea of proof: first observe that both surfaces $f\colon T^2\setminus (l_1\cup l_2)\to S^3$ and $\hat f\colon M\to S^3$ are the analytic continuations
 of the same CMC surface $\tilde f$ defined on $\CP^1\setminus\lambda(l_1\cup l_2).$ 
 This can be proven by carefully going through the construction of the associated families of the CMC surfaces $\tilde f$ and $f$ respectively $\hat f$.
 The associated families are constructed
  by an application of loop group factorization methods (see the proofs of Theorem 6 and Theorem 8 in \cite{He3}), and  it can be shown that the associated families must be related
  to each other by pullback since the unique parallel hermitian metrics are related by pullback.
  The last statement is equivalent to the fact that the surfaces are analytic continuations of each other.\\
  
  Finally, we show that the gauge classes $\mathcal D(\lambda_k)$ for $k=1,2$ are trivial on $M.$ By
  Remark \ref{2:1explicitRem}
it is enough to observe that the pullback of the Fuchsian system
\eqref{almost-trivial}  for $\epsilon=\tfrac{2\rho+1}{4}=\tfrac{p}{q}$ has trivial monodromy on $M.$
In the case of $q$ is odd the meromorphic function $Y\colon M\to\C$ determined by \eqref{algeq}
gives rise to a globally well-defined parallel frame $F$ of the pullback of the Fuchsian system
\eqref{almost-trivial}, namely
\[F=\dvector{\frac{1}{Y^p}&0\\0&Y^p}.\]
This follows from the observation that
\[\frac{dY}{Y}=\frac{1}{q} \frac{dY^q}{Y^q}=\frac{1}{q} d \log(\frac{Z}{(Z-1)(Z-m)})=\frac{1}{q}\pi^*(\frac{dz}{z}-\frac{dz}{z-1}-\frac{dz}{z-m}).\]
In the case of $q$ is even we obtain a double-valued parallel frame
\[F=\dvector{ Y^{-\tfrac{p}{2}}&0\\0&Y^{\tfrac{p}{2}} }   \]
which is good enough to obtain a closed surface on $\hat f\colon M\to S^3$ (recall the discussion in the beginning of the proof of Theorem \ref{slitting_tori}). 
  \end{proof}
\begin{Rem}
In the case $\rho=\tfrac{1}{6}$ Theorem \ref{branched_CMC} is equivalent to Theorem 8  in \cite{He3} and yields
 Lawson symmetric CMC surfaces of genus $2$ in terms of spectral data.
 More generally, for  $\rho=\tfrac{g-1}{2g+2}$ we obtain compact CMC
 surfaces
(without branch points) of genus $g$ with  $\Z_{g+1}$
 symmetry. 
\end{Rem}
\begin{Rem}
We have dealt with the even and the odd case by working on the 4-punctured sphere.
In the even case it would have been more natural to work directly on the 4-punctured torus in order to avoid some of the two-valued gauge transformations.
But then we would have needed to prove an analogon of \ref{atlambda=0} for the 4-punctured torus, which we have decided to avoid as being a repetetion.
\end{Rem}

\section{A generalized Whitham flow on the spectral data}\label{Whitham flow}
It is known that the generic spectral data $(\Sigma,\chi,\mathcal D)$ of CMC tori in the 3-sphere can be uniquely deformed such that induced deformation of the corresponding  surface preserves the closing conditions.  The existence of these so called Whitham deformation for CMC tori, introduced in \cite{HaKiSch1, HaKiSch2, KSS}, is shown by applying the implicit function theorem on some finite dimensional "function" space (in fact, the space
one is looking at is in general an affine space whose underlying vector space is a function space) over the compact spectral curve. 
We generalize this approach by taking the weight $\rho$ as flow parameter, which can be considered as the generalized genus of the immersion. Our flow interpolates between CMC surfaces of different genera (via CMC immersions on a torus with two cuts $l_1$ and $l_2$) preserving the intrinsic and most of the extrinsic closing conditions. If the initial surface is closed,  then for small $\rho\in \Q$ during the deformation the analytic continuation of the CMC surface is closed by Theorem \ref{branched_CMC}. For the CMC tori considered in Sections \ref{Tori_spec_gen_0} and \ref{Tori_spec_gen_1} as initial data, we show in the following that the closing conditions we impose determine a unique vector field for the deformation, and that its flow exists for small times. We conjecture that the flow exist for short times for generic CMC tori as initial surface.

 \subsection{Flowing from stable homogeneous tori}\label{31}
The first case we want to consider is the deformation of the spectral data for homogenous tori; see Section \ref{Tori_spec_gen_0}, with $\sqrt{2}\leq R < 2.$
The corresponding conformal type of the Riemann surface is given by  $T^2=\C/\Gamma$ with $\Gamma=2\Z+2\tau \Z$ for some
$\tau\in i\R$ with $1 \leq\Im (\tau) <\sqrt{3}.$ The 
spectral curve is  given by the two-fold covering 
\[\CP^1\ni \xi\mapsto \lambda = \xi^2 \in\CP^1\]
and the map $\chi$ is given by its lift
\[\hat\chi(\xi)= \tfrac{R\pi i}{4\tau}\xi d\bar w,\]
as derived in Section \ref{Tori_spec_gen_0}.
Again, we identify
\[\mathrm{Jac}(T^2)=\overline{H^0(T^2,K)}/\Lambda=\{\xi d\bar w\mid\xi\in\C\}/(\tfrac{\pi i}{2\tau}\Z+\tfrac{\pi i}{2}\Z)d\bar w.\]
Recall that the extrinsic closing condition, i.e., the Sym point condition, relates $R$ to the conformal type $\tau$ by \eqref{homogeneous_sym}.\\

Using the global coordinate $d\bar w$ to identify 
$\overline{H^0(T^2,K)}$ with $\C$
 the lattice $\Lambda$ corresponds to the lattice $\hat\Lambda$ generated by $\tfrac{\pi i}{2\tau}$ and $\tfrac{\pi i\tau}{2\tau}$ in $\C.$ \\

For  $d>0$ we consider the (complex) Banach space of bounded holomorphic functions on the disc
$\{\xi\in\C\mid \xi\bar\xi<d^2\}$ equipped with the supremum norm. 
We denote by $\mathcal B_d$ the real Banach subspace consisting of odd functions with respect to
$\xi\mapsto-\xi$ satisfying the reality condition
$f(\bar\xi)=\overline{f(\xi)}.$\\

Recall from Lemma \ref{real_analytic} that the Mehta-Seshadri section $\alpha^{MS}_\rho$ depends real analytically
on $[\xi]\in \mathrm{Jac}(T^2)$ as well as its "lift"
$\alpha_\rho^u\colon \C\setminus \hat\Lambda
\to\C$ does; see \eqref{lift_alpha}.
For $\rho\in]-\tfrac{1}{2},\tfrac{1}{2}[$ and a holomorphic function $f\in \mathcal B_d$ (with $d>1$)
such that $f(\xi)\notin\hat\Lambda$ for all $\xi\in S^1$ we consider the holomorphic function
\[\alpha^\rho(f)\colon U\subset\C\to\C\]
defined on an open neighborhood $U$ of the unit circle $S^1\subset\C$  uniquely determined by the property
\begin{equation}\label{holomorphic_extension}
\alpha^\rho(f)(\xi)=\alpha^u_\rho(f(\xi)) \text{ for all } \xi\in S^1. 
\end{equation}

\begin{Lem}\label{A}
Let $R\in \R^{>0}$ be such that $\hat{\chi}(\xi)=R\tfrac{\pi i}{4\tau}\xi$ 
 does not contain a lattice point in $\hat\Lambda$
along the unit circle. 
There exist real numbers $c,\delta>0$ and an open neighborhood $\mathcal U\subset\mathcal B_{1+c}$ of
$\hat{\chi}$
such that for all $-\delta<\rho<\delta$ and for all $f\in \mathcal U$ the holomorphic function $\alpha^\rho(f)$ is well-defined and bounded on
the annulus $\{\xi\in\C\mid \tfrac{1}{(1+c)}<\mid\xi\mid<1+c\}.$ Moreover, $\alpha^\rho(f)$ is odd and satisfies
$\alpha^\rho(f)(\bar\xi)=\overline{\alpha^\rho(f)(\xi)}.$
\end{Lem}
\begin{proof}
The image  $\hat{\chi}( S^1)$ is a compact subset $C$ of $\overline{H^0(T^2,K)}\cong\C$
which does not contain a lattice point in $\hat\Lambda.$ Hence there exists an open neighborhood 
$V\subset \C\times\C^2$
containing the set
\[\{ (0,R\xi,R\xi^{-1})\mid\xi\in S^1\subset\C\}\]
on which there is a bounded holomorphic 
function
\[F\colon V\to \C\]
with the property
\[F_\rho(\xi,\bar \xi)=\alpha^\rho_u(\xi)\]
for all $(\rho,\xi,\bar\xi)\in\R\times\C^2\cap V.$
Let  $An_c : =\{\xi\in\C\mid \tfrac{1}{(1+c)}<|\xi|<1+c\}.$
Then there exist $c>0$ and $\delta>0$ with
\[]-2\delta;2\delta[\;\times\;  \{  (\hat{\chi}(\xi), \overline {\hat{\chi}(\xi)} )\mid \xi\in An_{2c}\} \subset V.\]

Since  $\mathcal B_d$ is equipped with the supremum norm,  there exists an open neighborhood $\mathcal U$
of $\hat{\chi}\in\mathcal B_{1+c}$ such that for all $\rho\in]-\delta;\delta[,$ for all $f\in\mathcal U$ and for all
$\xi\in An_c$ we have
\[(\rho,f(\xi),\overline{f(\bar\xi^{-1})})\in V.\]
But for $\rho\in]-\delta;\delta[,$ $f\in \mathcal U$ and $\xi\in An_c$ we have
\[\alpha^\rho(f)(\xi)=F_\rho(f(\xi),\overline{f(\bar\xi^{-1})})\]
which proves that for all $f \in \mathcal U$ the holomorphic function $\alpha^\rho(f)$ is  bounded on $An_c.$
The remaining statements follow from Lemma \ref{a^u_symmetries}.
\end{proof}

\begin{Lem}\label{B}
Let $\tau\in i\R^{\geq1}$ be the conformal type of a homogeneous CMC torus with $\hat\chi(\xi)=R_0\tfrac{\pi i}{4\tau}\xi$ for
 $ R_0=\sqrt{1+\tau\bar\tau}\in [\sqrt{2}, 2[.$
Then there is an open neighborhood $U\subset\R^2$ of $(0,R_0) \in \R^2$ and a constant $\epsilon>0$
with the property that for all $(\rho,R) \in U$ there exists a unique holomorphic function
$\hat\chi^\rho_R\in \mathcal U\subset\mathcal B_{1+\epsilon}$  
such that the holomorphic function
\[\alpha^\rho(\hat\chi^\rho_R)\]
extends holomorphically to $\{\xi\mid 0<\xi\bar\xi\leq 1\}$ and has a first order pole at $\xi=0$ with
residue $R\tfrac{\pi i}{4\tau}.$ Moreover $\hat \chi^\rho_R$ depends smoothly on $(\rho,R).$
\end{Lem}
\begin{proof}
For $R \in ]1, 2[$ consider the holomorphic function $\hat\chi^0_R(\xi)=R\tfrac{\pi i}{4\tau}\xi$. The image of the unit circle under $\hat\chi$ does not contain a lattice point in $\hat\Lambda$.
By Lemma \ref{A} there exists an $\epsilon>0$ and an open neighborhood $\mathcal U$ of $\hat \chi^0_R$
in $\mathcal B_{1+\epsilon}$ such that  the holomorphic function
$\alpha^\rho(f)$ is bounded on the annulus $An_\epsilon$ for all $f\in \mathcal U$ and $\rho$ small enough.\\

Let  $\bar{\mathcal B}_{1+\epsilon}$ be  the space of odd, bounded holomorphic functions
on $\CP^1\setminus\{\xi \mid |\xi|\leq\tfrac{1}{1+\epsilon}\}$ satisfying
$f(\bar\xi)=\overline{f(\xi)}.$ Clearly,  $\mathcal B_{1+\epsilon}$ and $\bar{\mathcal B}_{1+\epsilon}$ are isomorphic
by $f\mapsto (\xi\mapsto f(\xi^{-1})).$
For $f\in\bar{\mathcal B}_{1+\epsilon}$ with
Laurent expansion $f=\sum_k f_k\xi^k$ along $S^1\subset\C$
define the bounded linear map
\begin{equation}\label{Mdef}M\colon \bar{\mathcal B}_{1+\epsilon}\to\bar{\mathcal B}_{1+\epsilon};\, M(f) =\sum_{k<0}f_k\xi^k.\end{equation}
Consider the  differentiable map between Banach spaces
\[A_\rho\colon \mathcal U\subset\mathcal B_{1+\epsilon}\to \bar{\mathcal B}_{1+\epsilon};\,f\mapsto M(\alpha^\rho(f)).\]
By Lemma \ref{A} $A_\rho$ depends smoothly on $\rho.$ Further,  for $\rho=0$ the differential of $A_{\rho}$ is  an (continuos) Banach space isomorphism
because $\alpha^0(f(\xi))=f(\xi^{-1}).$ 
Hence, the differential of $A_\rho$ is an isomorphism for small $\rho$ in a neighborhood of $\hat\chi.$
By
the implicit function theorem there exists an open neighborhood $U\subset\R^2$ of $(0,R_0)$ such that  for all $(\rho, R) \in U$ there exist a map $\chi^\rho_R \in \mathcal B_{1+ \epsilon}$
with \[A_\rho(\chi^\rho_R)(\xi)=R\tfrac{\pi i}{4\tau}\xi^{-1},\]
proving the lemma.
\end{proof}

\begin{The}\label{flow_homog}
Let $\tau\in i\R^{\geq1}$ be the conformal type of a homogeneous CMC torus with $\chi(\xi)=R_0\tfrac{\pi i}{4\tau}\xi,$ where
 $R_0=\sqrt{1+\tau\bar\tau} \in [\sqrt{2}, 2[.$
Then for $\rho$ small enough there exist  a unique
 $\chi^\rho\colon\{\xi\mid |\xi|<1+\epsilon\}\to \mathrm{Jac}(T^2)$ 
 and a lift $\mathcal D^\rho\colon\{\xi\mid |\xi|<1+\epsilon\}\to\mathcal A^1(T^2)$
 satisfying the conditions of Theorem \ref{slitting_tori}.
  \end{The}
 \begin{proof}
By Lemma \ref{B} we obtain for every $\rho$ near $ 0$ a one-dimensional set of functions
$\chi^\rho_R$ parametrized by $R$ satisfying the intrinsic closing condition that
$\alpha^\rho(\chi^\rho_R)$ extends to  the punctured disc $\{\xi\mid 0<\xi\bar\xi\leq 1\}$ with a first order pole at $\xi=0$ and
residue $R\tfrac{\pi i}{4\tau}.$ Moreover, $\chi^0_R(\xi)=R\tfrac{\pi i}{4\tau}\xi,$ and by restricting this map to $\xi\in S^1$ the differential of the map
\[(\xi,R)\in S^1\times\R^{>0}\mapsto \chi^0_R(\xi)\in \mathrm{Jac}(T^2)\] 
is invertible  everywhere. By the implicit function theorem there exists a smooth map
\[\rho\mapsto R(\rho)\]
with $R(0)=R_0$ such that the map \[\chi^\rho=\chi^\rho_{R(\rho)}\]
has the value $\tfrac{\pi i(1+\tau)}{4\tau}$ at some $\xi_1(\rho) \in S^1$. That $\chi^\rho$ is odd and real along the real axis
 $\R\subset\C$ implies  the existence of four distinct points on $\xi_i \in S^1$ with
$\chi(\xi_i)=\tfrac{\pi i(1+\tau)}{4\tau}$ mod $\hat \Lambda$ (for small $\rho$).
\end{proof}

 An immediate consequence of this Theorem together with Theorem \ref{slitting_tori} and Theorem \ref{branched_CMC} is:
\begin{Cor}
For large genus $g>>2$ we obtain new 1-parameter families of 
 compact branched CMC surfaces. The families are parametrized by the conformal type $\tau \in [1, \sqrt{3})i $ of the corresponding CMC torus.
  If the flow reaches $\rho = \tfrac{g-1}{2g +2}$ we obtain closed CMC immersions of genus g.
 \end{Cor}
\begin{Exa} We consider the case where the initial surface of the generalized Whitham flow is the (minimal) Clifford torus. The Clifford torus is a Lawson surface \cite{L}, namely $\xi_{1,1}$,  and has square conformal structure. This implies that there is an
additional symmetry on the Riemann surface and its Jacobian which we denote by $i$.  It is easy to show that for $\rho>0$
 the Mehta-Seshadri section of the square torus
 also has an additional symmetry: in terms of the lift to the universal covering we have for all
$\xi\notin\tilde\Lambda$
\begin{equation}\label{alpha-i-symmetry}
\alpha_\rho^u(i\xi)=-i\alpha_\rho^u(\xi).
\end{equation}
Consequently, we can restrict to the following subspaces of holomorphic functions:
\[\mathcal B^+_{1+\epsilon}=\{f\in\mathcal B_{1+\epsilon}\mid f(i\xi)=if(\xi)\}\]
and
\[\bar{\mathcal B}^-_{1+\epsilon}=\{f\in\bar{\mathcal B}_{1+\epsilon}\mid f(i\xi)=-if(\xi)\}\]
in the proof of the Theorem \ref{flow_homog}. Because of \eqref{alpha-i-symmetry} all maps remain well-defined by restricting to these two subspaces.  Thus  the map 
$\chi^\rho$ constructed  in Theorem \ref{flow_homog} satisfies 
\[\chi^\rho(i\xi)=i\chi^\rho(\xi)\]
for $\rho>0.$
Therefore the preimages of the Sym points  are $\xi=\pm e^{\pm\tfrac{ \pi i}{4}}$
and all surfaces $f^{\rho}$ within the flow starting at the Clifford torus are minimal. Let $S \subset T^2=\C/(2\Z+2i\Z)$ be the square defined by the vertices $[0],[\tfrac{1+i}{2}],[1],[\tfrac{1-i}{2}] \in T^2.$ Since the coordinate lines of $T^2$ are curvature lines of the corresponding surfaces (by reality of the Hopf differential), the diagonals are asymptotic lines. 
Because of the reality of the spectral data, the diagonals are geodesics on the surface.
Thus the image of the boundary of $S$ under the minimal immersion $f^{\rho}$ is a geodesic 4-gon, which is the geodesic polygon used in \cite{L} to construct the Lawson surfaces; see Figure \ref{fig:lawson}.
\end{Exa}

\subsection{Flowing from Delaunay tori}\label{32}
Next we  define the generalized Whitham flow starting at a slightly more complicated surface class: the 2-lobed Delaunay tori.
As we have seen in Section \ref{Tori_spec_gen_1} the spectral data $\chi\colon \Sigma\to \mathrm{Jac}(T^2)$ 
of a 2-lobed Delaunay torus does not lift to a well-defined function $\hat{\chi}\colon \Sigma\to \overline{H^0(T^2,K)}$ but rather has periods. Nevertheless, these periods
are constant along every continuous deformation and therefore  a deformation vector field of $\chi$ is still a well-defined map
into $\overline{H^0(T^2,K)} \cong \C.$ 
The spectral curve $\Sigma$ of a Delaunay torus is given
by the algebraic equation
\[\Sigma:\; y^2=\lambda(\lambda-r)(\lambda-\tfrac{1}{r})\]
for some $r\in]0;1[$, and we consider $y,\lambda\colon\Sigma\to \CP^1$ as meromorphic functions.
 It is well-known that $r$ is determined by $\tau_{spec}\in i\R$ and vice versa.
 The real involution which fixes the preimages $\lambda^{-1}(S^1)$ of the unit circle is given by
 \begin{equation}\label{real_involution}
 (y,\lambda)\mapsto(\bar y \bar\lambda^{-2},\bar\lambda^{-1}).
 \end{equation}
 
 Deformations of odd maps to $\mathrm{Jac}(T^2)$ can be identified
with odd functions on $\Sigma,$ and any odd holomorphic function $\hat f\colon U\subset \Sigma\to \C$ on an open subset $U$ is given by
\[\hat f=y f,\]
where $f\colon \lambda(U)\to \C$ is a holomorphic function on the spectral plane $\C$ with coordinate $\lambda$ (we use the shorthand notation $f=f\circ\lambda\colon U\to\C$ throughout  the section). Hence,
we work with the Banach space $\mathcal B_{d}$ of bounded holomorphic functions
\[f\colon\{\lambda\in\C\mid |\lambda|<d\}\to\C\] satisfying the condition $f(\bar\lambda)=\overline{f(\lambda)}.$ 
We also consider the Banach space $\bar{\mathcal B}_d$ of bounded holomorphic functions
on $\CP^1\setminus\{\lambda\in\C\mid |\lambda|\leq \tfrac{1}{d}\}$ satisfying $f(\bar\lambda)=\overline{f(\lambda)}.$
Clearly, $\mathcal B_d$ and $\bar{\mathcal B}_d$ are isomorphic via $f\mapsto (\lambda\mapsto f(\lambda^{-1})).$\\

We fix  $\rho\in]-\tfrac{1}{2};\tfrac{1}{2}[. $ For the 2-lobed Delaunay tori the map $\chi$ has period $2.$ Let 
\begin{equation}\label{Xbase_point}
\chi=\int(a\wp(\xi-\tfrac{\tau_{spec}}{2})+b) d\xi,
\end{equation} where $a$ and $b$ are defined in \eqref{2lobe_spec_periods2}
and $\xi$ is the holomorphic coordinate on $\Sigma \cong \C/\Gamma_{spec}$. Let  
\[\hat A(\xi)=-\overline{ \chi(\bar\xi+\tfrac{\tau_{Spec}}{2})}.\]
Then  every odd holomorphic map with period $2$ along $1 \in \pi_1(\Sigma)$ 
and trivial period along $\tau_{spec} \in \pi_1(\Sigma)$  is given by 
\[\tilde\chi= \chi+ y f\colon \lambda^{-1}(\{\lambda\in\C\mid |\lambda|<d\})\subset\Sigma\to \C\]
for a holomorphic function $f: \{\lambda\in\C\mid |\lambda|<d\}\rightarrow \C$. There exist a holomorphic function
\[\alpha^\rho(\tilde\chi)\colon U\to \C\]
defined on an open neighborhood $U$ of $\lambda^{-1}(S^1)\subset \Sigma$
satisfying the equation
\[[d+(\alpha^\rho(\tilde\chi)(\xi)+\hat A(\xi)) \tfrac{\pi i}{2\tau}dw-\tilde\chi(\xi)\tfrac{\pi i}{2\tau}d\bar w]=\alpha^{MS}_\rho([\tilde\chi(\xi)\tfrac{\pi i}{2\tau}d\bar w])\in\mathcal A^1(T^2)\]
for all $\xi\in \lambda^{-1}(S^1).$ The splitting of the anti-holomorphic structure  $\alpha_\rho^u$ into $\hat A$- and $\alpha^\rho(\tilde\chi)$-parts ensures that $\alpha^\rho(\tilde\chi)$ is single-valued on $U$ by the functional properties of $\alpha_\rho^u$; see \eqref{lift_alpha}. 
Analogously to  Lemma \ref{A} we obtain:
\begin{Lem}\label{A2}
Let $T^2$ be a rectangular torus and $\chi\colon\Sigma\to \mathrm{Jac}(T^2)$
be as  in \eqref{Xbase_point}. Assume that 
the image under 
$\chi$ of the preimage of the unit circle 
$\lambda^{-1}(S^1)\subset\Sigma$ does not contain a lattice point.
 Then 
there exist $c>0,$ $\delta>0$ and an open neighborhood $\mathcal U\subset\mathcal B_{1+c}$ of $0$
such that for all $-\delta<\rho<\delta$ and for all $f\in \mathcal U$ the holomorphic function $\alpha^\rho(\chi+y f)$ is well-defined and bounded on
$\lambda^{-1}\{\lambda\in\C\mid \tfrac{1}{(1+c)}<|\lambda|<1+c\}.$ Moreover, $\alpha^\rho( \chi+y f)$ is odd and satisfies
$\alpha^\rho(\chi+y f)(\bar\xi)=\overline{\alpha^\rho(\chi+y f)(\xi)}.$
\end{Lem}
Note that we do not demand in Lemma \ref{A2} that the Sym point conditions are satisfied. Thus  for a 
given conformal type $\tau$ of  $T^2$ there is an open  set $W_{\tau} \subset i\R$ determining a $1-$parameter family of genus $1$ spectral curves $\Sigma^{\tau_{spec}}$ such that the corresponding function 
$\chi^{\tau_{spec}}=\chi\colon\Sigma^{\tau_{spec}}\to \mathrm{Jac}(T^2)$ (as defined in \eqref{Xbase_point})
 satisfies
the assumptions of Lemma \ref{A2}. \\

By Lemma \ref{non_degenerate1}, $\mathcal D\colon\Sigma\to \mathcal A^1(T^2)$ must have a pole at $[\tfrac{1}{2}]\in\Sigma$ with residue determined by $\rho>0.$ Thus the main difference to the case of stable homogenous tori is that we need to satisfy the condition \eqref{a_spin_expansion} to flow starting at  the $2$-lobed Delaunay tori. The sign of the residue of $\mathcal D$ at $[\tfrac{1}{2}]$ introduces a freedom in the choice of the deformation vector field for $\rho=0.$

\begin{Lem}\label{B2}
Consider a 2-lobed Delaunay torus $T^2$ with spectral data
 $(\Sigma^{\tau_{spec}}, \chi)$ with $\tau_{spec}\in i\R.$
Then there is an open neighborhood $U\subset i\R\times\R^2$ of $(\tau_{spec},0,0)$ and an $\epsilon>0$
such that for all $(\tau_{s},\rho,R)\in U$ there exists a unique
$f_{\tau_{s},\rho,R}\in \mathcal B_{1+\epsilon}$
such that the holomorphic function
\[\alpha^\rho(\chi+y f_{\tau_{s},\rho,R})\colon\lambda^{-1}\{\lambda\in\C\mid \tfrac{1}{1+\epsilon}<|\lambda|<1+\epsilon\}\subset
\Sigma^{\tau_{s} }\to\C \]
extends meromorphically to $\lambda^{-1}\{\lambda\in \C\mid |\lambda|\leq 1\}\subset \Sigma^{\tau_{s}}$ with first order poles only at 
$[0],[\tfrac{1}{2}]\in\Sigma^{\tau_{s}}=\C/(\Z+\tau_{s}\Z)$  and such that its residue
at $[\tfrac{1}{2}]$ is $R.$
Moreover, $f_{\tau_{s},\rho,R}\in \mathcal B_{1+\epsilon}$ depends smoothly on $(\tau_{s},\rho,R).$
\end{Lem}
\begin{proof}
Analogously to  Lemma \ref{B} we obtain for small $\rho$ a map
\[A_\rho\colon \mathcal U\subset\mathcal B_{1+\epsilon}\to \bar{\mathcal B}_{1+\epsilon}\]
defined by 
\[A_\rho(f)=\tilde M(\alpha^\rho(f)),\]
where $\tilde M$ is the (continuous) linear map 
given by $\tilde M(y f\circ\lambda)=M(f)$ the principal part of $f$ as in \eqref{Mdef}.\\


For $\rho=0,$ the differential of the (linear) map
$A_0$
is not surjective anymore because  for $\hat f=yf\circ\lambda$
\[A_0(\hat f)(\lambda)=\lambda^{-2} f(\lambda^{-1})\]
by \eqref{real_involution}. Hence, we cannot control (by the implicit function theorem) the residue term at $(y,\lambda)=(0,0)\in\Sigma^{\tau_s}$
for the function $\alpha^\rho( \chi +y f_{\tau_{s},\rho,R}).$ Nevertheless, we still can guarantee that $
\alpha^\rho( \chi+yf_{\tau_{s},\rho,R})$ extends meromorphically and has only a first order pole at $(y,\lambda)=(0,0)\in\Sigma^{\tau_s}.$\\

The main difference is that we need to introduce a new parameter $R$ which gives the residue of the first order pole at
$[\tfrac{1}{2}]\in\Sigma.$ This new pole can easily be introduced for $\rho=0$ by adding to $\chi$ (an appropriate multiple of)
the unique odd meromorphic function $x\colon \Sigma^{\tau_{s}}\to\CP^1$ with first order poles at $[\tfrac{1+\tau_{s}}{2}]=\lambda^{-1}(\tfrac{1}{r})$ and $[\tfrac{\tau_{s}}{2}]=\lambda^{-1}(\infty)$ with residue $1$ and $-1,$ respectively.
\end{proof}

\begin{The}\label{flow_Delaunay}
Let $\tau\in i\R^{>0}$ be the conformal type of a 2-lobed Delaunay CMC torus with $\chi\colon\Sigma^{\tau_{spec}}\to \mathrm{Jac}(T^2)$ as described in Section \ref{Tori_spec_gen_1}.

 For small $\rho$, there exist an unique rectangular
 elliptic curve $\lambda_{\rho,+}\colon\Sigma^{\rho,+}\to\CP^1$ together with an odd holomorphic map
 \[\chi^{\rho,+}\colon\lambda_{\rho,+}^{-1}(\{\lambda\mid |\lambda|<1+\epsilon\})\subset\Sigma^{\rho,+}\to \mathrm{Jac}(T^2)\]
 and a lift $\mathcal D^{\rho,+}\colon\lambda_{\rho,+}^{-1}(\{\lambda\mid |\lambda|<1+\epsilon\})\to\mathcal A^1(T^2)$
 satisfying the conditions of Theorem \ref{slitting_tori} such that 
 the induced parabolic structure $\dbar^\lambda$ is stable for all $\lambda$ inside the unit disc (see Remark \ref{Sign}).

Moreover, for $\rho$ near 0, there exist an unique rectangular
 elliptic curve $\lambda_{\rho,-}\colon\Sigma^{\rho,-}\to\CP^1$ together with an odd holomorphic map
 \[\chi^{\rho,-}\colon\lambda_{\rho,-}^{-1}(\{\lambda\mid |\lambda|<1+\epsilon\})\subset\Sigma^{\rho,-}\to \mathrm{Jac}(T^2)\]
 and a lift $\mathcal D^{\rho,-}\colon\lambda_{\rho,-}^{-1}(\{\lambda\mid |\lambda|<1+\epsilon\})\to\mathcal A^1(T^2)$
 satisfying the conditions of Theorem \ref{slitting_tori} such that 
 there is exactly one $\lambda$ inside the unit disc where the induced holomorphic structure is unstable.

 Consequently, there exist two distinct flows starting at a 2-lobed Delaunay torus through closed CMC surfaces with boundaries in the sense of Theorem \ref{slitting_tori} and  Remark \ref{bound}. \end{The}
 \begin{figure}
\centering
\includegraphics[width=0.175\textwidth]{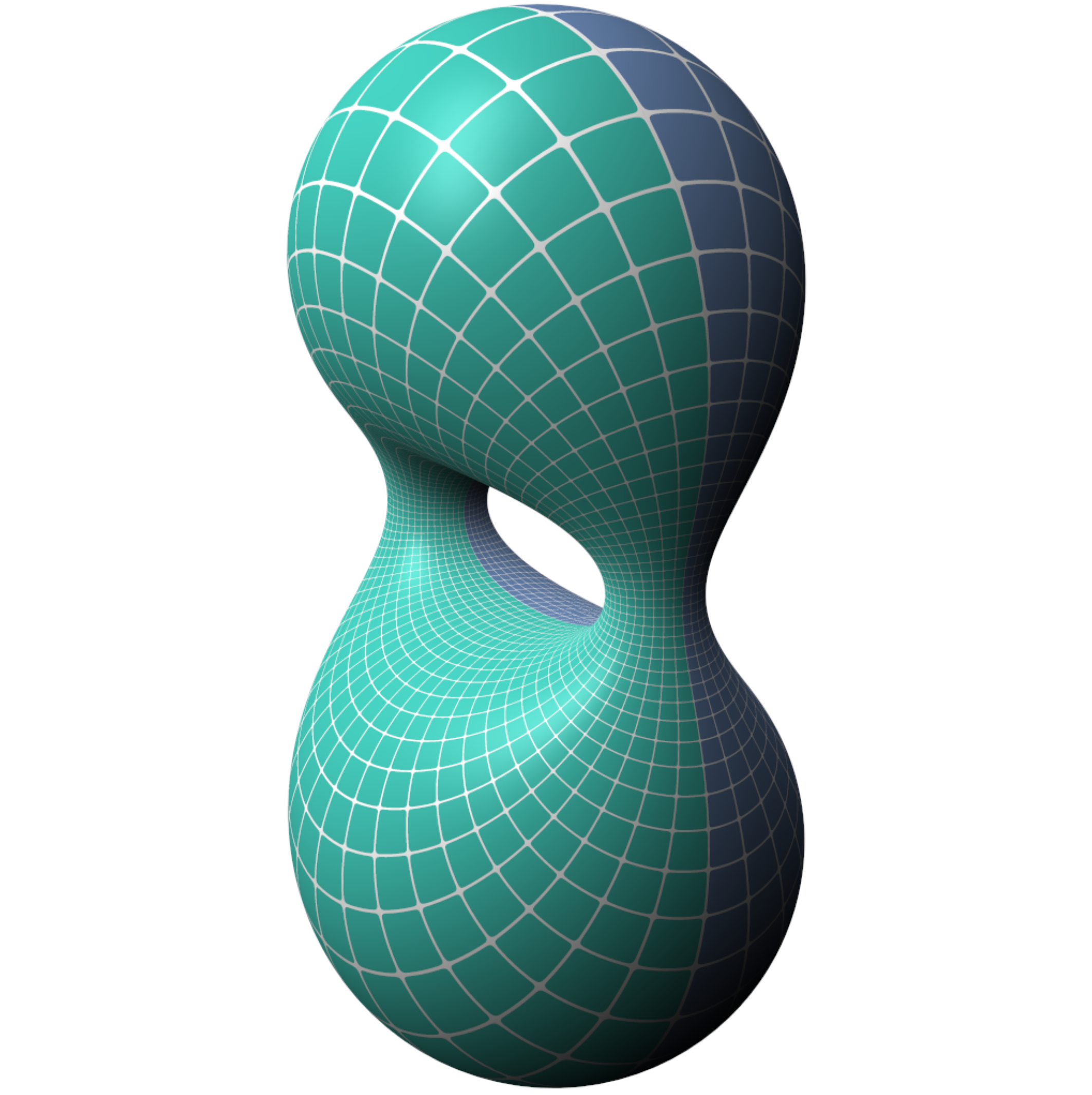}
\includegraphics[width=0.175\textwidth]{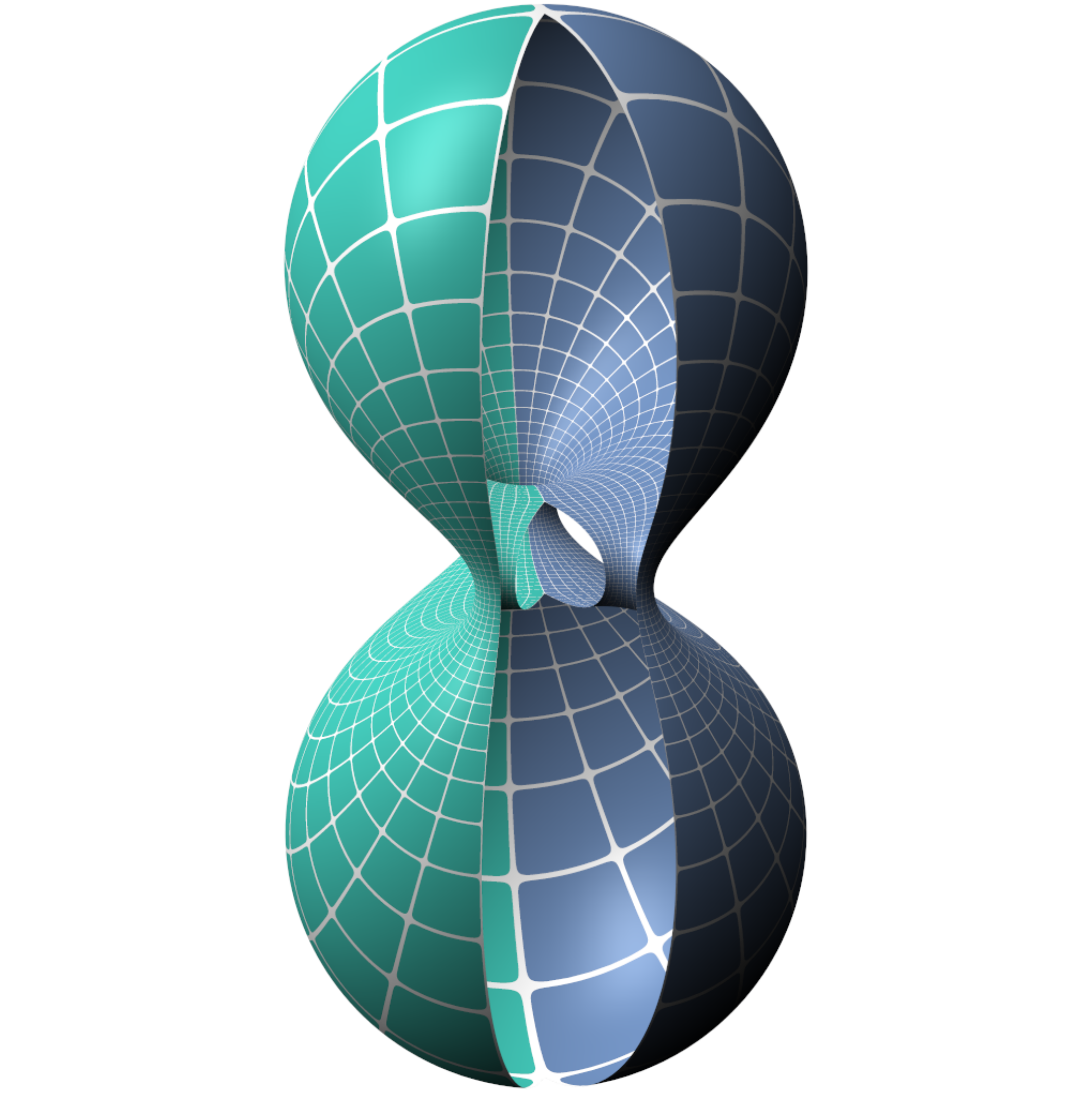}
\includegraphics[width=0.175\textwidth]{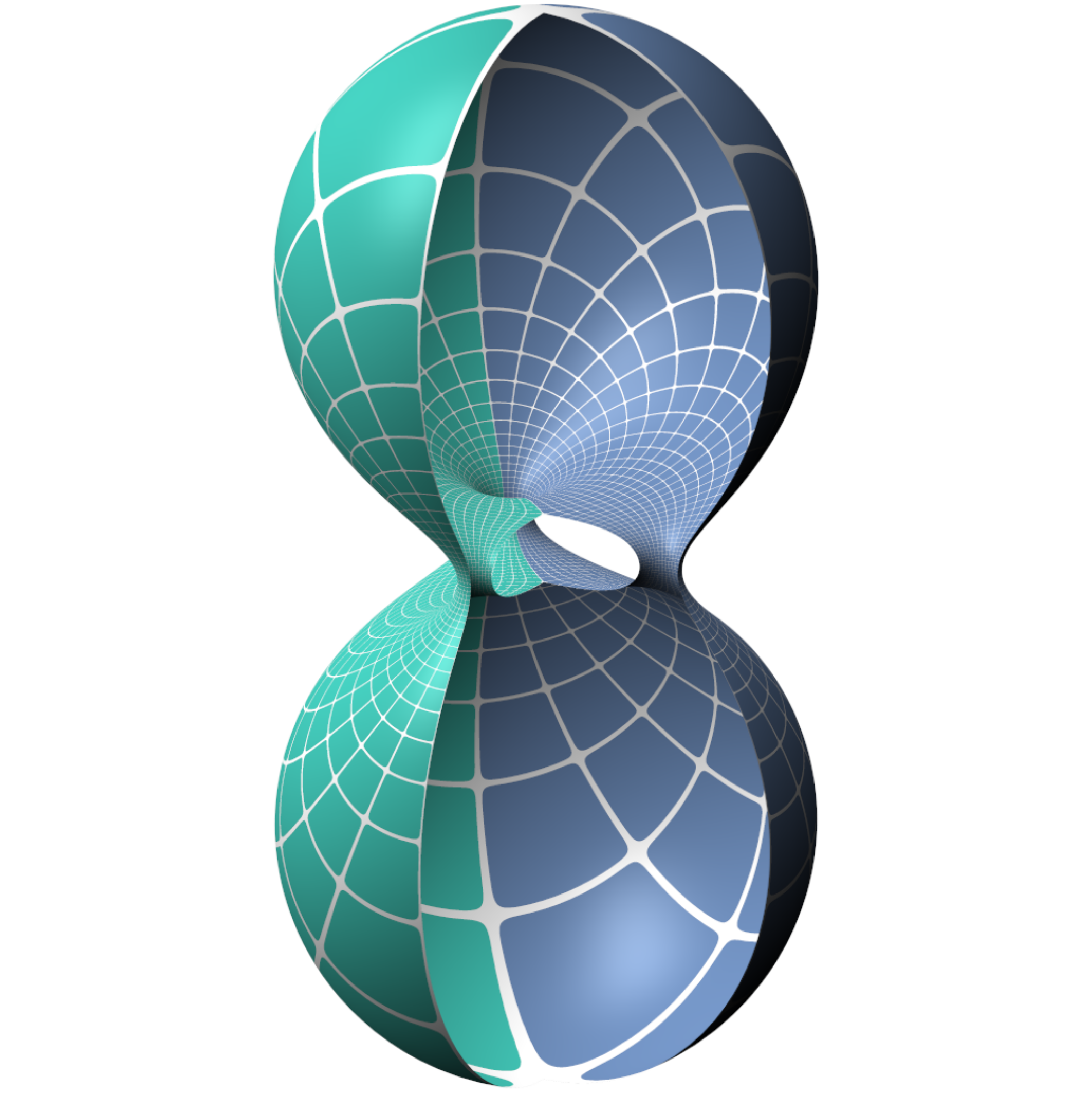}
\includegraphics[width=0.175\textwidth]{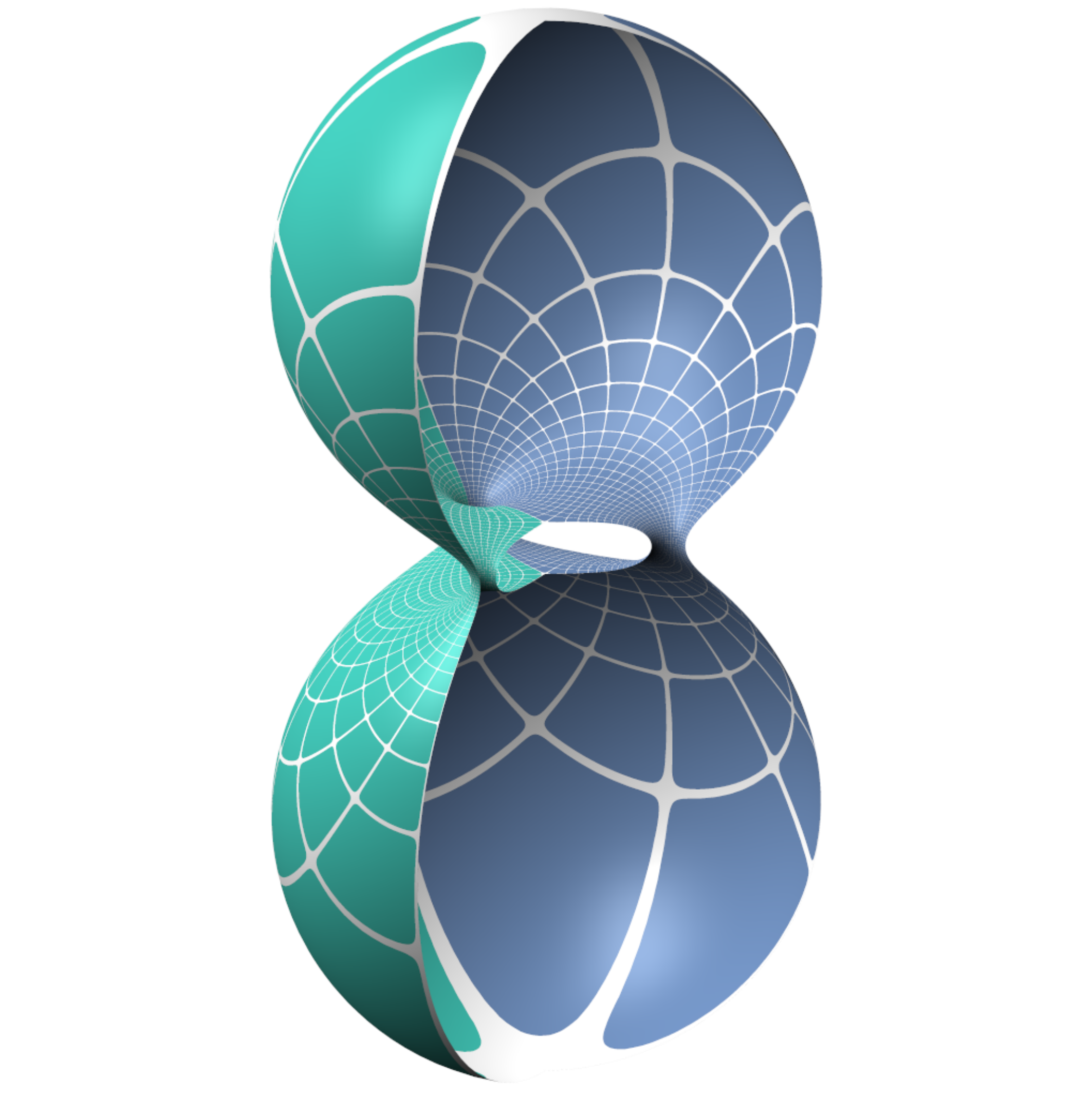}
\includegraphics[width=0.175\textwidth]{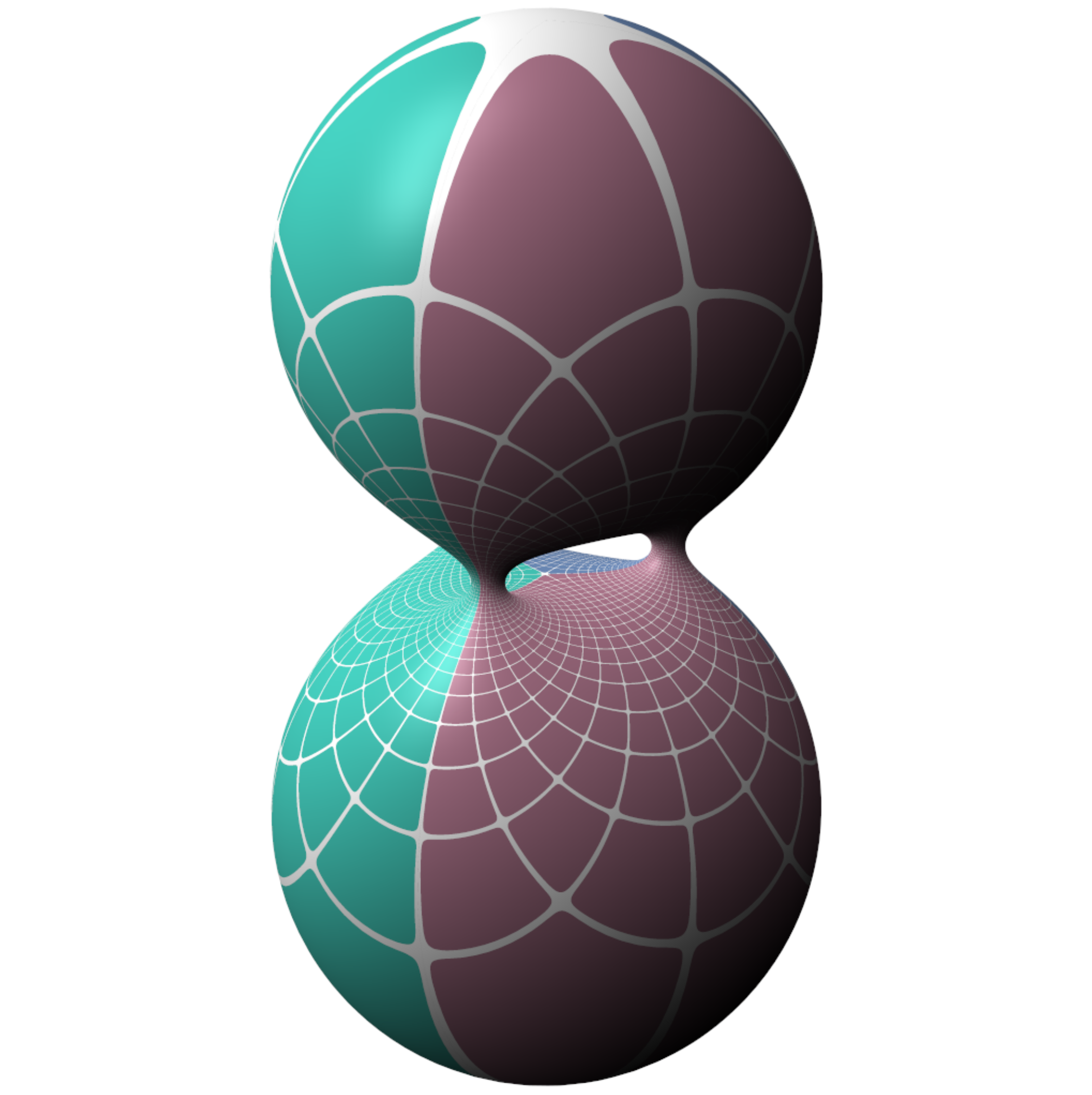}
\caption{
\footnotesize
The {\em stable} deformation of a 2-lobe Delaunay torus to a CMC surface of genus $2.$}
\label{fig:lawson2}
\end{figure}
\begin{proof}
The main difference from Theorem \ref{flow_homog} is that we
need to control the residue term at the branch point $[\tfrac{1}{2}]\in\Sigma^{\tau_s}.$
Consider for fixed ${\tau_s}\in i\R$ the following ODE on the 2-dimensional space
$U_{\tau_s}=\{(\rho,R)\mid ({\tau_s},\rho,R)\in U\}$ given by Lemma \ref{B2}:
\begin{equation}\label{ODE}
\dot{ R}^\pm(\rho)=\pm\frac{4x_1(\rho,R)-4\rho \dot x_1(\rho,R)}{(x_1(\rho,R))^2}
\end{equation}
where the dot denotes the derivative with respect to $\rho,$ and $x_1(\rho,R)$ is
the first order term of the expansion of $\chi+f_{\tau_s,\rho,R}$ at $[\tfrac{1}{2}]$ with respect 
to the holomorphic coordinate $\xi$ on $\Sigma^{\tau_s}=\C/(\Z+\tau_s\Z).$
 Note that the condition \eqref{a_spin_expansion} on the order 0 term $\bar\gamma$ is always
satisfied for $\alpha^\rho(\chi+y f_{\tau_{s},\rho,R})$ 
 as all relevant maps are odd. Therefore,
 each of the two ODEs \eqref{ODE} guarantees that condition (3) in Theorem \ref{slitting_tori}  is fulfilled by the  lift $\mathcal D$ (which is uniquely determined by the unitarity along the preimage of the unit circle) of the map \[\chi^\pm_{\tau_s,\rho}:=(\chi+ y f_{\tau_s,\rho,R^\pm(\rho)})\tfrac{\pi i}{2\tau} d\bar w\colon U\subset \Sigma\to \mathrm{Jac}(T^2).\]
Then for every small $\rho>0$ there is a real 1-dimensional space (para-metrized by $\tau_s\in i\R$) of appropriate spectral data
$(\Sigma^{\tau_s},\chi^\pm_{\tau_s,\rho})$ satisfying the intrinsic closing conditions. 
The extrinsic closing condition
is satisfied at $\rho=0$ for a unique $\tau_{s} = \tau_{spec},$ see Section \ref{Tori_spec_gen_1} and Lemma
\ref{non_degenerate2}. Thus by Lemma
\ref{non_degenerate2} we can apply (for each sign $\pm$) the implicit function theorem again to obtain a unique local parametrization
$\rho\mapsto\tau^{\pm}_s(\rho)$ such that the Sym point condition is satisfied by $\chi^\pm_{\tau^{\pm}_s(\rho),\rho}$ at a point
$\xi_1\in \lambda_{\rho,\pm}^{-1}(S^1)\subset \Sigma^{\tau_s^\pm(\rho)}.$ By reality of the spectral data,
there is a second point  $\xi_2\neq \pm\xi_1\in \lambda_{\rho,\pm}^{-1}(S^1)\subset \Sigma^{\tau_s^\pm(\rho)}$ at which the Sym point condition also holds.
 \end{proof}

For rational $\rho$ the analytic continuation of the corresponding surface closes to a compact (possibly branched) CMC surface. As a Corollary of this theorem and Theorem \ref{branched_CMC} we obtain:
\begin{Cor}
For every large genus $g>>2$ we obtain new 1-parameter families of 
 compact branched CMC surfaces. The families are parametrized by the conformal type $\tau \in (\sqrt{3},\infty)i $ of the corresponding CMC torus.
  If the flow reaches $\rho = \tfrac{g-1}{2g +2}$ we obtain closed CMC immersions of genus g.
 \end{Cor}

\begin{figure}
\centering
\includegraphics[width=0.175\textwidth]{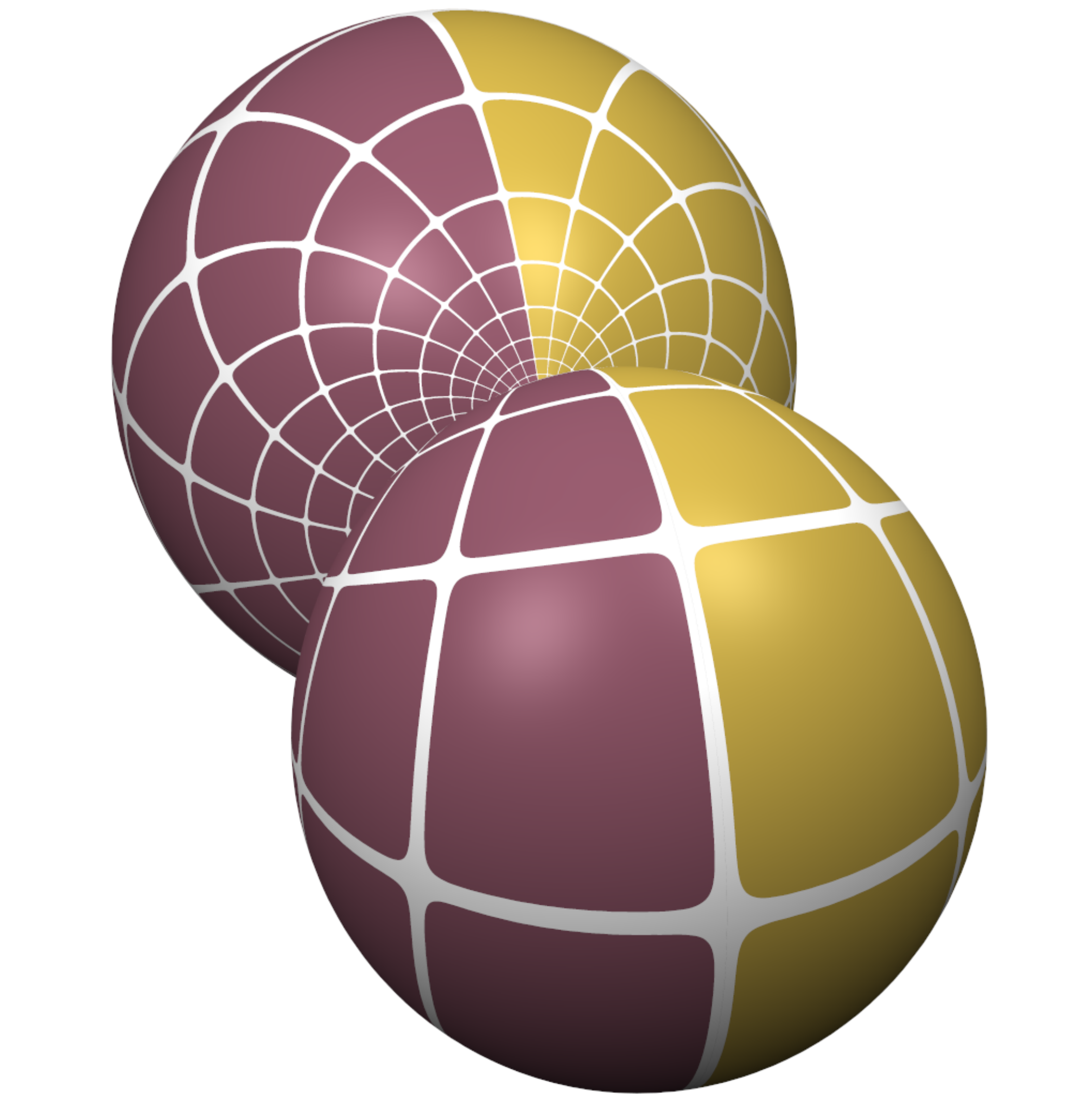}
\includegraphics[width=0.175\textwidth]{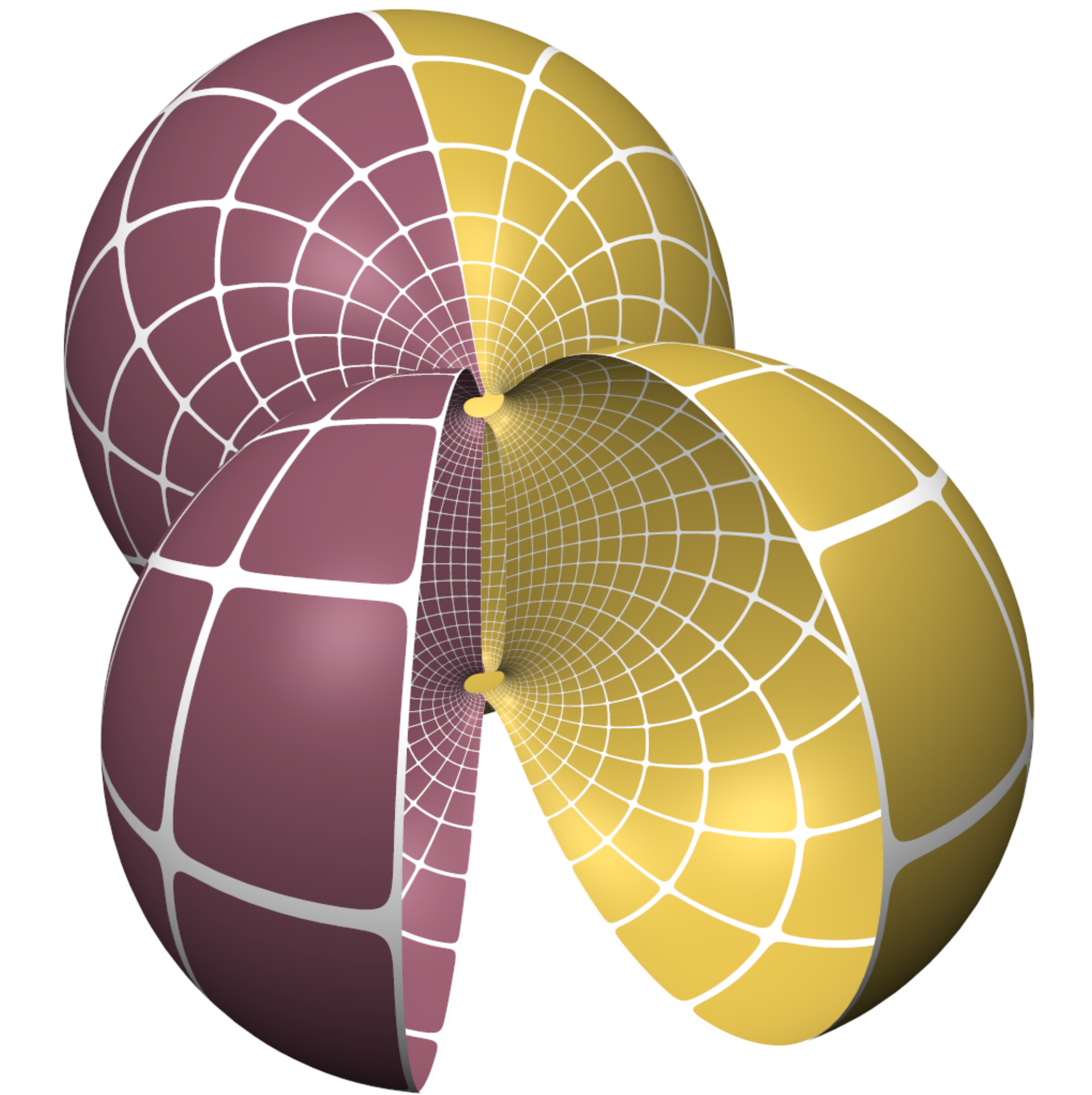}
\includegraphics[width=0.175\textwidth]{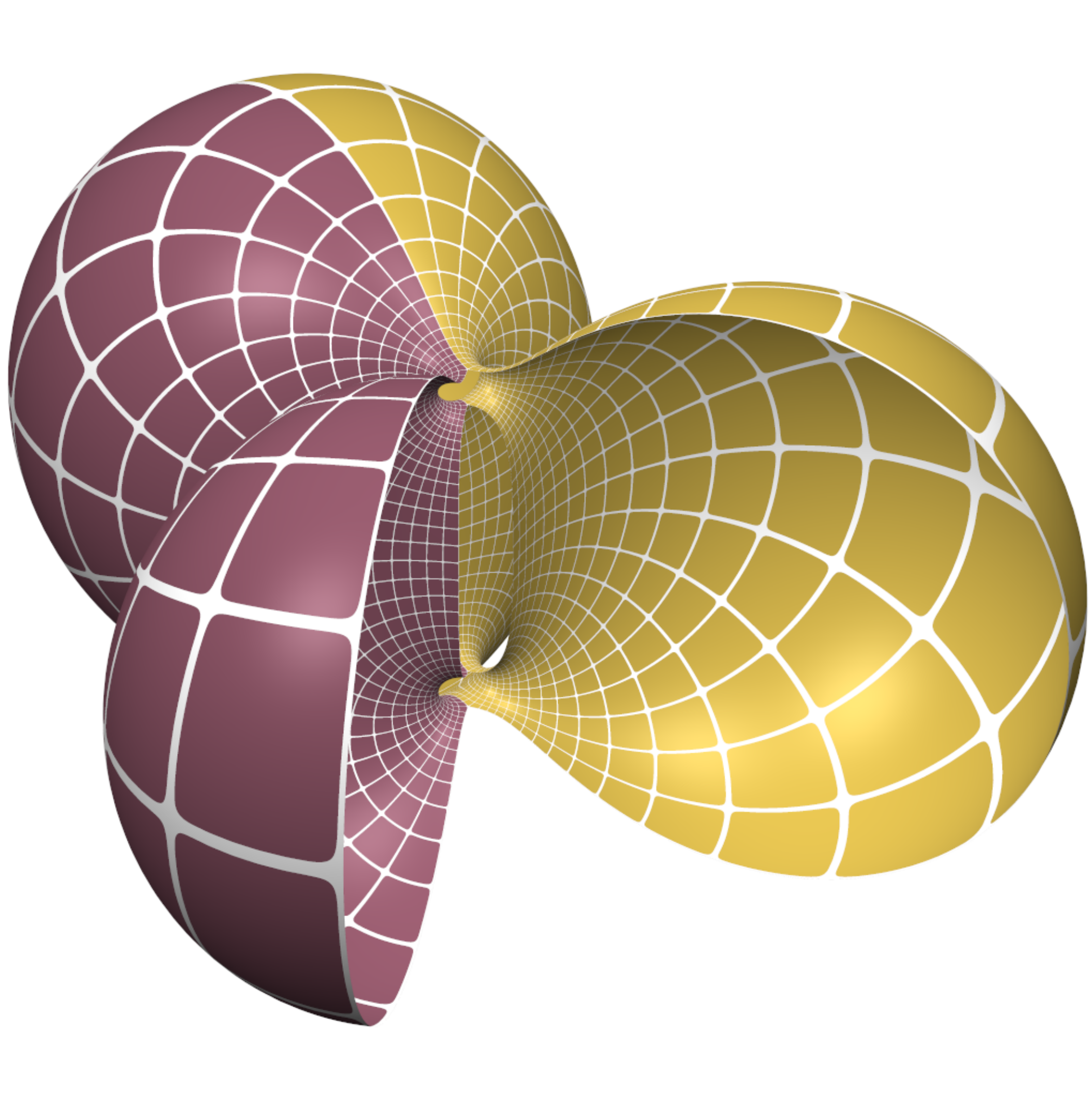}
\includegraphics[width=0.175\textwidth]{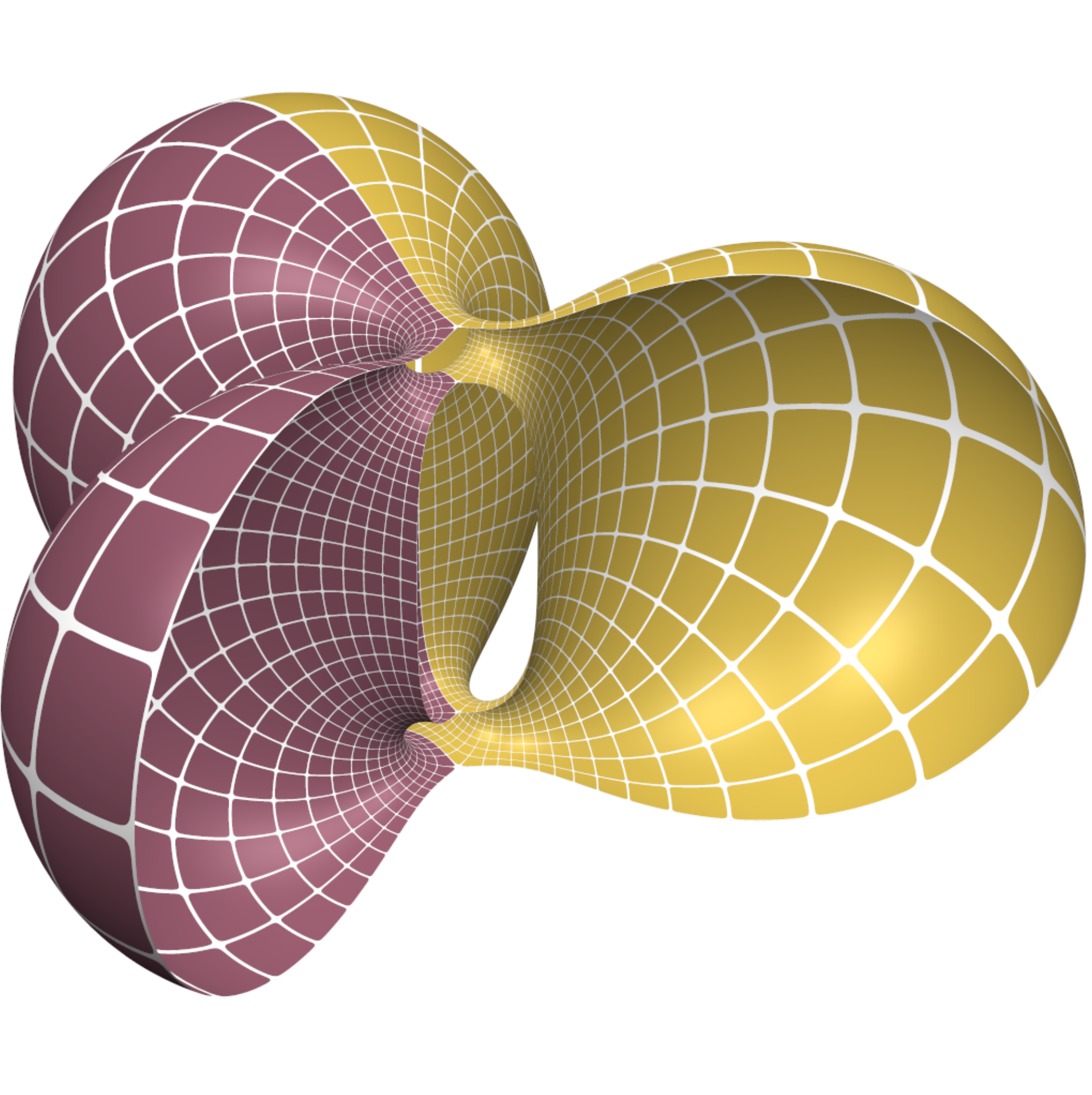}
\includegraphics[width=0.175\textwidth]{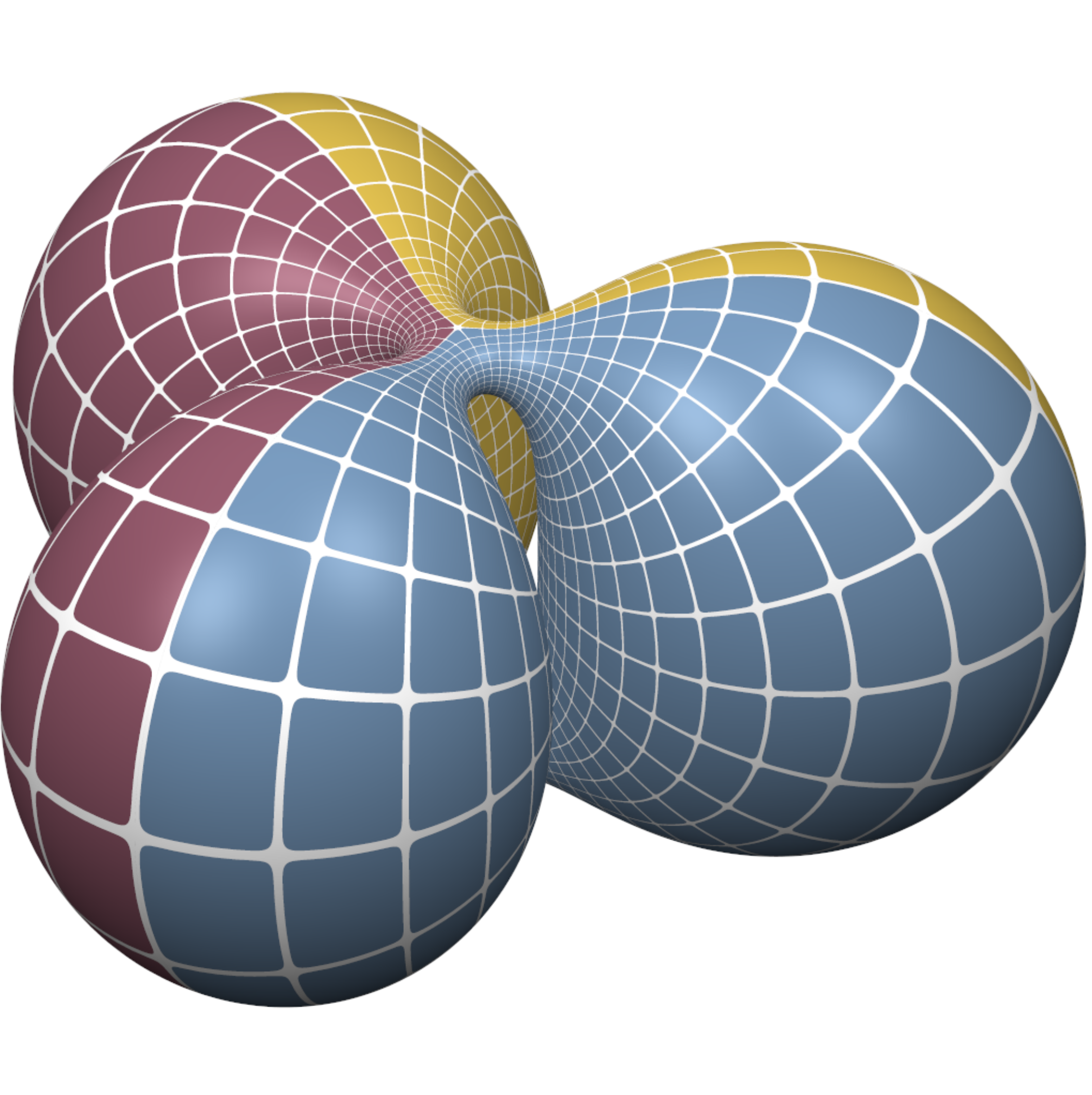}
\caption{
\footnotesize
The {\em unstable} deformation of a 2-lobe Delaunay torus to a CMC surface of genus $2.$}
\label{fig:lawson3}
\end{figure}
\subsection{Experiments and Conjectures}
We have numerically  implemented the flow of spectral data given by Theorem \ref{flow_homog} and Theorem \ref{flow_Delaunay}, and  experiments suggest  (for appropriate initial data) the long time existence
 of the flow (see \cite{HeHeSch3} for details).
 Consider the continuous family of CMC tori parametrized by $\tau$ consisting of the homogenous tori for $\tau \leq \sqrt{3}i$ and the $2$-lobed Delaunay tori for $\tau > \sqrt{3}i.$ This family converges to the branched double cover of a geodesic sphere as $\tau \to \infty$. The experiments indicate that this family of tori flows (if we choose the stable direction for $\tau > \sqrt{3}i$) to the Lawson symmetric CMC surfaces deforming Lawson's minimal surface and converging to a branched double cover of a geodesic sphere as the conformal type of the surfaces degenerates; see Figure \ref{fig:lawson2}.
 The flow of a $2$-lobed Delaunay torus, which has for positive $\rho$ an unstable holomorphic structure inside the unit disc, reaches CMC surfaces of genus $2$ for $\tau$ large enough and converges to a $3$-fold cover of a sphere as $\tau \to \infty$; see Figure \ref{fig:lawson3}. 
We have been able to flow to all
CMC surfaces of genus $2$ found  by a naive numerical search in \cite{HeS}. These experiments indicate that the moduli space of (symmetric) higher genus CMC surfaces inherits some structure of the moduli space of CMC tori. \\

We have used the parameter $\rho$ corresponding to the angle between curvature lines at the umbilics as 
the "flow parameter" in Theorem \ref{flow_homog} and Theorem \ref{flow_Delaunay}. 
The deformation of the conformal type of the torus $\tau\in i\R$ can be used as a second 
flow direction.
For  $\rho=0$ this recovers the Whitham flow of CMC tori. Using the 
methods of this paper it is possible to show the existence of a flow of CMC
surfaces by varying  the conformal type (for fixed 
$\rho$ near $0$).
It is clear that the $\rho$-flow and the conformal type flow commute.  It might be useful to
consider the extra flow direction in order to avoid possible singularities of the $\rho$-flow. We aim to study these phenomena and the long time existence of the $\rho$-flow in future work.\\

Further,
the experiments suggest that  the embeddedness of the cylinder, given by the image of $\{[w]\in T^2=\C/(2\Z+2\tau\Z)\mid
 \tau\leq\Im(w)\leq2\tau\}$ of the CMC immersion $f\colon T^2\to S^3$,
is preserved throughout the flow. By applying a combined version
of the $\rho$-flow with the conformal type flow one can preserve
minimality during the flow.
 Because the Clifford torus is the only embedded minimal tori in $S^3$ \cite{Br} (see also \cite{HaKiSch2} for
 a  Whitham theory approach), we conjecture that the Lawson surface $\xi_{2,1}$ is the only embedded minimal surface in $S^3$ of genus 2.



\date{\today}

\end{document}